\documentclass[10pt,twoside]{book}

\usepackage{amsfonts}
\usepackage{amssymb}
\usepackage{amsmath}
\usepackage{amsthm}
\usepackage{ngerman}
\usepackage{pstricks}
\usepackage{makeidx}

\def\negthickspace{\!\!\!}

\newcommand{\nicefrac}[2]
{\leavevmode \kern.1em\raise.5ex\hbox{\the\scriptfont0 #1}
             \kern-.1em/\kern-.15em\lower.25ex
             \hbox{\the\scriptfont0 #2}}

\newtheorem{theorem}{Theorem}
\newtheorem{proposition}{Proposition}
\newtheorem{corollary}{Corollary}

\newtheorem{definition}{Definition}
\newtheorem{lemma}{Lemma}

\theoremstyle{definition}
\newtheorem{remark}{Remark}

\theoremstyle{definition}

\makeindex

\begin{document} 

\pagestyle{empty}

\thispagestyle{empty}
\vspace*{5ex}
\begin{center}
{\Large{\sc Lectures on normal Coulomb frames}}\\[2ex]
{\Large{\sc in the normal bundle}}\\[2ex]
{\Large{\sc of twodimensional immersionen}}\\[2ex]
{\Large{\sc of higher codimension}}\\[10ex]
{\large{\sc Harmonic Mappings and Geometric Analysis}}\\[20ex]
{\large Steffen Fr\"ohlich}\\[70ex]
{\small\bf Abstract}\\[0.4cm]
\begin{minipage}[c][2.5cm][l]{12cm}
{\small We establish existence and regularity results for normal Coulomb frames in the normal bundle of two-dimensional surfaces of disc-type embedded in Euclidean spaces of higher dimensions.}\\[2ex]
{\small MCS 2000: 53A07, 35J47, 43A37}
\end{minipage}
\end{center}
\cleardoublepage
\pagestyle{empty}
\vspace*{25ex}
\begin{center}
{\Large\it To Mihaela}
\end{center}
\cleardoublepage
\pagestyle{myheadings}
\vspace*{15ex}
\noindent
{\Large\bf Introduction}\\[4ex]
A two-dimensional smooth and orientable surface ${\mathcal M}$ in three-dimensional Euclidean space $\mathbb R^3$ possesses a normal vector field uniquely determined up to orientation.\\[1ex]
In these lectures we consider surfaces in higher-dimensional Euclidean spaces, and we approach the following fundamental problems:
\begin{itemize}
\item[1.]
What are the geometric and analytical characteristics of {\it normal frames for surfaces with arbitrary codimension immersed in Euclidean spaces ${\mathbb R}^{n+2}?$}
\vspace*{-1.6ex}
\item[2.]
Are there {\it special normal frames for surfaces} useful for particular analytical and geometric purposes, and how can we construct such special systems?
\end{itemize}
Our three lectures presented on the following pages are organized as follows:
\begin{itemize}
\item
{\it Lecture I}\\[1ex]
We give a detailed treatment of the foundations of the theory of surfaces immersed in Euclidean spaces. This includes the definition of normal frames as well as orthogonal transitions betweeen such frames, furthermore the differential equations of Gau\ss{} and Weingarten as well as corresponding integrability conditions along with curvature quantities of the surfaces' tangential and normal bundles.
\item
{\it Lecture II}\\[1ex]
We consider surfaces immersed in Euclidean space $\mathbb R^4$ and construct so-called normal Coulomb frames critical for a functional of total torsion. This construction depends on the solvability of a Neumann boundary value problem. Using methods from potential theory and complex analysis we establish various analytical tools to control the so-called torsion coefficients of Coulomb frames, i.e. the connection coefficients of the normal bundle.
\item
{\it Lecture III}\\[1ex]
We consider surfaces immersed in Euclidean spaces $\mathbb R^{n+2}$ of arbitrary dimension. The construction of normal Coulomb frames turns out to be more intricate and requires a profound analysis of nonlinear elliptic systems in two variables. We benefit from the work of many authors from the theory of harmonic analysis: E. Heinz, F. Helein, F. M\"uller, T. Riviere, F. Sauvigny, A. Schikorra, E. Stein, F. Tomi, H. Wente, and many others.
\end{itemize}
Frames parallel in the normal bundle are special Coulomb frames, namely those with vanishing total curvature. We consider {\it non-parallel normal frames} extensively.\\[1ex]
Parallel normal frames for one-dimensional curves in space are already widely used in the research literature, see e.g. da Costa \cite{dacosta_1982} for a geometric presentation of physical problems in quantum mechanics, or Burchard and Thomas \cite{burchard_thomas_2002} for a sophisticated analytical description of the dynamics of Euler's elastic curves.\\[1ex]
Nevertheless, elementary treatments of higher dimensional and higher codimensional differential geometry is underrepresented in ordinary textbooks. An elaboration of problems and methods, considered from the perspectives of geometric analysis, is rather scarce.\\[1ex]
That is finally the issue the paper at hand focuses on. It continues the elementary differential geometry from B\"ar \cite{baer_2009}, Blaschke and Leichtwei\ss{} \cite{blaschke_leichtweiss_1973}, Brauner \cite{brauner_1981}, or Eschenburg and Jost \cite{eschenburg_jost_2007} for surfaces in $\mathbb R^3$ and furnishes this theory with classical and recent results from the harmonic analysis and the theory of nonlinear elliptic systems of partial differential equations.
\goodbreak\noindent
Most of the results presented here are obtained in a fruitful collaboration with Frank M\"uller from the University of Duisburg-Essen. The reader can find our original approaches in the following arXiv resources:
\begin{itemize}
\item[1.]
{\it On critical normal sections for two-dimensional immersions in $R^n$ and a Riemann-Hilbert problem.}\\
arXiv:0705.334.7, May 2007.
\vspace*{-1ex}
\item[2.]
{\it On critical normal sections for two-dimensional immersions in $\mathbb R^{n+2}.$}\\
arXiv:0709.0867, September 2007.
\vspace*{-1ex}
\item[3.]
{\it On the existence of normal Coulomb frames for two-dimensional immersions with higher codimension.}\\
arXiv:0906.1865, June 2009.
\end{itemize}
\vspace*{6ex}
Berlin, Oktober 2009\hfill{Steffen Fr\"ohlich}
\cleardoublepage
\noindent
{\Large\bf Contents Lecture I}\\[4ex]
{\Large\bf Basics}
\vspace*{3ex}
\begin{itemize}
\item[{\bf 1.}]
{\bf Regular surfaces\ \dotfill\ \pageref{section_regularsurfaces}}
\begin{itemize}
\item[1.1]
First definitions\ \dotfill\ \pageref{par_firstdefinitions}
\item[1.2]
Tangent space and normal space\ \dotfill\ \pageref{par_tangentspace}
\item[1.3]
Normal frames\ \dotfill\ \pageref{par_normalframes}
\end{itemize}
\item[{\bf 2.}]
{\bf Examples\ \dotfill\ \pageref{section_examples}}
\begin{itemize}
\item[2.1]
Surface graphs\ \dotfill\ \pageref{par_surfacegraphs}
\item[2.2]
Holomorphic surface graphs\ \dotfill\ \pageref{par_holomorpheGraphen}
\item[2.3]
The Veronese surface\ \dotfill\ \pageref{par_veronese}
\end{itemize}
\item[{\bf 3.}]
{\bf Fundamental forms\ \dotfill\ \pageref{section_fundamentalforms}}
\begin{itemize}
\item[3.1]
The first fundamental form\ \dotfill\ \pageref{par_firstfundamentalform}
\item[3.2]
The tensor of the second fundamental forms\ \dotfill\ \pageref{par_secondfundamentalform}
\item[3.3]
Conformal parameters\ \dotfill\ \pageref{par_konformeParameter}
\item[3.4]
Application to holomorphic surfaces\ \dotfill\ \pageref{par_application}
\item[3.5]
Outlook. Open problem\ \dotfill\ \pageref{par_outlook}
\end{itemize}
\item[{\bf 4.}]
{\bf Differential equations\ \dotfill\ \pageref{section_differentialequations}}
\begin{itemize}
\item[4.1]
The Christoffel symbols\ \dotfill\ \pageref{par_christoffel}
\item[4.2]
The Gau\ss{} equations\ \dotfill\ \pageref{par_gaussequations}
\item[4.3]
The torsion coefficients\ \dotfill\ \pageref{par_torsionen}
\item[4.4]
The Weingarten equations\ \dotfill\ \pageref{par_weingarten}
\end{itemize}
\item[{\bf 5.}]
{\bf Integrability conditions\ \dotfill\ \pageref{section_integrabilityconditions}}
\begin{itemize}
\item[5.1]
Problem statement\ \dotfill\ \pageref{par_problem_integrierbarkeit}
\item[5.2]
The integrability conditions of Codazzi and Mainardi\ \dotfill\ \pageref{par_codazzi_mainardi}
\item[5.3]
The integrability conditions of Gau\ss{}\ \dotfill\ \pageref{par_gaussintegrability}
\item[5.4]
The curvature tensor of the tangent bundle\ \dotfill\ \pageref{par_curvaturetangent}
\item[5.5]
The Gau\ss{} curvature. theorema egregium\ \dotfill\ \pageref{par_theoremaegregium}
\item[5.6]
The integrability conditions of Ricci\ \dotfill\ \pageref{par_ricci}
\end{itemize}
\item[{\bf 6.}]
{\bf The curvature of the normal bundle\ \dotfill\ \pageref{section_curvaturenormalbundle}}
\begin{itemize}
\item[6.1]
Problem statement\ \dotfill\ \pageref{par_problemnormal}
\item[6.2]
The curvature tensor of the normal bundle\ \dotfill\ \pageref{par_curvnormbundle}
\item[6.3]
The case $n=2$\ \dotfill\ \pageref{par_lecture2n2}
\item[6.4]
The normal sectional curvature\ \dotfill\ \pageref{par_normalenschnittkruemmung}
\item[6.5]
Preparations for the normal curvature vector I: Curvature matrices\ \dotfill\ \pageref{par_curvaturematrices}
\item[6.6]
Preparations for the normal curvature vector II: The exterior product\ \dotfill\ \pageref{par_preparations2}
\item[6.7]
The curvature vector of the normal bundle\ \dotfill\ \pageref{par_normalcurvaturevector}
\end{itemize}
\goodbreak
\item[{\bf 7.}]
{\bf Elliptic systems\ \dotfill\ \pageref{section_ellipticsystems}}
\begin{itemize}
\item[7.1]
The mean curvature vector\ \dotfill\ \pageref{par_meancurvaturevector}
\item[7.2]
The mean curvature system\ \dotfill\ \pageref{par_meancurvaturesystem}
\item[7.3]
Quadratic growth in the gradient. A maximum principle\ \dotfill\ \pageref{par_maximumprinciple}
\item[7.4]
A curvature estimate\ \dotfill\ \pageref{par_curvatureestimate}
\end{itemize}
\end{itemize}
\vspace*{10ex}
{\Large\bf Contents Lecture II}\\[4ex]
{\Large\bf Normal Coulomb frames in $\mathbb R^4$}
\vspace*{3ex}
\noindent
\begin{itemize}
\item[{\bf 8.}]
{\bf Problem statement. Curves in $\mathbb R^3$\ \dotfill\ \pageref{section_curves}}
\item[{\bf 9.}]
{\bf Torsion free normal frames\ \dotfill\ \pageref{section_torsionfree}}
\item[{\bf 10.}]
{\bf Examples\ \dotfill\ \pageref{section_examplesn4}}
\begin{itemize}
\item[10.1]
Spherical surfaces\ \dotfill\ \pageref{par_spherical}
\item[10.2]
The flat Clifford torus\ \dotfill\ \pageref{par_clifford}
\item[10.3]
Parallel type surface\ \dotfill\ \pageref{par_parallel}
\end{itemize}
\item[{\bf 11}]
{\bf Normal Coulomb frames\ \dotfill\ \pageref{section_normalcoulombframes}}
\begin{itemize}
\item[11.1]
The total torsion\ \dotfill\ \pageref{par_totaltorsion}
\item[11.2]
Definition of normal Coulomb frames\ \dotfill\ \pageref{par_definitioncoulomb}
\item[11.3]
The Euler-Lagrange equation\ \dotfill\ \pageref{par_eulerlagrangen4}
\item[11.4]
Construction of normal Coulomb frames via a Neumann problem\ \dotfill\ \pageref{par_neumannproblemn4}
\item[11.5]
Minimality property of normal Coulomb frames\ \dotfill\ \pageref{par_minimalityn4}
\end{itemize}
\item[{\bf 12.}]
{\bf Estimating the torsion coefficients\ \dotfill\ \pageref{section_estimatingtorsion}}
\begin{itemize}
\item[12.1]
Reduction for flat normal bundles\ \dotfill\ \pageref{par_reduction}
\item[12.2]
Estimates via the maximum principle\ \dotfill\ \pageref{par_maxprinciple}
\item[12.3]
Estimates via a Cauchy-Riemann boundary value problem\ \dotfill\ \pageref{par_cauchyriemann}
\end{itemize}
\item[{\bf 13.}]
{\bf Estimates for the total torsion\ \dotfill\ \pageref{section_estimatestotaltorsion}}
\item[{\bf 14.}]
{\bf An example: Holomorphic graphs\ \dotfill\ \pageref{section_holomorphicgraphs}}
\end{itemize}
\vspace*{10ex}
{\Large\bf Contents Lecture III}\\[4ex]
{\Large\bf Normal Coulomb frames in $\mathbb R^{n+2}$}
\vspace*{3ex}
\noindent
\begin{itemize}
\item[{\bf 15.}]
{\bf Problem statemtent\ \dotfill\ \pageref{section_problembeliebig}}
\item[{\bf 16.}]
{\bf The Euler-Lagrange equations\ \dotfill\ \pageref{section_eulerlagrangeallgemein}}
\begin{itemize}
\item[16.1]
Definition of normal Coulomb frames\ \dotfill\ \pageref{par_definitionallgemein}
\item[16.2]
The first variation\ \dotfill\ \pageref{par_firstvariation}
\item[16.3]
The integral functions\ \dotfill\ \pageref{par_integralfunctions}
\item[16.4]
A nonlinear elliptic system\ \dotfill\ \pageref{par_nonlinearellipticsystem}
\end{itemize}
\goodbreak
\item[{\bf 17.}]
{\bf Examples\ \dotfill\ \pageref{section_examplesallgemein}}
\begin{itemize}
\item[17.1]
The case $n=2$\ \dotfill\ \pageref{par_n2allgemein}
\item[17.3]
The case $n=3$\ \dotfill\ \pageref{par_n3allgemein}
\end{itemize}
\item[{\bf 18.}]
{\bf Quadratic growth in the gradient\ \dotfill\ \pageref{section_quadraticgrowth}}
\begin{itemize}
\item[18.1]
A Grassmann-type vector\ \dotfill\ \pageref{par_grassmanngrowth}
\item[18.2]
Nonlinear systems with quadratic growth\ \dotfill\ \pageref{par_quadraticgrowth}
\end{itemize}
\item[{\bf 19.}]
{\bf Torsion free normal frames\ \dotfill\ \pageref{section_torsionfree}}
\begin{itemize}
\item[19.1]
The case $n=3$\ \dotfill\ \pageref{par_torsionfreen3}
\item[19.2]
An auxiliary function\ \dotfill\ \pageref{par_auxiliaryfunction}
\item[19.3]
The case $n>3$\ \dotfill\ \pageref{par_zerotorsion}
\end{itemize}
\item[{\bf 20.}]
{\bf Non-flat normal bundles\ \dotfill\ \pageref{section_nonflat}}
\begin{itemize}
\item[20.1]
An upper bound via Wente's $L^\infty$-estimate\ \dotfill\ \pageref{par_wentebound}
\item[20.2]
An upper bound via Poincare's inequality\ \dotfill\ \pageref{par_poincarebound}
\item[20.3]
An estimate for the torsion coefficients\ \dotfill\ \pageref{par_torsionestimate}
\end{itemize}
\item[{\bf 21.}]
{\bf Bounds for the total torsion\ \dotfill\ \pageref{section_boundstotaltorsion}}
\begin{itemize}
\item[21.1]
Upper bounds\ \dotfill\ \pageref{par_upperbounds}
\item[21.2]
A lower bound\ \dotfill\ \pageref{par_lowerbounds}
\end{itemize}
\item[{\bf 22.}]
{\bf Existence and regularity of weak normal Coulomb frames\ \dotfill\ \pageref{section_existence}}
\begin{itemize}
\item[22.1]
Regularity results for the homogeneous Poisson problem\ \dotfill\ \pageref{par_poisson}
\item[22.2]
Existence of weak normal Coulomb frames\ \dotfill\ \pageref{par_weakexistence}
\item[22.3]
$H_{loc}^{2,1}$-regularity of weak normal Coulomb frames\ \dotfill\ \pageref{par_Hframes}
\end{itemize}
\item[{\bf 23.}]
{\bf Classical regularity of normal Coulomb frames\ \dotfill\ \pageref{section_classical}}
\item[{\bf 24.}]
{\bf Estimates for the area element $W$\ \dotfill\ \pageref{section_Westimates}}
\begin{itemize}
\item[24.1]
Inner estimates\ \dotfill\ \pageref{par_inner}
\item[24.2]
Global estimates\ \dotfill\ \pageref{par_global}
\item[24.3]
Straightening the boundary\ \dotfill\ \pageref{par_straight}
\end{itemize}
\end{itemize}
\cleardoublepage
\noindent
  $$ $$\\[15ex]
\pagestyle{empty}
{\bf{\sc{\Large Lecture I}}}\\[4ex]
{\bf{\sc{\huge Basics}}}\\[6ex]
\rule{108ex}{0.2ex}
\vspace*{20ex}
\begin{itemize}
\item[1.]
Regular surfaces
\item[2.]
Examples
\item[3.]
Fundamental forms
\item[4.]
Differential equations
\item[5.]
Integrability conditions
\item[6.]
The curvature of the normal bundle
\item[7.]
Elliptic systems
\end{itemize}
\cleardoublepage\noindent
\pagestyle{myheadings}
\markboth{Normal Coulomb frames}{Normal Coulomb frames}
\thispagestyle{empty}
\vspace*{10ex}
\section{Regular surfaces}
\label{section_regularsurfaces}
\subsection{First definitions}
\label{par_firstdefinitions}
Let $n\ge 1$ be an integer. The main objects of our considerations are vector-valued mappings
  $$X=X(u,v)=\big(x^1(u,v),\ldots x^{n+2}(u,v)\big),\quad(u,v)\in\overline B,$$
defined on the topological closure $\overline B\subset\mathbb R^2$ of the open unit disc
  $$B:=\big\{(u,v)\in\mathbb R^2\,:\,u^2+v^2<1\big\}.$$
From the point of view of differential geometry we always want to assume
\begin{itemize}
\item[$\longrightarrow$]
$\mbox{rank}\,DX(u,v)=2$ for all $(u,v)\in\overline B$
\end{itemize}
for the Jacobian $DX\in\mathbb R^{2\times(n+2)}$ of $X,$ i.e. at each point $w\in\overline B$ there is a non-degenerate, two-dimensional tangent plane, see the next paragraph for more details; but also
\begin{itemize}
\item[$\longrightarrow$]
$X\in C^{k,\alpha}(\overline B,\mathbb R^{n+2})$ with an integer $k\ge 3$ and a H\"older exponent $\alpha\in(0,1).$ This latter regularity assumption is particularly needed for our analytical considerations in the third lecture.
\end{itemize}
{\it Thus the mapping $X$ represents a regular surface or two-dimen\-sional immersion of disc-type.}\index{immersion}
\subsection{Tangent space and normal space}
\label{par_tangentspace}
For $X$ represents an immersion, at each point $w\in\overline B$ there exist two linearly independent {\it tangent vectors}\index{tangent vector}
  $$X_u=\frac{\partial X}{\partial u}
    \quad\mbox{und}\quad
    X_v=\frac{\partial X}{\partial v}\,,$$
the derivatives of $X,$\footnote{Symbols like $X_u,$ $N_{\sigma,u},$ $g_{ij,u}$ etc. denote partial derivatives w.r.t. $u.$} spanning the two-dimensional {\it tangent space at}\index{tangent space} $w\in\overline B:$
  $${\mathbb T}_X(w):=\mbox{span}\,\big\{X_u(w),X_v(w)\big\}\cong\mathbb R^2\,.$$
Its orthogonal complement forms the {\it normal space at}\index{normal space} $w\in\overline B,$ i.e.
  $${\mathbb N}_X(w):=\big\{Z\in\mathbb R^{n+2}\,:\,X_u\cdot Z=X_v\cdot Z=0\big\}\cong\mathbb R^n$$
where $X\cdot Y$ denotes the Euclidean inner product between two vectors $X,Y\in\mathbb R^{n+2},$ that is
  $$X\cdot Y:=\sum_{i=1}^{n+2}x^iy^i\,.$$
\subsection{Normal frames}
\label{par_normalframes}
At each point $w\in\overline B$ we choose $n\ge 1$ unit normal vectors $N_\sigma=N_\sigma(w),$ $\sigma=1,\ldots,n,$ satisfying the orthonormality relations\index{orthonormality relations}
  $$N_\sigma\cdot N_\vartheta
    =\delta_{\sigma\vartheta}
    =\left\{
       \begin{array}{ll}
         1 & \mbox{for}\ \sigma=\vartheta \\
         0 & \mbox{for}\ \sigma\not=\vartheta
       \end{array}
     \right.
    \quad\mbox{for all}\ \sigma,\vartheta=1,\ldots,n$$
with Kronecker's symbol\index{Kronecker symbol} $\delta_{\sigma\vartheta}.$ We also use notations like $\delta_{\sigma\vartheta},$ $\delta_\sigma^\vartheta,$ or $\delta^{\sigma\vartheta}$ for the Kronecker symbol. Now choose $N_\sigma(w)$ in such a way that
\begin{itemize}
\item[(i)]
they span the normal space $\mathbb N_X(w)$ at $w\in\overline B,$
\vspace*{-1.4ex}
\item[(ii)]
they are oriented in the following sense
 $$\det\big(X_u,X_v,N_1,\ldots,N_n\big)>0.$$
\end{itemize}
Thanks to the contractibility of the domain $\overline B$ we may set
\begin{definition}
The matrix
  $$N=(N_1,\ldots,N_n)\in C^{k-1,\alpha}(\overline B,\mathbb R^{n\times(n+2)}),$$
which consists of $n\ge 1$ orthonormal unit normal vectors $N_\sigma=N_\sigma(w),$ oriented in the above sense and spanning the $n$-dimensional normal space ${\mathbb N}(w)$ at each $w\in\overline B,$ i.e. moving $C^{k-1,\alpha}$-smooth along the whole surface $X,$ is called a normal frame.\index{normal frame}
\end{definition}
\vspace*{6ex}
\begin{center}
\rule{35ex}{0.1ex}
\end{center}
\vspace*{6ex}
\section{Examples}
\label{section_examples}
We briefly consider some explicit given surfaces which simultaneously represent important examples for our following analysis.
\subsection{Surface graphs}
\label{par_surfacegraphs}
\begin{definition}
A surface graph\index{surface graph} is a mapping
  $$\mathbb R^2\ni(x,y)\mapsto\big(x,y,z_1(x,y),\ldots,z_n(x,y)\big)\in\mathbb R^{n+2}$$
with sufficiently smooth functions $z_\sigma,$ $\sigma=1,\ldots,n,$ generating the graph.
\end{definition}
\noindent
Graphs are always immersions. We can even specify a possible normal frame\index{normal frame} in the form
  $$\begin{array}{lll}
      N_1\negthickspace
      & = & \displaystyle\negthickspace
            \frac{1}{\sqrt{1+|\nabla z_1|^2}}\,\big(z_{1,x},z_{1,y},1,0,0,\ldots,0\big), \\[4ex]
      N_2\negthickspace
      & = & \displaystyle\negthickspace
            \frac{1}{\sqrt{1+|\nabla z_2|^2}}\,\big(z_{2,x},z_{2,y},0,1,0,\ldots,0\big)
            \quad\mbox{etc.}
    \end{array}$$
with $z_{\sigma,x}$ and $z_{\sigma,y}$ denoting the partial derivatives of $z_\sigma,$ and $\nabla z_\sigma=(z_{\sigma,x},z_{\sigma,y})\in\mathbb R^2$ its Euclidean gradient.
\begin{definition}
These special unit normal vectors $N_\sigma$ are called the Euler unit normals of the graph $X.$\index{Euler unit normals}
\end{definition}
\noindent
In general, Euler unit normals are not orthogonal, but by means of Gram-Schmidt orthogonalization\index{Gram Schmidt orthogonalization} we can always construct an orthonormal basis from the Euler normal frame.
\subsection{Holomorphic surface graphs}
\label{par_holomorpheGraphen}
Now let us consider surface graphs of the special form\index{holomorphic graph}
  $$\mathbb R^2\ni(x,y)\mapsto\big(x,y,\varphi(x,y),\psi(x,y)\big)\in\mathbb R^4$$
with $\varphi$ and $\psi$ being real resp. imaginary part of a complex-valued holomorphic function
  $$\Phi(x,y)=\varphi(x,y)+i\psi(x,y)\in\mathbb C\,.$$
For example, say $\varphi+i\psi=(x+iy)^2=x^2-y^2+2ixy,$ so that $\varphi$ and $\psi$ solve the Cauchy-Riemann equations\index{Cauchy-Riemann equations}
  $$\varphi_x=\psi_y\,,\quad
    \varphi_y=-\psi_x\,.$$
Then the associated Euler unit normals
  $$\begin{array}{lll}
      N_1\negthickspace
      & = & \negthickspace\displaystyle
            \frac{1}{\sqrt{1+|\nabla\varphi|^2}}\,\big(\varphi_x,\varphi_y,1,0\big), \\[4ex]
      N_2\negthickspace
      & = & \negthickspace\displaystyle
            \frac{1}{\sqrt{1+|\nabla\psi|^2}}\,\big(\psi_x,\psi_y,0,1\big)
            \,=\,\frac{1}{\sqrt{1+|\nabla\varphi|^2}}\,\big(-\varphi_y,\varphi_x,0,1\big)
    \end{array}$$
form an orthonormal frame (in the normal space) because we immediately verify
  $$N_1\cdot N_2
    =\frac{1}{1+|\nabla\varphi|^2}\,\big(-\varphi_x\varphi_y+\varphi_y\varphi_x\big)
    =0.$$
\subsection{The Veronese surface}
\label{par_veronese}
The following surface\index{Veronese surface}
  $$\lambda
    \left(
      \frac{yz}{\sqrt{3}}\,,
      \frac{xz}{\sqrt{3}}\,,
      \frac{xy}{\sqrt{3}}\,,
      \frac{x^2-y^2}{2\sqrt{3}}\,,
      \frac{1}{6}\,\big(x^2+y^2-2z^2\big)
    \right)
    \quad\mbox{with}\quad
    x^2+y^2+z^2=3,\quad
    \lambda\in\mathbb R,$$
first described by Giuseppe Veronese (1854-1917), is a special example for our analysis.
\begin{proposition}
(Chen and Ludden \cite{chen_ludden_1972}, 1972)\\
The Veronese surface is the only compact surface (without boundary) in $\mathbb R^5$ with parallel mean curvature vector and constant normal curvature. Furthermore, it has constant Gau\ss{} curvature.
\end{proposition}
\noindent
The contents of this proposition will become clear after the study of this first lecture.
\vspace*{6ex}
\begin{center}
\rule{35ex}{0.1ex}
\end{center}
\vspace*{6ex}
\section{Fundamental forms}
\label{section_fundamentalforms}
\subsection{The first fundamental form}
\label{par_firstfundamentalform}
Let $u^1=u$ and $u^2=v.$
\begin{definition}
The first fundamental form\index{first fundamental form} $g=(g_{ij})_{i,j=1,2}\in\mathbb R^{2\times 2}$ of $X$ is given by
  $$g_{ij}:=X_{u^i}\cdot X_{u^j}\,,\quad i,j=1,2.$$
\end{definition}
\noindent
The differential line element\index{line element} of the surface reads as follows
  $$ds^2=\sum_{i,j=1}^2g_{ij}\,du^idu^j\,.$$
\begin{remark}
Formally it results from inserting the representation of $X$ into the Euclidean form
  $$ds^2=\sum_{k,\ell=1}^{n+2}\delta_{k\ell}\,dx^kdx^\ell$$
of the embedding space $\mathbb R^{n+2}$ with coordinates $x^k,$ $k=1,\ldots,n+2,$ by means of the following calculation
  $$ds^2
    =\sum_{k,\ell=1}^n
     \delta_{k\ell}(x_u^k\,du+x_v^k\,dv)(x_u^\ell\,du+x_v^\ell\,dv)
    =\sum_{k=1}^n
     \Big\{
       (x_u^k)^2\,du^2+2(x_u^kx_v^k)\,dudv+(x_v^k)^2\,dv^2
     \Big\}\,.$$
\end{remark}
\noindent
We note that the first fundamental form is invertible on account of the regularity property $\mbox{rank}\,DX=2$ for the Jacobian $DX$ of the mapping $X.$ For its inverse we write
  $$g^{-1}=(g^{ij})_{i,j=1,2}\in\mathbb R^{2\times 2}\,.$$
At each point $w\in\overline B$ it holds
  $$\sum_{j=1}^2g_{ij}g^{jk}=\delta_i^k$$
with Kronecker's symbol $\delta_i^k.$
\subsection{The tensor of the second fundamental forms}
\label{par_secondfundamentalform}
\begin{definition}
To each unit normal vector $N_\sigma$ of a given normal frame $N=(N_1,\ldots,N_n)$ we assign a second fundamental form with coefficients\index{second fundamental form}
  $$L_{\sigma,ij}
    :=X_{u^iu^j}\cdot N_\sigma\,,
    \quad i,j=1,2,\ \sigma=1,\ldots,n.$$
\end{definition}
\noindent
Notice that
  $$X_{u^iu^j}\cdot N_\sigma
    =-X_{u^i}\cdot N_{\sigma,u^j}$$
which follows directly after differentiation of the orthogonality relations $X_{u^i}\cdot N_\sigma=0$ for all $i=1,2,$ and $\sigma=1,\ldots,n.$ In case $n=1$ of one codimension there is only one second fundamental form.
\subsection{Conformal parameters}
\label{par_konformeParameter}
Mainly we will work with a {\it conformal parameter system}\index{conformal parameters} $(u,v)\in\overline B$ on whose account the first fundamental form takes the diagonal form
  $$g_{11}=W=g_{22}\,,\quad
    g_{12}=0
    \quad\mbox{in}\ \overline B.$$
Then the area element
  $$W:=\sqrt{g_{11}g_{22}-g_{12}}$$
represents the {\it conformal factor}\index{conformal factor} w.r.t. this special parametrization. Introducing conformal parameters is justified by the following results.
\begin{proposition}
(Sauvigny \cite{sauvigny_2005}, 2005)\\
Assume that the coefficients $a,$ $b$ and $c$ of the Riemannian metric\index{Riemannian metric}
  $$ds^2=a\,du^2+2b\,dudv+c\,dv^2$$
are of class $C^{1+\alpha}(\overline B,\mathbb R)$ with $\alpha\in(0,1).$ Then there is a conformal parameter system $(u,v)\in\overline B.$
\end{proposition}
\noindent
Recall that $ds^2$ is called of {\it Riemannian type} if $ac-b^2>0$ in $\overline B.$\\[1ex]
The regularity condition required here is satisfied in our situation because $g_{ij}\in C^{k-1}(\overline B)$ with $k\ge 3.$ While Sauvigny's result holds {\it in the large,} i.e. on the whole closed disc $\overline B,$ another optimal resultat {\it in the small} is the following.
\begin{proposition}
(Chern \cite{chern_1955}, 1955)\\
Assume that the coefficients of the Riemannian metric
  $$ds^2=a\,du^2+2b\,dudv+c\,dv^2$$
are H\"older continuous in $\overline B.$ Then for every point $w\in B$ there exists an open neighborhood over which the surface can be pararmetrized conformally.
\end{proposition}
\subsection{Application to holomorphic surfaces}
\label{par_application}
Let $w=u+iv\in\overline B.$ As in paragraph \ref{par_holomorpheGraphen} we consider mappings\index{holomorphic surface}
  $$X(w)=\big(\Phi(w),\Psi(w)\big)\colon\overline B\longrightarrow\mathbb C\times\mathbb C$$
with complex-valued holomorphic functions $\Phi=(\varphi_1,\varphi_2)$ and $\Psi=(\psi_1,\psi_2).$ Real- and imaginary part are solutions of the Cauchy-Riemann equations\index{Cauchy-Riemann equations}
  $$\varphi_{1,u}=\varphi_{2,v}\,,\quad
    \varphi_{1,v}=-\varphi_{2,u}\,,\quad
    \psi_{1,u}=\psi_{2,v}\,,\quad
    \psi_{1,v}=-\psi_{2,u}\,.$$
Now we calculate the coefficients of the first fundamental form
  $$\begin{array}{lll}
      g_{11}\negthickspace
      & = & \negthickspace\displaystyle
            X_u\cdot X_u
            \,=\,(\varphi_{1,u},\varphi_{2,u},\psi_{1,u},\psi_{2,u})^2
            \,=\,\varphi_{1,u}^2+\varphi_{2,u}^2+\psi_{1,u}^2+\psi_{2,u}^2\,, \\[2ex]
      g_{22}\negthickspace
      & = & \negthickspace\displaystyle
            X_v\cdot X_v
            \,=\,\varphi_{1,v}^2+\varphi_{2,v}^2+\psi_{1,v}^2+\psi_{2,v}^2
            \,=\,\varphi_{2,u}^2+\varphi_{1,u}^2+\psi_{2,u}^2+\psi_{1,u}^2
            \,=\,g_{11}
    \end{array}$$
as well as
  $$g_{12}
    =X_u\cdot X_v
    =\varphi_{1,u}\varphi_{1,v}+\varphi_{2,u}\varphi_{2,v}+\psi_{1,u}\psi_{1,v}+\psi_{2,u}\psi_{2,v}
    =0.$$
Thus we have proved the
\begin{proposition}
The map $(\Phi(w),\Psi(w)$ with holomorphic functions $\Phi$ and $\Psi$ is conformally parametrized.\index{conformal parameters}
\end{proposition}
\subsection{Outlook. Open problems}
\label{par_outlook}
We want to itemize some important problems {\it we do not address here} but for which at least partial approaches or even solutions exist in the literature.
\subsubsection{Riemannian embedding space}
From the analytical and from the geometric point of view it is of interest to consider immersions in general Riemannian spaces.\index{Riemannian space} For example, let $X\colon\overline B\to{\mathcal N}^{n+2}$ with a $(n+2)$-dimensional manifold ${\mathcal N}$ equipped with a Riemannian metric $\eta_{k\ell}$ for $k,\ell=1,\ldots,n+2$ satisfying
  $$\sum_{i,j=1}^{n+2}
    \eta_{ij}\xi^i\xi^j>0
    \quad\mbox{for all}\ \xi=(\xi^1,\ldots,\xi^{n+2})\in\mathbb R^{n+2}\,.$$
The corresponding line element\index{line element} is given by
  $$ds^2=\sum_{i,j=1}^{n+2}\eta_{ij}\,dx^idx^j$$
from where we infer the induced line element of the surface
  $$ds^2
    =\sum_{i,j=1}^2\sum_{k,\ell=1}^{n+2}
     \eta_{k\ell}x_{u^i}^kx_{u^j}^\ell\,du^idu^j
    =\sum_{i,j=1}^2
     \gamma_{ij}\,du^idu^j
    \quad\mbox{with}\quad
    \gamma_{ij}=\sum_{k,\ell=1}^{n+2}\eta_{k\ell}x_{u^i}^kx_{u^j}^\ell\,.$$
This form is of Riemannian type and admits again a conformal reparametrization.
\subsubsection{Lorentz spaces and spaceforms}
Beside the Riemanniann spherical spaceform of constant curvature $+1$ we would also like to work with immersions living in pseudo-Euclidean spaces with non-positive signature, in a hyperbolic spaceform of curvature $-1,$ or even curved hyperbolic spaces, for all of them are particularly relevant for many applications in applications.\index{Lorentz space} The simplest example of such an embedding manifold is the fourdimensional Minkowski space $\mathbb R^{3+1}$ with metric\index{Minkowski space}
  $$(\eta_{ij})_{i,j=1,\ldots,4}
   =\left(
      \begin{array}{cccc}
        -1 & 0 & 0 & 0 \\[1ex]
         0 & 1 & 0 & 0 \\[1ex]
         0 & 0 & 1 & 0 \\[1ex]
         0 & 0 & 0 & 1
      \end{array}
    \right).$$
While most of our calculations from Lecture II can also be carried out for surfaces $X\colon\overline B\to\mathbb R^{3+1},$ it seems already difficult to us to succeed if our immersions live in higher dimensional Minkowski- or Lorentz spaces.
\subsubsection{Higher dimensional manifolds}
Various problems appear if we wish to consider higher dimensional manifolds instead of surfaces. This concerns particularly the theory of nonlinear elliptic systems we apply in Lecture III. Our theory is adjusted to the twodimensional situation.
\subsubsection{General vector bundles}
Orthonormal tangent frames are considered in Helein \cite{helein_2002}, and we develop a theory of orthonormal frames in the normal bundle of surfaces. Thus the question arises whether it is possible to work with arbitrary vector bundles over manifolds. Up to now we can not give a complete answer for several methods adapted to our special situation here are applied in Lecture III.\index{vector bundle}
\subsubsection{Mean curvature flow}
Another item are geometric flows for surfaces in higher dimensional spaces, in particular the mean curvature flow, see Ecker \cite{ecker_2004}. The literature covers numerous contributions on this problem, but mainly for surfaces with either mean curvature vector parallel in the normal bundle or even with flat normal bundles. Both approaches cover immersions close to geometric objects in Euclidean space $\mathbb R^3.$\index{mean curvature flow}
\vspace*{6ex}
\begin{center}
\rule{35ex}{0.1ex}
\end{center}
\vspace*{6ex}
\section{Differential equations}
\label{section_differentialequations}
The system $\{X_u,X_v,N_1,\ldots,N_n\}$ forms a {\it moving $(n+2)$-frame} for the immersion $X.$ In this chapter we want to quantify the rate of change of this frame under infinitesimal variations.
\subsection{The Christoffel symbols}
\label{par_christoffel}
To evaluate the derivatives of $X$ we need
\begin{definition}
The connection coefficients of the tangent bundle\index{tangent bundle}\footnote{The tangent bundle is the collection $\displaystyle\bigcup_{w\in\overline B}(w,{\mathbb T}_X(w)).$} are the Christoffel symbols\index{Christoffel symbols}
  $$\Gamma_{ij}^k
    :=\sum_{\ell=1}^2
      \frac{1}{2}\,g^{k\ell}\,
      \big(
        g_{\ell i,u^j}+g_{j\ell,u^i}-g_{ij,u^\ell}
      \big),\quad
     i,j,k=1,2.$$
\end{definition}
\noindent
Using conformal parameters from paragraph \ref{par_konformeParameter}, the Christoffel symbols take the following form
  $$\begin{array}{l}
      \displaystyle
      \Gamma_{11}^1=\frac{W_u}{2W}\,,\quad
      \Gamma_{12}^1=\Gamma_{21}^1=\frac{W_v}{2W}\,,\quad
      \Gamma_{22}^1=-\frac{W_u}{2W}\,,\\[3ex]
      \displaystyle
      \Gamma_{11}^2=-\frac{W_v}{2W}\,,\quad
      \Gamma_{12}^2=\Gamma_{21}^2=\frac{W_u}{2W}\,,\quad
      \Gamma_{22}^2=\frac{W_v}{2W}
    \end{array}$$
with the area element $W.$ The Christoffel symbols encode the way of parallel transport of surface vector fields. Our main objects of investigation are the connection coefficients of the normal bundle.
\subsection{The Gau\ss{} equations}
\label{par_gaussequations}
Now we come to the\index{Gauss equations}
\begin{proposition}
Let the immersion $X$ together with a normal frame $N$ be given. Then there hold
  $$X_{u^iu^j}
    =\sum_{k=1}^2
     \Gamma_{ij}^kX_{u^k}
     +\sum_{\sigma=1}^n
      L_{\sigma,ij}N_\sigma$$
for $i,j=1,2.$
\end{proposition}
\begin{proof}
We follow Blaschke and Leichtwei\ss{} \cite{blaschke_leichtweiss_1973}, \S 57 and evaluate the following ansatz
  $$X_{u^iu^j}
    =\sum_{k=1}^2
     a_{ij}^kX_{u^k}
     +\sum_{\vartheta=1}^n
       b_{\vartheta,ij}N_\vartheta$$
with functions $a_{ij}^k$ and $b_{\vartheta,ij}$ to be determined. A first multiplication by $N_\omega$ gives
  $$L_{\omega,ij}
    =X_{u^iu^j}\cdot N_\omega
    =\sum_{\vartheta=1}^nb_{\vartheta,ij}N_\vartheta\cdot N_\omega
    =\sum_{\vartheta=1}^nb_{\vartheta,ij}\delta_{\vartheta\omega}
    =b_{\omega,ij}\,.$$
To compute the $a_{ij}^k$ we multiply our ansatz by $X_{u^\ell}$ and arrive at
  $$X_{u^iu^j}\cdot X_{u^\ell}
    =\sum_{k=1}^2a_{ij}^kg_{k\ell}
    =:a_{i\ell j}\,.$$
Note that $a_{i\ell j}=a_{j\ell i}.$ We calculate
  $$a_{i\ell j}=(X_{u^i}\cdot X_{u^\ell})_{u^j}-X_{u^i}\cdot X_{u^\ell u^j}=g_{i\ell,u^j}-a_{\ell ij}$$
which implies $g_{i\ell,u^j}=a_{i\ell j}+a_{\ell ij}.$ We infer
  $$g_{j\ell,u^i}+g_{\ell i,u^j}-g_{ij,u^\ell}
    =a_{j\ell i}+a_{\ell ji}+a_{\ell ij}+a_{i\ell j}-a_{ij\ell}-a_{ji\ell}
    =2a_{i\ell j}\,,$$
and therefore it holds
  $$\sum_{k=1}^2a_{ij}^kg_{k\ell}
    =\frac{1}{2}\,(g_{j\ell,u^i}+g_{\ell i,u^j}-g_{ij,u^\ell}).$$
Rearranging this identity for the unknown functions $a_{ij}^k$ shows
  $$a_{ij}^m
    =\frac{1}{2}\,\sum_{\ell=1}^2g^{m\ell}(g_{j\ell,u^i}+g_{\ell i,u^j}-g_{ij,u^\ell})$$
proving the statement.
\end{proof}
\subsection{The torsion coefficients}
\label{par_torsionen}
To determine the infinitesimal variation of the unit normal vectors $N_\sigma$ of some fixed chosen normal frame $N$ we need the following connection coefficients of the normal bundle.\footnote{The normal bundle is the collection $\displaystyle\bigcup_{w\in\overline B}(w,{\mathbb N}_X(w)),$ see section 7.}
\begin{definition}
The connection coefficients of the normal bundle\index{normal bundle} are the torsion coefficients\index{torsion coefficient}
  $$T_{\sigma,i}^\vartheta:=N_{\sigma,u^i}\cdot N_\vartheta$$
for $i=1,2$ and $\sigma,\vartheta=1,\ldots,n.$
\end{definition}
\noindent
Taking $N_\sigma\cdot N_\vartheta=\delta_{\sigma\vartheta}$ into accound we immediately infer
\begin{proposition}
The torsion coefficients are skew-symmetric w.r.t. interchanging $\sigma\leftrightarrow\vartheta,$ i.e. it holds
  $$T_{\sigma,i}^\vartheta=-T_{\vartheta,i}^\sigma\,.$$
\end{proposition}
\begin{remark}
The torsion coefficients behave like tensors of rank $1$ w.r.t. the $i$-index, and therefore they depend on the parametrization.
\end{remark}
\noindent
To justify the name ``torsion coeffcient'' we consider an arc-length parametrized curve $c(s)$ in $\mathbb R^3$ together with the moving $3$-frame $(t(s),n(s),b(s))$ consisting of the unit tangent vector $t(s),$ the unit normal vector $n(s)$ and the unit binormal vector $b(s).$ Then its curvature\index{curvature of curves} $\kappa(s)$ and torsion\index{torsion of curves} $\tau(s)$ are given by
  $$\kappa(s)=|t'(s)|,\quad
    \tau(s)=n'(s)\cdot b(s),$$
and this already clarifies the analogy to our definition of the torsion coefficients.\\[1ex]
In fact it was Weyl in \cite{weyl_1922} who first used the terminology ``torsion'':  {\it Aus einem normalen Vektor ${\mathbf n}$ in $P$ entsteht ein Vektor ${\mathbf n}'+d{\mathbf t}$ (${\mathbf n}'$ normal, $d{\mathbf t}$ tangential). Die infinitesimale lineare Abbildung ${\mathbf n}\to{\mathbf n}'$ von ${\mathfrak N}_P$ auf ${\mathfrak N}_{P'}$ ist die Torsion.}
\subsection{The Weingarten equations}
\label{par_weingarten}
Now we determine the variation of the unit normal vectors $N_\sigma$ of a given normal frame $N.$
\begin{proposition}
Let the immersion $X$ together with a normal frame $N$ be given. Then there hold\index{Weingarten equations}
  $$N_{\sigma,u^i}
    =-\sum_{j,k=1}^2
      L_{\sigma,ij}g^{jk}X_{u^k}
      +\sum_{\vartheta=1}^n
       T_{\sigma,i}^\vartheta N_\vartheta$$
for $i=1,2$ and $\sigma=1,\ldots,n.$
\end{proposition}
\begin{proof}
For the proof we determine the functions $a_{\sigma,i}^k$ and $b_{\sigma,i}^\vartheta$ of the ansatz
  $$N_{\sigma,u^i}
    =\sum_{k=1}^2
     a_{\sigma,i}^kX_{u^k}
     +\sum_{\vartheta=1}^nb_{\sigma,i}^\vartheta N_\vartheta\,.$$
Multiplication by $X_{u^\ell}$ gives
  $$-L_{\sigma,i\ell}
    =N_{\sigma,u^i}\cdot X_{u^\ell}
    =\sum_{k=1}^2
     a_{\sigma,i}^k X_{u^k}\cdot X_{u^\ell}
    =\sum_{k=1}^2
     a_{\sigma,i}^kg_{k\ell}\,,
    \quad\mbox{therefore}\quad
    a_{\sigma,i}^m=-\sum_{\ell=1}^2L_{\sigma,i\ell}g^{\ell m}\,.$$
A second multiplication by $N_\omega$ shows
  $$T_{\sigma,i}^\omega
    =N_{\sigma,u^i}\cdot N_\omega
    =\sum_{\vartheta=1}^nb_{\sigma,i}^\vartheta N_\vartheta\cdot N_\omega
    =\sum_{\vartheta=1}^nb_{\sigma,i}^\vartheta\delta_{\vartheta\omega}
    =b_{\sigma,i}^\omega$$
proving the statement.
\end{proof}
\vspace*{6ex}
\begin{center}
\rule{35ex}{0.1ex}
\end{center}
\vspace*{6ex}
\section{Integrability conditions}
\label{section_integrabilityconditions}
\subsection{Problem statement}
\label{par_problem_integrierbarkeit}
In view of $X\in C^{3+\alpha}(\overline B,\mathbb R^{n+2})$ there hold various integrability conditions which we present in this chapter. In particular, a further differentiation of the Gau\ss{} equations gives us necessary conditions for the third derivatives of the surface vector $X$ resp. the second derivatives of the tangent vectors $X_{u^i}:$
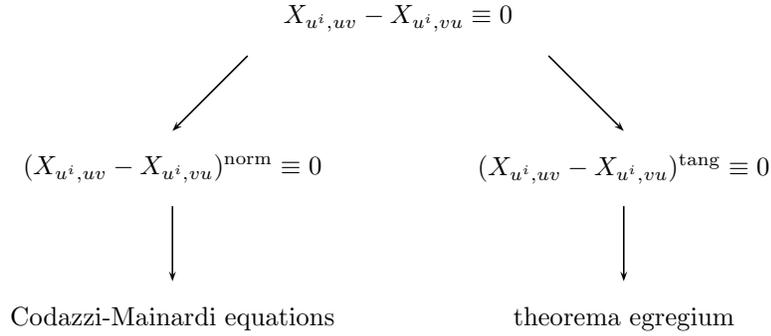
\begin{center}
\begin{pspicture}(0,0)(10,5)
\rput[c](5.0,4.5){$X_{u^i,uv}-X_{u^i,vu}\equiv 0$}
\rput[c](2.0,2.5){$(X_{u^i,uv}-X_{u^i,vu})^{\rm norm}\equiv 0$}
\rput[c](8.0,2.5){$(X_{u^i,uv}-X_{u^i,vu})^{\rm tang}\equiv 0$}
\rput[c](2.0,0.5){Codazzi-Mainardi equations\index{Codazzi-Mainardi equations}}
%\rput[c](2.0,0.0){(Section 3.2)}
\rput[c](8.0,0.5){theorema egregium\index{theorema egregium}}
%\rput[c](8.0,0.0){(Sections 3.3, 3.5)}
\psline[linewidth=0.6pt]{->}(3.0,4.0)(2.0,3.0)
\psline[linewidth=0.6pt]{->}(7.0,4.0)(8.0,3.0)
\psline[linewidth=0.6pt]{->}(2.0,2.0)(2.0,1.0)
\psline[linewidth=0.6pt]{->}(8.0,2.0)(8.0,1.0)
\end{pspicture}
\end{center}
where the upper ``norm'' and ``tang'' mean the normal resp. tangent components. Analogously we proceed with differentiating the Weingarten equations to get
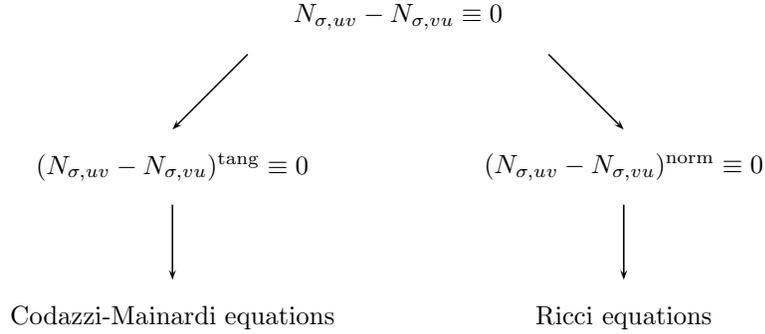
\begin{center}
\begin{pspicture}(0,0)(10,5)
\rput[c](5.0,4.5){$N_{\sigma,uv}-N_{\sigma,vu}\equiv 0$}
\rput[c](2.0,2.5){$(N_{\sigma,uv}-N_{\sigma,vu})^{\rm tang}\equiv 0$}
\rput[c](8.0,2.5){$(N_{\sigma,uv}-N_{\sigma,vu})^{\rm norm}\equiv 0$}
\rput[c](2.0,0.5){Codazzi-Mainardi equations\index{Codazzi-Mainardi equations}}
%\rput[c](2.0,0.0){(Section 3.2)}
\rput[c](8.0,0.5){Ricci equations\index{Ricci equations}}
%\rput[c](8.0,0.0){(Sections 3.6, 3.7)}
\psline[linewidth=0.6pt]{->}(3.0,4.0)(2.0,3.0)
\psline[linewidth=0.6pt]{->}(7.0,4.0)(8.0,3.0)
\psline[linewidth=0.6pt]{->}(2.0,2.0)(2.0,1.0)
\psline[linewidth=0.6pt]{->}(8.0,2.0)(8.0,1.0)
\end{pspicture}
\end{center}
To be more precise, let us start with the Gau\ss{} equations from paragraph \ref{par_gaussequations}, i.e.
  $$X_{u^iu}
    =\Gamma_{i1}^1X_u+\Gamma_{i1}^2X_v
     +\sum_{\sigma=1}^nL_{\sigma,i1}N_\sigma\,,\quad
    X_{u^iv}
    =\Gamma_{i2}^1X_u+\Gamma_{i2}^2X_v
     +\sum_{\sigma=1}^nL_{\sigma,i2}N_\sigma$$
for $i=1,2.$ We differentiate the first equation w.r.t. $v,$
  $$\begin{array}{lll}
      X_{u^iuv}\negthickspace
      & = & \displaystyle\negthickspace
            \left\{
              \Gamma_{i1,v}^1+\Gamma_{i1}^1\Gamma_{12}^1+\Gamma_{i1}^2\Gamma_{22}^1
              -\sum_{\ell=1}^2\sum_{\sigma=1}^n
               L_{\sigma,i1}L_{\sigma,2\ell}g^{\ell 1}
            \right\}X_u \\[4ex]
      &   & \displaystyle\negthickspace
            +\left\{
               \Gamma_{i1,v}^2+\Gamma_{i1}^1\Gamma_{12}^2+\Gamma_{i1}^2\Gamma_{22}^2
               -\sum_{\ell=1}^2\sum_{\sigma=1}^n
                L_{\sigma,i1}L_{\sigma,2\ell}g^{\ell 2}
             \right\}X_v^t \\[4ex]
      &   & \displaystyle\negthickspace
            +\,\sum_{\omega=1}^n
              \left\{ 
                L_{\omega,i1,v}
                +\Gamma_{i1}^1L_{\omega,12}
                +\Gamma_{i1}^2L_{\omega,22}  
                +\sum_{\sigma=1}^n
                 L_{\sigma,i1}T_{\sigma,2}^\omega
              \right\}N_\omega\,,
    \end{array}$$
and the second equations w.r.t. $u,$
  $$\begin{array}{lll}
      X_{u^ivu}\negthickspace
      & = & \displaystyle\negthickspace
            \left\{
              \Gamma_{i2,u}^1+\Gamma_{i2}^1\Gamma_{11}^1+\Gamma_{i2}^2\Gamma_{12}^1
              -\sum_{\ell=1}^2\sum_{\sigma=1}^n
               L_{\sigma,i2}L_{\sigma,1\ell}g^{\ell 1}
            \right\}X_u \\[4ex]
      &   & \displaystyle\negthickspace
            +\left\{
               \Gamma_{i2,u}^2+\Gamma_{i2}^1\Gamma_{11}^2+\Gamma_{i2}^2\Gamma_{12}^2
               -\sum_{\ell=1}^2\sum_{\sigma=1}^n
                L_{\sigma,i2}L_{\sigma,1\ell}g^{\ell 2}
             \right\}X_v^t \\[4ex]
      &   & \displaystyle\negthickspace
            +\,\sum_{\omega=1}^n   
              \left\{
                L_{\omega,i2,u}
                +\Gamma_{i2}^1L_{\omega,11}
                +\Gamma_{i2}^2L_{\omega,12}
                +\sum_{\sigma=1}^n
                 L_{\sigma,i2}T_{\sigma,1}^\omega
              \right\}N_\omega\,.
    \end{array}$$
Comparing the tangent and normal parts of these two identities gives the first set of integrability conditions.
\subsection{The integrability conditions of Codazzi and Mainardi}
\label{par_codazzi_mainardi}
From $(X_{u^iuv}-X_{u^ivu})^{\rm norm}\equiv 0$ we infer (we interchange $\sigma$ and $\omega$)
\begin{proposition}
Let the immersion $X$ together with a normal frame $N$ be given. Then there hold\index{Codazzi-Mainardi equations}
  $$L_{\sigma,i1,v}+\Gamma_{i1}^1L_{\sigma,12}+\Gamma_{i1}^2L_{\sigma,22}
    +\sum_{\omega=1}^nL_{\omega,i1}T_{\omega,2}^\sigma
    =L_{\sigma,i2,u}+\Gamma_{i2}^1L_{\sigma,11}+\Gamma_{i2}^2L_{\sigma,12}
     +\sum_{\omega=1}^nL_{\omega,i2}T_{\omega,1}^\sigma$$
for $i=1,2$ and $\sigma=1,2,\ldots,n.$
\end{proposition}
\noindent
{\it These equations contain the torsion coefficients from above.} These coefficients do not appear in case $n=1.$
\subsection{The integrability conditions of Gau\ss{}}
\label{par_gaussintegrability}
From $(X_{u^iuv}-X_{u^ivu})^{\rm tang}\equiv 0$ we infer\index{Gauss integrability conditions}
\begin{proposition}
Let the immersion $X$ together with a normal frame $N$ be given. Then there hold
  $$\Gamma_{i1,v}^\ell-\Gamma_{i2,u}^\ell
    +\sum_{m=1}^2
     \Gamma_{i1}^m\Gamma_{m2}^\ell
    -\sum_{m=1}^2
     \Gamma_{i2}^m\Gamma_{m1}^\ell
    =\sum_{m=1}^2\sum_{\sigma=1}^n
     (L_{\sigma,i1}L_{\sigma,2m}-L_{\sigma,i2}L_{\sigma,1m})g^{m\ell}$$
for $i,\ell=1,2.$
\end{proposition}
\noindent
Note that {\it these equations do not contain the torsion coefficients.} Rather they belong to the {\it inner geometry} of the surface as will become more clear in the next paragraph.
\subsection{The curvature tensor of the tangent bundle}
\label{par_curvaturetangent}
The left hand side of the Gau\ss{} integrability conditions gives reason to our next
\begin{definition}
The curvature tensor of the tangent bundle\index{curvature tensor of tangent bundle} of the immersion $X,$ the so-called Riemannian curvature tensor\index{Riemannian curvature tensor}, is given by components
  $$R_{ijk}^\ell
    :=\Gamma_{ij,u^k}^\ell-\Gamma_{ik,u^j}^\ell
      +\sum_{m=1}^2
       \Big\{
         \Gamma_{ij}^m\Gamma_{mk}^\ell-\Gamma_{ik}^m\Gamma_{mj}^\ell
       \Big\}$$
for $i,j,k,\ell=1,2.$ Its covariant components are
  $$R_{nijk}
    =\sum_{\ell=1}^2
     R_{ijk}^\ell g_{\ell n}\,.$$
\end{definition}
\noindent
In our case of two dimensions of $X,$ these $R_{nijk}$ reduce to one essential component:
  $$\begin{array}{l}
      R_{1111}=0,\quad
      R_{2222}=0,\quad
      R_{1222}=0,\quad
      R_{2111}=0,\quad
      R_{2221}=0,\quad
      R_{1112}=0,      \\[1ex]
      R_{1122}=0,\quad
      R_{2211}=0,\quad
      R_{1121}=0,\quad
      R_{1211}=0,\quad
      R_{2212}=0,\quad
      R_{2122}=0,      \\[1ex]
      R_{2112}=R_{1221}=-R_{2121}=-R_{1212}\,.
    \end{array}$$
\subsection{The Gau\ss{} curvature. theorema egregium}
\label{par_theoremaegregium}
{\it This one essential component $R_{2112}$ of the Riemannian curvature tensor represents the inner curvature of the surface} $X.$ This is the contents of the next
\begin{proposition}
Let the immersion $X$ together with a normal frame $N$ be given. Then it holds\index{theorema egregium}
  $$R_{2112}=KW^2$$
with the Gaussian curvature\index{Gaussian curvature} of the immersion defined by
  $$K:=\sum_{\sigma=1}^nK_\sigma
    \quad\mbox{with}\ K_\sigma:=\frac{L_{\sigma,11}g_{22}-2L_{\sigma,12}g_{12}+L_{\sigma,22}g_{22}}{W^2}$$
and the area element $W.$
\end{proposition}
\begin{proof}
Using the Gau\ss{} integrability conditions we compute
  $$\begin{array}{lll}
      R_{2112}\negthickspace
      & = & \displaystyle\negthickspace
            \sum_{\ell=1}^2
            R_{112}^\ell g_{\ell2}
            \,=\,\sum_{\ell,m=1}^2\sum_{\sigma=1}^n
                 (L_{\sigma,11}L_{\sigma,2m}-L_{\sigma,12}L_{\sigma,1m})
                 g^{m\ell}g_{\ell 2} \\[3ex]
      & = & \displaystyle\negthickspace
            \sum_{\ell=1}^2\sum_{\sigma=1}^n
            (L_{\sigma,11}L_{\sigma,22}-L_{\sigma,12}L_{\sigma,12})
            g^{2\ell}g_{\ell 2}
            \,=\,\sum_{\sigma=1}^nK_\sigma W^2
            \,=\,KW^2\,,
    \end{array}$$
and the statement follows.
\end{proof}
\begin{remark}
The Gauss curvature $K$ does neither depend on the choice of the parametrization $(u,v)\in\overline B$ nor on the choice of the normal frame $N$ (of course, its components $K_\sigma$ are not invariant). We skip a proof of these invariance properties but similar calculations follow below.
\end{remark}
\subsection{The integrability conditions of Ricci}
\label{par_ricci}
Now we want to derive an integrability condition which has no counterpart in case $n=1$ of one codimension. Let us first compute the second derivatives of the unit normal vectors of some normal frame $N:$
  $$\begin{array}{lll}
      N_{\sigma,uv}\negthickspace
      & = & \displaystyle\negthickspace
            -\,\sum_{j,k=1}^2
               L_{\sigma,1j,v}g^{jk}X_{u^k}
            -\sum_{j,k=1}^2
             L_{\sigma,1j}g^{jk}_{\ \ ,v}X_{u^k}
            -\sum_{j,k=1}^2
             L_{\sigma,1j}g^{jk}X_{u^kv}
            +\sum_{\omega=1}^nT_{\sigma,1,v}^\omega N_\omega
            +\sum_{\omega=1}^nT_{\sigma,1}^\omega N_{\omega,v} \\[4ex]
      & = & \displaystyle\negthickspace
            -\,\sum_{j,k=1}^2
               \left\{
                 L_{\sigma,1j,v}g^{jk}
                 +L_{\sigma,1j}g^{jk}_{\ \ ,v}
                 +\sum_{m=1}^2
                  L_{\sigma,1j}g^{jm}\Gamma_{m2}^k
                 +\sum_{\vartheta=1}^nT_{\sigma,1}^\vartheta L_{\vartheta,2j}g^{jk}
              \right\}\,X_{u^k} \\[5ex]
      &   & \displaystyle\negthickspace
            -\,\sum_{\omega=1}^n
               \left\{
                 \sum_{j,k=1}^2
                 L_{\sigma,1j}g^{jk}L_{\omega,k2}
                 -T_{\sigma,1,v}^\omega
                 -\sum_{\vartheta=1}^nT_{\sigma,1}^\vartheta T_{\vartheta,2}^\omega
               \right\}\,N_\omega
    \end{array}$$
as well as
  $$\begin{array}{lll}
      N_{\sigma,vu}\negthickspace
      & = & \displaystyle\negthickspace
             -\,\sum_{j,k=1}^2
                \left\{
                  L_{\sigma,2j,u}g^{jk}+L_{\sigma,2j}g^{jk}_{\ \ ,u}
                  +\sum_{m=1}^2
                   L_{\sigma,2j}g^{jm}\Gamma_{m1}^k
                  +\sum_{\vartheta=1}^nT_{\sigma,2}^\vartheta L_{\vartheta,1j}g^{jk}
                \right\}\,X_{u^k} \\[5ex]
      &   & \displaystyle\negthickspace
            -\,\sum_{\omega=1}^n
               \left\{
                 \sum_{j,k=1}^2
                 L_{\sigma,2j}g^{jk}L_{\omega,k1}-T_{\sigma,2,u}^\omega
                 -\sum_{\vartheta=1}^nT_{\sigma,2}^\vartheta T_{\vartheta,1}^\omega
               \right\}\,N_\omega
    \end{array}$$
using the Gauss and the Weingarten equations. It holds necessarily
  $$N_{\sigma,uv}-N_{\sigma,vu}\equiv 0\quad\mbox{for all}\ \sigma=1,\ldots,n.$$
\begin{remark}
A comparison between the tangential parts yields again the Codazzi-Mainardi equations.
\end{remark}
\noindent
The integrability condition $(N_{\sigma,uv}-N_{\sigma,vu})^{\rm norm}\equiv 0$ {\it is trivially satisfied in case one codimension:}
  $$T_{\sigma,i}^\vartheta\equiv 0\quad\mbox{and}\quad
    \sum_{j,k=1}^2L_{1j}g^{jk}L_{k2}
    =\sum_{j,k=1}^2 L_{2j}g^{jk}L_{k1}$$
by symmetry of the $g^{jk}$ and the coefficients $L_{ij}$ of the second fundamental form. {\it But in general case of higher codimension these identites are not trivial,} rather they yield the so-called {\it Ricci integrability equations.}
\begin{proposition}
There hold\index{Ricci equations}
  $$T_{\sigma,2,u}^\omega-T_{\sigma,1,v}^\omega
    +\sum_{\vartheta=1}^n
     T_{\sigma,2}^\vartheta T_{\vartheta,1}^\omega
    -\sum_{\vartheta=1}^n
     T_{\sigma,1}^\vartheta T_{\vartheta,2}^\omega
    =\sum_{j,k=1}^2
     (L_{\sigma,2j}L_{\omega,k1}-L_{\sigma,1j}L_{\omega,k2})g^{jk}$$
for $\sigma,\omega=1,\ldots,n.$
\end{proposition}
\noindent
Both sides of these identitites vanish identically if $n=1.$
\vspace*{6ex}
\begin{center}
\rule{35ex}{0.1ex}
\end{center}
\vspace*{6ex}
\section{The curvature of the normal bundle}
\label{section_curvaturenormalbundle}
\subsection{Problem statement}
\label{par_problemnormal}
In the same manner as we derived the Riemannian curvature tensor from the Gau\ss{} integrability conditions we proceed to derive the curvature tensor in the normal space from the Ricci integrability conditions. In this section we give a detailed description of this curvature quantity.
\begin{definition}
The normal bundle\index{normal bundle} of the immersion $X$ is given by
  $${\mathcal N}(X)=\bigcup_{w\in\overline B}\big(w,{\mathbb N}_X(w)\big).$$
\end{definition}
\noindent
Here are some simple examples:
\begin{itemize}
\item[1.]
The normal bundle of a surface in $\mathbb R^3$ is the collection of all normal lines, thus it resembles the Grassmann manifold $G_{3,1}.$\index{Grassmann manifold $G_{3,1}$}
\vspace*{-1.6ex}
\item[2.]
Tubular neighborhoods\index{tubular neighborhoods} of curves or surfaces are resembled by its normal bundle, see also the parallel type surfaces in the next chapter.
\end{itemize}
In case of higher codimension the normal bundle possesses an own non-trivial geometry. In particular, we can assign a curvature to the normal bundle. If this curvature vanishes identically then the normal bundle is called {\it flat.}\index{flat normal bundle} But in general it is curved. For example, $X(z)=(z,z^2)$ has non-flat normal bundle. To develop a possible analytical method to describe curved normal bundles is our concern.
\subsection{The curvature tensor of the normal bundle}
\label{par_curvnormbundle}
This is the tensor consisting of the components (compare with the considerations from paragraph \ref{par_ricci})\index{curvature tensor of normal bundle}
  $$S_{\sigma,ij}^\omega
   :=T_{\sigma,i,u^j}^\omega-T_{\sigma,j,u^i}^\omega
      +\sum_{\vartheta=1}^n
       \big(
         T_{\sigma,i}^\vartheta T_{\vartheta,j}^\omega
         -T_{\sigma,j}^\vartheta T_{\vartheta,i}^\omega
       \big)
    =\sum_{m,n=1}^2
     (L_{\sigma,im}L_{\omega,jn}-L_{\sigma,jm}L_{\omega,in})g^{mn}\,.$$
Note that the second identity follows from the integrability conditions of Ricci. The $S_{\sigma,ij}^\omega$ now take the role of the $R_{ijk}^\ell.$ Without proof we want to remark that the $S_{\sigma,ij}^\omega$ behave like a tensor of rank $2$ under parameter transformations, and are neither invariant w.r.t. parameter transformations nor on rotations of the normal frame, see our considerations below. Furthermore it is sufficient to focus on the components $S_{\sigma,12}^\omega$ because all other components vanish or are equal to $S_{\sigma,12}^\vartheta$ up to sign. Using conformal parameters we arrive at
  $$S_{\sigma,12}^\omega
    =\frac{1}{W}\,(L_{\sigma,11}-L_{\sigma,22})L_{\omega,12}
     -\frac{1}{W}\,(L_{\omega,11}-L_{\omega,22})L_{\sigma,12}\,.$$
A general definition of curvatures for connections can be found in Helein \cite{helein_2002}, chapter 2.
\subsection{The case $n=2$}
\label{par_lecture2n2}
\label{case_n2}
The definition for $S_{\sigma,12}^\vartheta$ takes a particular form in the special case $n=2:$
  $$S_{1,12}^2
    =T_{1,1,v}^2-T_{1,2,u}^2
     +T_{1,1}^1T_{1,2}^2+T_{1,1}^2T_{2,2}^2
     -T_{1,2}^1T_{1,1}^2-T_{1,2}^2T_{2,1}^2
    =\mbox{div}\,(-T_{1,2}^2,T_{1,1}^2)$$
where only $\sigma=1$ and $\vartheta=2$ are taken into account. Thus we define
\begin{definition}
The normal curvature of the two-dimensional surface $X\colon\overline B\to\mathbb R^4$ is given by\index{normal curvature}
  $$S_N:=\frac{1}{W}\,S_{1,12}^2=-\,\frac{1}{W}\,\mbox{\rm div}\,(T_{1,2}^2,-T_{1,1}^2).$$
\end{definition}
\noindent
Now {\it this curvature $S_N$ does not depend on the parametrization,} and, as we will see later, {\it it does not depend on the choice of the normal frame either.} It belongs to the inner geometry of the surface. The general situation of higher codimension is treated next.
\subsection{The normal sectional curvature}
\label{par_normalenschnittkruemmung}
{\it But if $n>2$ then the components $S_{\sigma,12}^\omega$ depend on the choice of the normal frame.} So let us fix an index pair $(\sigma,\omega)\in\{1,\ldots,n\}\times\{1,\ldots,n\}.$
\begin{proposition}
The quantity $S_{\sigma,12}^\omega$ is invariant w.r.t. rotations of the normal frame $\{N_\sigma,N_\omega\}$ spanning the plane ${\mathcal E}=\mbox{\rm span}\,\{N_\sigma,N_\omega\},$ i.e. under $SO(2)$-regulare mappings of the form
  $$\widetilde N_\sigma=\cos\varphi N_\sigma+\sin\varphi N_\omega\,,\quad
    \widetilde N_\omega=-\sin\varphi N_\sigma+\cos\varphi N_\omega\,.$$
\end{proposition}
\begin{proof}
For the proof we use conformal parameters $(u,v)\in\overline B.$\footnote{Note that in worst case the $S_{\sigma,12}^\omega$ differ by a Jacobian after parameter transformations.} First, with $\widetilde L_{\sigma,ij}=-N_{\sigma,u^i}\cdot X_{u^j}$ we compute
  $$\begin{array}{lll}
      W\widetilde S_{\sigma,12}^\omega\negthickspace
      & = & \displaystyle\negthickspace
            (\widetilde L_{\sigma,11}\widetilde L_{\omega,12}
            -\widetilde L_{\sigma,21}\widetilde L_{\omega,11})
            +(\widetilde L_{\sigma,12}\widetilde L_{\omega,22}
            -\widetilde L_{\sigma,22}\widetilde L_{\omega,21}) \\[2ex]
      & = & \displaystyle\negthickspace
            (\cos\varphi\,L_{\sigma,11}+\sin\varphi\,L_{\omega,11})
            (-\sin\varphi\,L_{\sigma,12}+\cos\varphi\,L_{\omega,12}) \\[2ex]
      &   & \displaystyle\negthickspace
            -\,(\cos\varphi\,L_{\sigma,21}+\sin\varphi\,L_{\omega,21})
               (-\sin\varphi\,L_{\sigma,11}+\cos\varphi\,L_{\omega,11}) \\[2ex]
      &   & \displaystyle\negthickspace
            +\,(\cos\varphi\,L_{\sigma,12}+\sin\varphi\,L_{\omega,12})
               (-\sin\varphi\,L_{\sigma,22}+\cos\varphi\,L_{\omega,22}) \\[2ex]
      &   & \displaystyle\negthickspace
            -\,(\cos\varphi\,L_{\sigma,22}+\sin\varphi\,L_{\omega,22})
               (-\sin\varphi\,L_{\sigma,21}+\cos\varphi\,L_{\omega,21}).
    \end{array}$$
Collecting and evaluating all the trigonometic squares gives
  $$W\widetilde S_{\sigma,12}^\omega
    =(L_{\sigma,11}-L_{\sigma,22})L_{\omega,12}
     -(L_{\omega,11}-L_{\omega,22})L_{\sigma,12}
    =-WS_{\sigma,12}^\omega$$
proving the statement.
\end{proof}
\noindent
Thus we can make the following
\begin{definition}
The quantity $W^{-1}S_{\sigma,12}^\omega$ represents a parametrization invariant quantity
  $$S_{N,\sigma}^\omega:=\frac{1}{W}\,S_{\sigma,12}^\omega$$
called the normal sectional curvature\index{normal sectional curvature} w.r.t. the plane ${\mathcal E}=\mbox{\rm span}\,\{N_\sigma,N_\omega\}.$
\end{definition}
\noindent
If $n=2$ then there is only one normal sectional curvature $S_N$ which is even independent of the choice of the normal frame.
\subsection{Preparations for the normal curvature vector I: Curvature matrices}
\label{par_curvaturematrices}
For the following we let
  $$T_i=(T_{\sigma,i}^\vartheta)_{\sigma,\vartheta=1,\ldots,n}\in\mathbb R^{n\times n}\,,\quad
    S_{12}=(S_{\sigma,12}^\vartheta)_{\sigma,\vartheta=1,\ldots,n}\in\mathbb R^{n\times n}\,.$$
We consider general rotations
  $$R=(R_\sigma^\vartheta)_{\sigma,\vartheta=1,\ldots,n}\in C^{k-1,\alpha}(\overline B,SO(n))$$
as special orthogonal mappings in the normal space which transform a given normal frame $N$ into a new one $\widetilde N$ by means of the transformation
  $$\widetilde N_\sigma
    =\sum_{\vartheta=1}^n
     R_\sigma^\vartheta N_\vartheta$$
for $\sigma=1,\ldots,n.$
\begin{lemma}
There hold the transformation rule\index{curvature matrix}
  $$\widetilde S_{12}=R\circ S_{12}\circ R^t$$
with $R^t$ being the transposition of $R.$
\end{lemma}
\begin{proof}
For the proof we consider
$$\begin{array}{lll}
      \displaystyle
      \widetilde T_{\sigma,i}^\vartheta\negthickspace
      & = & \negthickspace\displaystyle
            \widetilde N_{\sigma,u^i}\cdot\widetilde N_\vartheta
            \,=\,\sum_{\alpha=1}^n
                 \big(R_{\sigma,u^i}^\alpha N_\alpha+R_\sigma^\alpha N_{\alpha,u^i}\big)\cdot
                 \sum_{\beta=1}^nR_\vartheta^\beta N_\beta \\[3ex]
     & = & \displaystyle\negthickspace
           \sum_{\alpha,\beta=1}^n
           \big(R_{\sigma,u^i}^\alpha R_\vartheta^\beta\delta_{\alpha\beta}
            +R_\sigma^\alpha R_\vartheta^\beta T_{\alpha,i}^\beta\big)
           \,=\,\sum_{\alpha=1}^n
                R_{\sigma,u^i}^\alpha (R_\alpha^\vartheta)^t
                 +\sum_{\alpha,\beta=1}^nR_\sigma^\alpha T_{\alpha,i}^\beta(R_\beta^\vartheta)^t
    \end{array}$$
taking account of $(R_\vartheta^\alpha)_{\vartheta,\alpha=1,\ldots,n}=(R_\alpha^\vartheta)_{\alpha,\vartheta=1,\ldots,n}^t.$ Thus we arrive at the rule
  $$\widetilde T_i=R_{u^i}\circ R^t+R\circ T_i\circ R^t\,.$$
Using this formula we evaluate
  $$\widetilde S_{12}
    =\widetilde T_{1,v}-\widetilde T_{2,u}-\widetilde T_1\circ\widetilde T_2^t+\widetilde T_2\circ\widetilde T_1^t\,.$$
First
  $$\begin{array}{lll}
      \widetilde T_{1,v}-\widetilde T_{2,u}\negthickspace
      & = & \displaystyle\negthickspace
            (R_u\circ R^t+R\circ T_1\circ R^t)_v
            -(R_v\circ R^t+R\circ T_2\circ R^t)_u \\[2ex]
      & = & \displaystyle\negthickspace
            R_u\circ R_v^t-R_v\circ R_u^t
            +R\circ(T_{1,v}-T_{2,u})\circ R^t \\[2ex]
      &   & \displaystyle\negthickspace
            +\,R_v\circ T_1\circ R^t
            +R\circ T_1\circ R_v^t
            -R_u\circ T_2\circ R^t
            -R\circ T_2\circ R_u^t\,,
    \end{array}$$
and furthermore
  $$\begin{array}{lll} 
     \widetilde T_1\circ\widetilde T_2^t-\widetilde T_2\circ\widetilde T_1^t\negthickspace
      & = & \displaystyle\negthickspace
            (R_u\circ R^t+R\circ T_1\circ R^t)
            \circ(R\circ R_v^t+R\circ T_2^t\circ R^t) \\[2ex]
      &   & \displaystyle\negthickspace
            -\,(R_v\circ R^t+R\circ T_2\circ R^t)
               \circ(R\circ R_u^t+R\circ T_1^t\circ R^t) \\[2ex]
      & = & \displaystyle\negthickspace
            R_u\circ R_v^t
            +R_u\circ T_2^t\circ R^t
            +R\circ T_1\circ R_v^t
            +R\circ T_1\circ T_2^t\circ R^t \\[2ex]
      &   & \displaystyle\negthickspace
            -\,R_v\circ R_u^t
            -R_v\circ T_1^t\circ R^t
            -R\circ T_2\circ R_u^t
            -R\circ T_2\circ T_1^t\circ R^t\,
    \end{array}$$
because $R\circ R^t=R^t\circ R=\mbox{id}.$ Taking both identities together gives
  $$\begin{array}{l}
      \widetilde T_{1,v}-\widetilde T_{2,u}-\widetilde T_1\circ\widetilde T_2^t+\widetilde T_2\circ\widetilde T_1^t \\[2ex]
      \hspace*{3ex}\displaystyle
      =\,R\circ(T_{1,v}-T_{2,u}-T_1\circ T_2^t+T_2\circ T_1^t)\circ R^t \\[2ex]
      \hspace*{6ex}\displaystyle
      +\,R_v\circ T_1\circ R^t+R\circ T_1\circ R_v^t
      -R_u\circ T_2\circ R^t-R\circ T_2\circ R_u^t \\[2ex]
      \hspace*{6ex}\displaystyle
      -\,R_u\circ T_2^t\circ R^t-R\circ T_1\circ R_v^t
      +R_v\circ T_1^t\circ R^t+R\circ T_2\circ R_u^t \\[2ex]
      \hspace*{3ex}\displaystyle
      =\,R\circ
         (T_{1,v}- T_{2,u}-T_1\circ T_2^t+T_2\circ T_1^t)
         \circ R^t
    \end{array}$$
using $T_i=-T_i^t$. This proves the statement.
\end{proof}
\subsection{Preparations for the normal curvature vector II: The exterior product}
\label{par_preparations2}
For the following algebraic concepts of the Grassmann geometry we refer to Cartan \cite{cartan_1974} or Heil \cite{heil_1974}.
\begin{definition}
The {\it exterior product}\index{exterior product}
\begin{equation*}
  \wedge\colon\mathbb R^n\times\mathbb R^n\to\mathbb R^N\,,\quad
  N=\binom{n}{2}=\frac{n(n+1)}{2}\,,
\end{equation*}
is defined by means of the following rules:
\begin{itemize}
\item[(E1)]
\quad
The mapping $\mathbb R^n\times\mathbb R^n\ni(X,Y)\mapsto X\wedge Y\in\mathbb R^N$ is bilinear,
  $$(\alpha_1X_1+\alpha_2X_2)\wedge(\beta_1Y_1+\beta_2Y_2)
    =\alpha_1\beta_2\,X_1\wedge Y_1
     +\alpha_1\beta_2\,X_1\wedge Y_2
     +\alpha_2\beta_1 X_2\wedge Y_1
     +\alpha_2\beta_2 X_2\wedge Y_2$$
for all $\alpha_i,\beta_i\in\mathbb R$ and $X_i,Y_i\in\mathbb R^n;$ and $\wedge$ is skew-symmetric,
  $$X\wedge Y=-Y\wedge X$$
for all $X,Y\in\mathbb R^n;$ in particular, it holds $X\wedge X=0.$
\vspace*{0.6ex}
\item[(E2)]
\quad
Let $e_1=(1,0,0,\ldots,0)\in\mathbb R^n,$ $e_2=(0,1,0,\ldots,0)\in\mathbb R^n$ etc. represent the standard orthonormal basis in $\mathbb R^n.$ Then we set
  $$\begin{array}{rcl}
      e_1\wedge e_2
      & := & (1,0,0,\ldots,0,0)\in\mathbb R^N\,, \\[2ex]
      e_1\wedge e_3
      & := & (0,1,0,\ldots,0,0)\in\mathbb R^N\,, \\[2ex]
      & & \vdots\\[2ex]
      e_{n-1}\wedge e_n
      & := & (0,0,0,\ldots,0,1)\in\mathbb R^N\,.
    \end{array}$$
\end{itemize}
\end{definition}
\noindent
From this settings we immediately obtain the
\begin{lemma}
The vectors $e_k\wedge e_\ell$ form a basis of $\mathbb R^N$ which is orthonormal w.r.t. the Euldidean metric
  $$(e_i\wedge e_j)\cdot(e_k\wedge e_\ell)
    =\left\{
       \begin{array}{l}
         1\quad\mbox{if}\ i=k\ \mbox{and}\ j=\ell \\[1ex]
         0\quad\mbox{if}\ i\not=k\ \mbox{or}\ j\not=\ell
       \end{array}
     \right..$$
\end{lemma}
\begin{lemma}
For two vectors $X=(x^1,\ldots,x^n)$ and $Y=(y^1,\ldots,y^n)$ it holds
  $$X\wedge Y=\sum_{1\le i<j\le n}(x^iy^j-x^jy^i)e_i\wedge e_j\,.$$
\end{lemma}
\begin{proof}
We compute
  $$X\wedge Y
    =\left(\sum_{i=1}^nx^ie_i\right)\wedge\left(\sum_{j=1}^ny^je_j\right)
    =\sum_{i,j=1}^nx^iy^je_i\wedge e_j
    =\sum_{1\le i<j\le n}(x^iy^j-x^jy^i)e_i\wedge e_j\,,$$
proving the statement.\qed
\end{proof}
\noindent
Note the following
\begin{remark}
For $n=3$ we have
  $$e_1\wedge e_2=(1,0,0),\quad
    e_1\wedge e_3=(0,1,0),\quad
    e_2\wedge e_3=(0,0,1),$$
and then some vectors $X=(x^1,x^2,x^3)$ and $Y=(y^1,y^2,y^3)$ we compute
  $$\begin{array}{lll}
      X\wedge Y
      & = & \displaystyle
            x^1y^2\,e_1\wedge e_2
            -x^1y^3\,e_1\wedge e_3
            +x^2y^1\,e_2\wedge e_1
            +x^2y^3\,e_2\wedge e_3
            -x^3y^1\,e_3\wedge e_1
            +x^3y^2\,e_3\wedge e_2 \\[2ex]
      & = & \displaystyle
            (x^1y^2-x^2y^1)e_1\wedge e_2
            +(x^3y^1-x^1y^3)e_1\wedge e_3
            +(x^2y^3-x^3y^2)e_2\wedge e_3 \\[2ex]
      & = & \displaystyle
            (x^1y^2-x^2y^1,x^3y^1-x^1y^3,x^2y^3-x^3y^2).
    \end{array}$$
In other words, the usual vector product $X\times Y$ in $\mathbb R^3$ does not coincide with the exterior product $X\wedge Y,$
  $$X\times Y=(x^2y^3-x^3y^2,x^3y^1-x^1y^3,x^1y^2-x^2y^1)\not=X\wedge Y.$$
\end{remark}
\noindent
Without proof we want to collect some algebraic and analytical properties of the exterior product.
\begin{lemma}
For arbitrary vectors $A,B,C\in\mathbb R^n$ there hold
\begin{itemize}
\item[$\bullet$]
$(\lambda A)\wedge B=\lambda(A\wedge B);$
\vspace*{0.6ex}
\item[$\bullet$]
$(A+B)\wedge C=A\wedge C+B\wedge C;$
\vspace*{0.6ex}
\item[$\bullet$]
$(A\wedge B)_{u^i}=A_{u^i}\wedge B+A\wedge B_{u^i}\,.$
\end{itemize}
Finally, let two vectors $X=(x^1,x^2,0,\ldots,0)$ and $Y=(y^1,y^2,0,\ldots,0)$ be given, i.e. assume that $X,Y\in\mbox{\rm span}\,\{e_1,e_2\}.$ Then it holds
\begin{itemize}
\item[$\bullet$]
$X\wedge Y
 \perp\mbox{\rm span}\,
      \big\{
        X\wedge e_3,\ldots,X\wedge e_n,Y\wedge e_3,\ldots,Y\wedge e_n,e_3\wedge e_n,\ldots,e_{n-1}\wedge e_n
      \big\}.$
\end{itemize}
\end{lemma}
\subsection{The curvature vector of the normal bundle}
\label{par_normalcurvaturevector}
Now let us come back to transformation rule for $\widetilde S_{12}$ which is the basic for definiting of a new geometric curvature quantity.
\begin{definition}
The curvature vector of the normal bundle is given by\index{curvature vector of normal bundle}
  $${\mathfrak S}_N
    :=\frac{1}{W}\,
      \sum_{1\le\sigma<\vartheta\le n}
      S_{\sigma,12}^\vartheta\,N_\sigma\wedge N_\vartheta\,.$$
Here $\wedge$ denotes the exterior product between two vectors in $\mathbb R^{n+2}$ from the previous paragraph. If $n=2$ then ${\mathfrak S}_N$ is a scalar, and we simply write $S_N={\mathfrak S}_N$ as before.
\end{definition}
\begin{proposition}
The curvature vector of the normal bundle neither depends on the parametrization nor on the choice of the normal frame. In particular, its length
  $$|{\mathfrak S}_N|
    =\sqrt{{\mathfrak S}_N\cdot{\mathfrak S}_N}
    =\sqrt{\,\frac{1}{W}\sum_{1\le\sigma<\vartheta\le n}(S_{\sigma,12}^\vartheta)^2}$$
represents a geometric quantity, the so-called curvature of the normal bundle.
\end{proposition}
\begin{proof}
We check the invariance w.r.t. rotations: Using our transformation rule from paragraph \ref{par_curvaturematrices} we get
  $$\begin{array}{lll}
      \displaystyle
      \sum_{\sigma,\vartheta=1}^n
      \widetilde S_{\sigma,12}^\vartheta\,\widetilde N_\sigma\wedge\widetilde N_\vartheta\negthickspace
      & = & \displaystyle\negthickspace
            \sum_{\sigma,\vartheta=1}^n
            \sum_{\alpha,\beta=1}^n
            \widetilde S_{\sigma,12}^\vartheta R_\sigma^\alpha R_\vartheta^\beta\,N_\alpha\wedge N_\beta \\[4ex]
      & = & \displaystyle\negthickspace
            \sum_{\sigma,\vartheta=1}^n
            \sum_{\alpha,\beta=1}^n
            (R^t)_\alpha^\sigma\widetilde S_{\sigma,12}^\vartheta R_\vartheta^\beta\,N_\alpha\wedge N_\beta \\[4ex]
      & = & \displaystyle\negthickspace
            \sum_{\alpha,\beta=1}^n
            S_{\alpha,12}^\beta\,N_\alpha\wedge N_\beta
    \end{array}$$
which already proves the statement.
\end{proof}
\noindent
We want to point out that in case $n=2$ we can distinguish the signs of the normal curvature $S_N$ - positive or negative. In contrast to this special situation, the normal curvature is vector-valued if $n>2$ so that in general we can not speak of ``negatively'' or ``positively'' curved normal bundles.\\[1ex]
It seems to us the the normal curvature vector ${\mathfrak S}_N$ has not been considered in the literature so far, though manifolds with normal curvature are already widely studied, see e.g. Asperti \cite{asperti_1983}.
\vspace*{6ex}
\begin{center}
\rule{35ex}{0.1ex}
\end{center}
\vspace*{6ex}
\section{Elliptic systems}
\label{section_ellipticsystems}
\subsection{The mean curvature vector}
\label{par_meancurvaturevector}
The construction of suitable moving frames $N$ in the normal bundle of surfaces requires a profound knowledge on the analytical behaviour of these geometric objects. For that purpose we already start establishing some basic estimates for conformally parametrized immersions of prescribed mean curvature vector.
\begin{definition}
The mean curvature $H_N$ of an immersion $X$ {\it w.r.t. any unit normal vector} $N_\sigma$ is defined as\index{mean curvature}
  $$H_{N_\sigma}
   :=\frac{1}{2}\,\sum_{i,j=1}^2g^{ij}L_{N_\sigma,ij}
    =\frac{L_{N_\sigma,11}g_{22}-2L_{N_\sigma,12}g_{12}+L_{N_\sigma,22}g_{11}}{2W^2}\,.$$
\end{definition}
\noindent
Consider a normal frame $N=(N_1,\ldots,N_n),$ and set $H_\sigma:=H_{N_\sigma}.$
\begin{definition}
The mean curvature vector $H\in\mathbb R^n$ of the immersion $X$ is\index{mean curvature vector}
  $$H:=\sum_{\sigma=1}^nH_\sigma N_\sigma\,.$$
\end{definition}
\noindent
For surfaces in $\mathbb R^3$ there is, up to orientation, exactly one mean curvature and, thus, exactly one mean curvature vector. We could be misleaded to believe that in the general situation here the mean curvature vector $H$ could {\it replace} this special unit normal vector $N$ for surfaces in one codimension. This is not the case since, for example, for minimal surfaces the vector $H$ {\it vanishes identically.}
\begin{definition}
The immersion $X$ is called a minimal surface if and only if\index{minimal surface}
  $$H\equiv 0\quad\mbox{in}\ B.$$
\end{definition}
\noindent
The property $H\equiv 0$ does not depend on the choice of the normal frame $N.$ In fact, it holds even more: {\it The mean curvature vector $H$ neither depends on the parametrization nor on the choice of the normal frame.}\\[1ex]
Minimal surfaces are the topic of a huge literature: Courant \cite{courant_1950}, Nitsche \cite{nitsche_1975}, Osserman \cite{osserman_1986}, Dierkes et al. \cite{dhkw_1992}, Colding and Minicozzi \cite{colding_minicozzi_1999}, Eschenburg and Jost \cite{eschenburg_jost_2007} to enumerate only some few significant distributions and to illustrate the importance of this surface class in the field of geometric analysis.
\subsection{The mean curvature system}
\label{par_meancurvaturesystem}
From the Gauss equations together with the conformal representation of the Christoffel symbols from paragraph \ref{par_christoffel} we derive an elliptic system for conformally parametrized immersions with prescribed mean curvature vector $H$ as follows.
\begin{proposition}
Given the conformally parametrized immersion $X$ of prescribed mean curvature vector $H.$ Then it holds
  $$\triangle X=2\sum_{\vartheta=1}^nH_\vartheta WN_\vartheta=2HW
    \quad\mbox{in}\ B$$
where $N$ is an arbitrary normal frame.\index{mean curvature system}
\end{proposition}
\begin{proof}
From the Gauss equations we infer
  $$\triangle X
    =(\Gamma_{11}^1+\Gamma_{22}^1)X_u+(\Gamma_{11}^2+\Gamma_{22}^2)X_v
     +\sum_{\vartheta=1}^n(L_{\vartheta,11}+L_{\vartheta,22})N_\vartheta\,.$$
Note that
  $$\Gamma_{11}^1+\Gamma_{22}^1=\frac{W_u}{2W}-\frac{W_u}{2W}=0,\quad
    \Gamma_{11}^2+\Gamma_{22}^2=-\frac{W_v}{2W}+\frac{W_v}{2W}=0,$$
as well as
  $$L_{\vartheta,11}+L_{\vartheta,22}=2H_\vartheta W$$
from the definition of $H_\vartheta.$ The statement follows.
\end{proof}
\noindent
This system generalizes the classical mean-curvature-system
  $$\triangle X=2HWN$$
from Hopf \cite{hopf_1950} for $n=1$ where $H\in\mathbb R$ denotes the scalar mean curvature of $X.$
\subsection{Quadratic growth in the gradient. A maximum principle}
\label{par_maximumprinciple}
We want to give a geometric application of the classical maximum principle for subharmonic functions: Given an upper bound $|H|\le h_0$ in $\overline B$ for the conformally parametrized immersion $X,$ we infer
  $$|\triangle X|\le 2h_0W\le h_0|\nabla X|^2\quad\mbox{in}\ B$$
on account of
  $$\begin{array}{lll}
      W\negthickspace
      & = & \negthickspace\displaystyle
            \sqrt{(X_u\cdot X_u)(X_v\cdot X_v)-(X_u\cdot X_v)^2}
            \,=\,\sqrt{(X_u\cdot X_u)^2} \\[2ex]
      & = & \negthickspace\displaystyle
            |X_u||X_u|
            \,\le\,\frac{1}{2}\,\big(X_u^2+X_u^2\big)
            \,=\,\frac{1}{2}\,\big(X_u^2+X_v^2)
            \,=\,\frac{1}{2}\,|\nabla X|^2\,.
    \end{array}$$
Thus $X$ is solution of a nonlinear elliptic system with quadratic growth in the gradient. Systems of this kind will play an important role for all of our considerations.
\begin{proposition}
Let $X\in C^{3+\alpha}(\overline B,\mathbb R^{n+2})$ be an immersion with prescribed mean curvature vector $H.$ Let $|H|\le h_0$ in $\overline B,$ and suppose that $h_0\le 1.$ Then it holds\index{geometric maximum principle}
  $$\max_{(u,v)\in\overline B}|X(u,v)|^2
    =\max_{(u,v)\in\partial B}|X(u,v)|^2\,.$$
\end{proposition}
\begin{proof}
We introduce conformal parameters $(u,v)\in\overline B$ which does not affect the maximum norm of $X.$ We compute
  $$\triangle|X|^2
    =2\big(|\nabla X|^2+X\cdot\triangle X\big)
    \ge 2\big(|\nabla X|^2-h_0|\nabla X|^2\big)
    =2|\nabla X|^2(1-h_0)
    \ge 0.$$
Thus $|X|^2$ is subharmonic, and the result follows from the classical maximum principle.
\end{proof}
\noindent
For further considerations we want to refer to Dierkes \cite{dierkes_2005} and the references therein. The method of proof presented here goes already back to E. Heinz (see Sauvigny \cite{sauvigny_2005}, volume II, chapter XII).
\subsection{A curvature estimate}
\label{par_curvatureestimate}
We want to conclude this first lecture with some applications of the theory of harmonic mappings to curvature estimates for conformally parametrized minimal surfaces. Our first observations is based upon the representation formula\index{curvature estimates}
  $$S_{\sigma,12}^\omega
    =\frac{1}{W}\,(L_{\sigma,11}-L_{\sigma,22})L_{\omega,12}
     -\frac{1}{W}\,(L_{\omega,11}-L_{\omega,22})L_{\sigma,12}$$
of the normal curvature tensor from paragraph \ref{par_curvnormbundle}. Applying the Cauchy-Schwarz inequality gives
  $$|S_{\sigma,12}^\omega|
    \le\frac{1}{2W}\,(L_{\sigma,11}^2+2L_{\sigma,12}^2+L_{\sigma,22}^2)
       +\frac{1}{2W}\,(L_{\omega,11}^2+2L_{\omega,12}^2+L_{\omega,22}^2).$$
On the other hand we verify
  $$2H_\sigma^2-K_\sigma
    =\frac{L_{\sigma,11}^2+2L_{\sigma,11}L_{\sigma,22}+L_{\sigma,22}^2}{2W^2}
     -\frac{L_{\sigma,11}L_{\sigma,22}-L_{\sigma,12}^2}{W^2}
    =\frac{L_{\sigma,11}^2+2L_{\sigma,12}^2+L_{\sigma,22}^2}{2W^2}$$
so that we arrive at the
\begin{proposition}
It holds
  $$|S_{\sigma,12}^\omega|
    \le(2H_\sigma^2-K_\sigma)W+(2H_\omega^2-K_\omega)W.$$
\end{proposition}
\noindent
In particular, we conclude that spherical surfaces characterized by the property $2H_\sigma^2-K_\sigma\equiv 0$ w.r.t. arbitrary $N_\sigma$ have flat normal bundle: $S_{\sigma,12}^\omega=0.$ Furthermore, we infer that bounds for $|S_{\sigma,12}^\vartheta|$ are achieved by establishing bound for the surface curvatures and its area element $W.$\\[1ex]
We want to demonstrate how the right hand side of this inequality can be controlled by means of minimal graphs on closed discs $\overline B_r$ of radius $r>0.$ Let $X\colon\overline B_r\to\mathbb R^{n+2}$ be conformal representation of such a minimal graph.\footnote{We can make use of Riemann's mapping theorem to introduce conformal parameters $(u,v)\in\overline B_r.$} Suppose furthermore the growth condition
  $$|X(u,v)|\le\Omega r^\varepsilon$$
with a universal constants $\Omega>0,$ and later we will chose the parameter $\varepsilon\ge 0$ small enough. First of all, there hold $H_\sigma\equiv 0$ for all $\sigma=1,\ldots,n.$ Thus the mean curvature systems reduces to
  $$\triangle X=0\quad\mbox{in}\ B_r\,.$$
Applying potential theoretic estimates, see e.g. Gilbarg and Trudinger \cite{gilbarg_trudinger_1983}, Theorem 4.6, we find
\begin{lemma}
There is a real constant $C_1\in(0,+\infty)$ so that
  $$|X_{u^iu^j}(0,0)|\le C_1\|X\|_{C^0(B_r)}\,,\quad i,j=1,2.$$
\end{lemma}
\noindent
Here $\|X\|_{C^0(B_r)}$ denotes the Schauder maximum norm\index{Schauder maximum norm} of the mapping $X,$ i.e.
  $$\|X\|_{C^0(B_r)}=\sup_{(u,v)\in B_r}|X(u,v)|$$
which does not depend on the parametrization.\\[1ex]
This lemma is the first step for us to establish an upper bound for the Gaussian curvature $K_\sigma.$ Namely, we estimate as follows
  $$|K_\sigma(0,0)|
    \le\frac{|L_{\sigma,11}(0,0)||L_{\sigma,22}(0,0)|+2|L_{\sigma,12}(0,0)|^2}{W(0,0)^2}
    \le\frac{|X_{uu}(0,0)||X_{vv}(0,0)|+2|X_{uv}(0,0)|}{W(0,0)^2}$$
which leads us to
  $$|K_\sigma(0,0)|
    \le\frac{2C_1^2\|X\|_{C^0(B_r)}^2}{W(0,0)^2}
    \le\frac{2C_1^2\Omega^2}{W(0,0)^2}\,r^{2\varepsilon}\,.$$
Thus it remains to find a lower bound for the area element at the origin $(0,0)\in B_r:$ To this end we consider the plane mapping\index{plane mapping}
  $$F(u,v)=\big(x^1(u,v),x^2(u,v)\big)\colon\overline B_r\longrightarrow\mathbb R^2$$
of the graph's conformal representation $X.$ Following Sauvigny \cite{sauvigny_2005}, volume II, chapter XII, Satz 1 there is a universal constant $C_2\in(0,\infty)$ so that it holds
  $$|\nabla F(0,0)|\ge C_2r.$$
The proof of this lower estimate is very intricate. Originally it goes back to E. Heinz in 1952, and it makes essential use of ingredients characterizing our regular parameter transformation:
\begin{itemize}
\item[(i)]
$F(0,0)=(0,0);$
\vspace*{-1ex}
\item[(ii)]
$F\big|_{\partial B_r}\colon{\partial B_r}\to\partial B_r$ is a positively oriented and toplogical mapping;
\vspace*{-1ex}
\item[(iii)]
$J_F(u,v)>0$ for the Jacobian of $F.$
\end{itemize}
Especially the third property is fulfilled for conformal reparametrizations of {\it surface graphs.}\index{conformal parameters} Proving  (iii) for general self-intersecting immersions in higher dimensional spaces turns out to be difficult.\\[1ex]
Now we obtain
  $$2W(0,0)
    =|\nabla x^1(0,0)|^2+|\nabla x^2(0,0)|^2+\ldots+|\nabla x^{n+2}(0,0)|^2
    \ge|\nabla F(0,0)|^2
    \ge C_2^2r^2\,.$$
Collecting all the obtained estimates proves the following Bernstein-Liouville type result.\index{Bernstein-Liouville type result}
\begin{theorem}
(Fr\"ohlich \cite{froehlich_2005})\\
For the conformally parametrized minimal graph $X\colon\overline B_r\to\mathbb R$ it holds\index{minimal graph}
  $$|K_\sigma(0,0)|\le\frac{2C_1^2\Omega^2}{C_2^2}\,\frac{r^{2\varepsilon}}{r^4}
    \quad\mbox{for all}\ \sigma=1,\ldots,n.$$
\end{theorem}
\noindent
In particular, if $\varepsilon\in[0,2),$ and if $X$ is defined on the whole plane $\mathbb R^2$ so that we can go to the limit $r\to\infty$ due to a theorem of Hadamard (see e.g. Klingenberg \cite{klingenberg_1973}, Theorem 6.4.4), and we infer
  $$|K_\sigma(u,v)|\longrightarrow 0\quad\mbox{for}\ r\to\infty$$
which holds for all $\sigma=1,\ldots,n$ and all $(u,v)\in\mathbb R^2.$ {\it Thus the complete and entire minimal graph $X$ with growth exponent $\varepsilon\in[0,2)$ represents a plane.}\\[1ex]
Note that his result is sharp in the sense that $X(z)=(z,z^2),$ defined on the whole plane $\mathbb R^2,$ has quadratic growth with $\varepsilon=2$ and is not a plane!\\[1ex]
A curvature estimate for surface graphs of prescribed mean curvature and theorems of Bernstein type for minimal graphs can be found for example in Bergner and Fr\"ohlich \cite{bergner_froehlich_2008}. Curvature estimates resting upon methods of Schoen, Simon, Yau \cite{ssy_1975} and Ecker, Huisken \cite{ecker_2004}, \cite{ecker_huisken_1989} can be found in Wang \cite{wang_2002}, \cite{wang_2004}, Fr\"ohlich and Winklmann \cite{froehlich_winklmann_2007}, or Xin \cite{xin_2009}.
\vspace*{6ex}
\begin{center}
\rule{35ex}{0.1ex}
\end{center}
\vspace*{6ex}
\cleardoublepage\noindent
  $$ $$\\[15ex]
\thispagestyle{empty}
{\bf{\sc{\Large Lecture II}}}\\[4ex]
{\bf{\sc{\huge Normal Coulomb Frames in $\mathbb R^4$}}}\\[6ex]
\rule{108ex}{0.2ex}
\vspace*{20ex}
\begin{itemize}
\item[8.]
Problem statement. Curves in $\mathbb R^3$
\item[9.]
Torsion free normal frames
\item[10.]
Examples
\item[11.]
Normal Coulomb frames
\item[12.]
Estimating the torsion coefficients
\item[13.]
Estimates for the total torsion
\item[14.]
An example: Holomorphic graphs
\end{itemize}
\cleardoublepage\noindent
\thispagestyle{empty}
\vspace*{10ex}
\section{Problem statement. Curves in $\mathbb R^3$}
\label{section_curves}
As in paragraph \ref{par_torsionen} we consider an arc-length parametrized curve $c(s)$ in $\mathbb R^3$ with unit tangent vector $t(s)=c'(s)$ and unit normal vector $n(s)=\frac{t'(s)}{|t'(s)|}.$ The associated torsion $\tau$ is given by $\tau=n'\cdot b=-n\cdot b',$ $n$ being the unit normal and $b$ the binormal vector. We introduce a new normal frame $(\widetilde t,\widetilde n)$ by means of
  $$\widetilde n=\cos\varphi\,n+\sin\varphi\,b,\quad
    \widetilde b=-\sin\varphi\,n+\cos\varphi\,b.$$
The new torsion $\widetilde\tau$ associated to this frame then satisfies\index{torsion of curves}
  $$\begin{array}{lll}
      \widetilde\tau\negthickspace
      & = & \negthickspace\displaystyle
            \widetilde n'\cdot\widetilde b
            \,=\,(-\varphi'\sin\varphi\,n+\cos\varphi\,n'+\varphi'\cos\varphi\,b+\sin\varphi\,b')
                 \cdot(-\sin\varphi\,n+\cos\varphi\,b) \\[2ex]
      & = & \negthickspace\displaystyle
            \varphi'\sin^2\varphi-\sin^2\varphi\,(b'\cdot n)+\cos^2\varphi\,(n'\cdot b)+\varphi'\cos^2\varphi
            \,=\,\varphi'+\tau,
    \end{array}$$
so in particular constructing a normal frame $(\widetilde n,\widetilde b)$ which is {\it free of torsion,}\index{torsion free normal frame} i.e. fulfilling $\widetilde\tau\equiv 0,$ starting with a given frame $(n,b)$ reduces to solving the ordinary initial value problem
  $$\varphi'(s)=-\tau(s),\quad
    \varphi(s_0)=\tau_0$$
with some initial value $\tau_0.$
\begin{proposition}
Rotating the standard frame $(t,n)$ by an angle
  $$\varphi(s)=-\int\limits_{s_0}^s\tau(\sigma)\,d\sigma+\varphi_0$$
with arbitrary $\varphi_0\in\mathbb R$ generates a normal frame $(\widetilde t,\widetilde n)$ which is free of torsion.
\end{proposition}
\noindent
Such a torsion-free normal frame is called {\it parallel} because all derivatives of normal vectors are tangential to the curve, i.e. parallel to the tangent vector $t(s).$ This would easily follow from the Fren\'et equations for curves, but let us refer to our next considerations. Parallel normal frames are special Coulomb frames as will become clear in the following.
\vspace*{6ex}
\begin{center}
\rule{35ex}{0.1ex}
\end{center}
\vspace*{6ex}
\section{Torsion free normal frames}
\label{section_torsionfree}
The question arises whether there is a similar construction of torsion-free normal frames {\it if the underlying manifold has two dimensions.} In this chapter we focus on the case of two codimensions, i.e. we consider regular surfaces $X\colon\overline B\to\mathbb R^4.$ So let a normal frame $N=(N_1,N_2)$ be given. Then by means of the following $SO(2)$-valued transformation
  $$\widetilde N_1=\cos\varphi\,N_1+\sin\varphi\,N_2\,,\quad
    \widetilde N_2=-\sin\varphi\,N_1+\cos\varphi\,N_2$$
with a rotation angle $\varphi$ we arrive at a new normal frame $\widetilde N.$
\begin{lemma}
There hold
  $$\widetilde T_{1,1}^2=T_{1,1}^2+\varphi_u\,,\quad
    \widetilde T_{1,2}^2=T_{1,2}^2+\varphi_v\,.$$
\end{lemma}
\noindent
We omit the proof of this lemma which follows the same lines as our calulation at the beginning of this chapter. However, that special angle $\varphi$ which carries $(N_1,N_2)$ into a normal frame which is {\it free of torsion,} i.e. which satisfies
  $$\widetilde T_{1,1}^2=0
    \quad\mbox{and}\quad
    \widetilde T_{1,2}^2=0
    \quad\mbox{everywhere in}\ B,$$
can now be computed as solution of the linear system of partial differential equations
  $$\varphi_u=-T_{1,1}^2
    \quad\mbox{and}\quad
    \varphi_v=-T_{1,2}^2\,.$$
Recall that such a linear system is solvable if and only if the integrability condition
  $$0=-\varphi_{uv}+\varphi_{vu}
     =T_{1,1,v}^2-T_{1,2,u}^2
     =\mbox{div}\,(-T_{1,2}^2,T_{1,1}^2)
     =S_NW
     \quad\mbox{in}\ B$$
is satisfied with the curvature $S_N$ of the normal bundle from paragraph \ref{case_n2} and the area element $W$ of the immersion. Thus we have proved
\begin{theorem}
The immersion $X\colon\overline B\to\mathbb R^4$ admits a torsion-free normal frame $N=(N_1,N_2)$ if and only if the curvature of its normal bundle vanishes identically in $\overline B.$\index{torsion free normal frame}
\end{theorem}
\noindent
{\it This torsion-free frame is parallel} in the sense that its derivatives have no normal parts, i.e. the Weingarten equations from paragraph
\ref{par_weingarten} take the form
  $$N_{\sigma,u}=-\,\frac{L_{\sigma,11}}{W}\,X_u-\frac{L_{\sigma,12}}{W}\,X_v\,,\quad
    N_{\sigma,v}=-\,\frac{L_{\sigma,12}}{W}\,X_u-\frac{L_{\sigma,22}}{W}\,X_v$$
for $\sigma=1,2$ and using conformal parameters. This frame is not uniquely determined, rather we can rotate the whole frame by a constant angle $\varphi_0$ without effecting the torsion coefficients because the above differential equations contain only derivatives of $\varphi.$\\[1ex]
We want to remark that existence of such parallel frames in case of vanishing curvature is well settled. With the considerations here we give a new proof of this fact. But rather we intend to establish existence of regular normal frames if the normal bundle is curved, and these frames should replace parallel frames in this more general situation.
\vspace*{6ex}
\begin{center}
\rule{35ex}{0.1ex}
\end{center}
\vspace*{6ex}
\section{Examples}
\label{section_examplesn4}
\subsection{Spherical surfaces}
\label{par_spherical}
Suppose $|X(u,v)|=1$ for all $(u,v)\in\overline B.$ We immediately compute\index{spherical surface}
  $$X_u\cdot X=0,\quad X_v\cdot X=0,$$
i.e. $X$ itself is our first unit normal vector, say $X=N_1.$ A second one follows after completion of $\{X_u,X_v,N_1\}$ to a basis of the whole embedding space $\mathbb R^4.$ Then the normal frame $(N_1,N_2)$ is free of torsion because
  $$T_{1,1}^2=N_{1,u}\cdot N_2=X_u\cdot N_2=0,\quad
    T_{1,2}^2=N_{1,v}\cdot N_2=X_v\cdot N_2\,.$$
\subsection{The flat Clifford torus}
\label{par_clifford}
This surface is build up from the product (see e.g. do Carmo \cite{docarmo_1992}, chapter 6)\index{flat Clifford torus}
  $$X(u,v)=\frac{1}{\sqrt{2}}\,(\cos u,\sin u,\cos v,\sin v)\sim S^1\times S^1\,.$$
We assign the moving $4$-frame consisting of
  $$X_u=\frac{1}{\sqrt{2}}\,(-\sin u,\cos u,0,0),\quad
    X_v=\frac{1}{\sqrt{2}}\,(0,0,-\sin v,\cos v)$$
as well as
  $$N_1=\frac{1}{\sqrt{2}}\,(\cos u,\sin u,\cos v,\sin v),\quad
    N_2=\frac{1}{\sqrt{2}}\,(-\cos u,-\sin u,\cos v,\sin v).$$
This special normal frame $N=(N_1,N_2)$ is free of torsion.
\subsection{Parallel type surfaces}
\label{par_parallel}
Consider the normal transport\index{parallel type surfaces}
  $$R(u,v)=X(u,v)+f(u,v)N_1(u,v)+g(u,v)N_2(u,v).$$
If the functions $f$ and $g$ are constant then we say $R$ {\it is the parallel surface of} $X$ and vice versa, at least if the surfaces are immersed in $\mathbb R^3.$ Parallelity in higher codimensional space depends on the curvature $S_N$ of the normal bundle.
\begin{proposition}
The normal transport $R$ of an immersion $X\colon\overline B\to\mathbb R^4$ is parallel, i.e. $R_{u^i}\cdot N_\sigma=0,$ if and only if $S_N\equiv 0.$\index{normal curvature}
\end{proposition}
\begin{proof}
For the proof we use the Weingarten equations and compute the normal parts $R_u^\perp$ and $R_v^\perp$ of the tangential vectors $R_u$ resp. $R_v,$
  $$\begin{array}{lll}
      R_u^\perp\negthickspace
      & = & \negthickspace\displaystyle
            f_uN_1+g_uN_2+fN_{1,u}^\perp+gN_{2,u}^\perp
            \,=\,(f_u-gT_{1,1}^2\big)N_1+\big(g_u+fT_{1,1}^2\big)N_2\,, \\[2ex]
      R_v^\perp\negthickspace
      & = & \negthickspace\displaystyle
            f_vN_1+g_vN_2+fN_{1,v}^\perp+gN_{2,v}^\perp
            \,=\,\big(f_v-gT_{1,2}^2\big)N_1+\big(g_v+fT_{1,2}^2\big)N_2\,.
    \end{array}$$
The condition of parallelity leads us to the first order system
  $$f_u-gT_{1,1}^2=0,\quad
    f_v-gT_{1,2}^2=0,\quad
    g_u+fT_{1,1}^2=0,\quad
    g_v+fT_{1,2}^2=0.$$
We differentiate the first two equations and make use of the other two conditions to get
  $$\begin{array}{lll}
      0\negthickspace
      & = & \negthickspace\displaystyle
            f_{uv}-g_vT_{1,1}^2-gT_{1,1,v}^2
            \,=\,f_{uv}+fT_{1,1}^2T_{1,2}^2-gT_{1,1,v}^2\,, \\[2ex]
      0\negthickspace
      & = & \negthickspace\displaystyle
            f_{vu}-g_uT_{1,2}^2-gT_{1,2,u}^2
            \,=\,f_{vu}+fT_{1,1}^2T_{1,2}^2-gT_{1,2,u}^2\,,
    \end{array}$$
and a comparison of the right hand sides shows
  $$0=-gT_{1,1,v}^2+gT_{1,2,u}^2=-g\cdot S_NW.$$
Similarly we find $0=f\cdot S_NW,$ which proves the statement of the proposition.
\end{proof}
\noindent
Parallel type surface are widely used in geometry and mathematical physics. We would like to refer the reader to da Costa \cite{dacosta_1982} for an application in quantum mechanics in curved spaces.
\vspace*{6ex}
\begin{center}
\rule{35ex}{0.1ex}
\end{center}
\vspace*{6ex}
\section{Normal Coulomb frames}
\label{section_normalcoulombframes}
\subsection{The total torsion}
\label{par_totaltorsion}
But what happens if $S_N\not=0?$ The first fact we can immediately state is that {\it there is no parallel frame.} Thus it is desirable to construct frames possessing similiar features. For this purpose we make the following
\begin{definition}
The total torsion\index{total torsion} of the normal frame $N=(N_1,N_2)$ is given by
  $${\mathcal T}(N)
    =\sum_{i,j=1}^2\sum_{\sigma,\vartheta=1}^2
     \int\hspace{-1.5ex}\int\limits_{\hspace{-2.0ex}B}
     g^{ij}T_{\sigma,i}^\vartheta T_{\sigma,j}^\vartheta W\,dudv.$$
\end{definition}
\noindent
Using conformal parameters and taking the skew-symmetry of the torsion coefficients into account shows the convex character of the total torsion functional for a fixed surface
  $${\mathcal T}(N)
    =2\int\hspace{-1.5ex}\int\limits_{\hspace{-2.0ex}B}
      \Big\{
        (T_{1,1}^2)^2+(T_{1,2}^2)^2
      \Big\}\,dudv.$$
We mention {\it that the functional ${\mathcal T}(N)$ does not depend on the choice of the parametrization by its definition.}
\subsection{Definition of normal Coulomb frames}
\label{par_definitioncoulomb}
{\it Rather it depends on the choice of the normal frame.} In particular, if the immersion admits a frame parallel in the normal bundle, then ${\mathcal T}(N)=0$ for this special frame. But on the other hand the total torsion can be made arbitrarily large! So our goal is to construct normal frames which give ${\mathcal T}(N)$ the smallest possible value. Thus our next
\begin{definition}
The frame $N=(N_1,N_2)$ is called a normal Coulomb frame\index{normal Coulomb frame} if it is critical for the functional ${\mathcal T}(N)$ of total torsion w.r.t. to $SO(2)$-valued variations of the form
  $$\widetilde N_1=\cos\varphi\,N_1+\sin\varphi\,N_2\,,\quad
    \widetilde N_2=-\sin\varphi\,N_1+\cos\varphi\,N_2\,.$$
\end{definition}
\subsection{The Euler-Lagrange equation}
\label{par_eulerlagrangen4}
\label{par_euler_n4}
Because the new and the old torsion coefficients are related by
  $$\widetilde T_{1,1}^2=T_{1,1}^2+\varphi_u\,,\quad
    \widetilde T_{1,2}^2=T_{1,2}^2+\varphi_v\,,$$
we can immediately calculate the difference between the new and the old total torsion by partial integration
  $$\begin{array}{lll}
      \displaystyle
      {\mathcal T}(\widetilde N)-{\mathcal T}(N)\negthickspace
      & = & \negthickspace\displaystyle
            2\int\hspace{-1.5ex}\int\limits_{\hspace{-2.0ex}B}|\nabla\varphi|^2\,dudv
            +4\int\hspace{-1.5ex}\int\limits_{\hspace{-2.0ex}B}
              (T_{1,1}^2\varphi_u+T_{1,2}^2\varphi_v)\,dudv \\[5ex]
      & = & \negthickspace\displaystyle
            2\int\hspace{-1.5ex}\int\limits_{\hspace{-2.0ex}B}|\nabla\varphi|^2\,dudv
            +4\int\limits_{\partial B}(T_{1,1}^2,T_{1,2}^2)\cdot\nu^t\,\varphi\,ds
            -4\int\hspace{-1.5ex}\int\limits_{\hspace{-2.0ex}B}
              \mbox{div}\,(T_{1,1}^2,T_{1,2}^2)\varphi\,dudv
    \end{array}$$
where $\nu$ denotes the outer unit normal vector at the boundary $\partial B,$ and $\varphi$ is an arbitrary rotation angle. This already gives us a criterion for $N=(N_1,N_2)$ being a critical point:
\begin{proposition}
Let $\{N_1,N_2\}$ be critical for ${\mathcal T}(N).$ Then using conformal parameters $(u,v)\in\overline B,$ the torsion coefficients satisfy the first order Neumann boundary value problem
  $$\mbox{\rm div}\,(T_{1,1}^2,T_{1,2}^2)=0\quad\mbox{in}\ B,\quad
    (T_{1,1}^2,T_{1,2}^2)\cdot\nu^t=0\quad\mbox{on}\ \partial B.$$
\end{proposition}
\noindent
The conservation law structure of this Euler-Lagrange equation explains the terminology {\it normal Coulomb frame} in analogy to Coulomb gauges from physics.\index{Coulomb gauge}
\subsection{Constuction of normal Coulomb frames via a Neumann problem}
\label{par_neumannproblemn4}
How can we construct a normal Coulomb frame $N$ from a given frame $\widetilde N?$ For a critical normal frame $N$ we have to solve the boundary value problem
  $$\begin{array}{l}
      0\,=\,\mbox{div}\,(T_{1,1}^2,T_{1,2}^2)
       \,=\,\mbox{div}\,(\widetilde T_{1,1}^2-\varphi_u,\widetilde T_{1,2}^2-\varphi_v)\quad\mbox{in}\ B, \\[2ex]
      0\,=\,(T_{1,1}^2,T_{1,2}^2)\cdot\nu^t
       \,=\,(\widetilde T_{1,1}^2-\varphi_u,\widetilde T_{1,2}^2-\varphi_v)\cdot\nu^t\quad\mbox{on}\ \partial B
    \end{array}$$
by virtue of the Euler-Lagrange equation from the preceding paragraph. This implies our next result.
\begin{proposition}
The given normal frame $\widetilde N$ transforms into a normal Coulomb frame $N$ by means of our $SO(2)$-action if and only if\index{normal Coulomb frame}
  $$\begin{array}{l}
      \Delta\varphi=\mbox{\rm div}\,(\widetilde T_{1,1}^2,\widetilde T_{1,2}^2)\quad\mbox{in}\ B, \\[0.4cm]
      \displaystyle
      \frac{\partial\varphi}{\partial\nu}=(\widetilde T_{1,1}^2,\widetilde T_{1,2}^2)\cdot\nu^t
      \quad\mbox{on}\ \partial B
    \end{array}$$
holds true for the rotation angle $\varphi=\varphi(u,v)$.
\end{proposition}
\noindent
What can be said about the solvability of this Neumann boundary value problem\index{Neumann boundary value problem}? It is well known that the solvability of the Neumann problem
\begin{equation*}
    \Delta\varphi=f\quad\mbox{in}\ B,\quad
    \frac{\partial\varphi}{\partial\nu}=g\quad\mbox{on}\ \partial B
\end{equation*}
depends necessarily and sufficiently on the integrability condition
\begin{equation*}
  \int\hspace*{-1.5ex}\int\limits_{\hspace*{-2.0ex}B}f\,dudv
  =\int\limits_{\partial B}g\,ds,
\end{equation*}
which is fulfilled in our proposition! {\it Thus starting from any given normal frame $\widetilde N$ it is always possible to construct a normal Coulomb frame which is critical for the functional of total torsion.} For a general orientation on Neumann boundary value problems we want to refer to Courant and Hilbert \cite{courant_hilbert_1962}.
\subsection{Minimality property of normal Coulomb frames}
\label{par_minimalityn4}
Let $N$ be a normal Coulomb frame. From our little calculation from paragraph \ref{par_euler_n4} we infer
  $$\begin{array}{lll}
      {\mathcal T}(\widetilde N)\negthickspace
      & = & \negthickspace\displaystyle
            {\mathcal T}(N)
            +2\int\hspace{-1.5ex}\int\limits_{\hspace{-2.0ex}B}
             |\nabla\varphi|^2\,dudv
            +4\int\limits_{\partial B}(T_{1,1}^2,T_{1,2}^2)\cdot\nu\,\varphi\,ds
            -4\int\hspace{-1.5ex}\int\limits_{\hspace{-2.0ex}B}
              \mbox{div}\,(T_{1,1}^2,T_{1,2}^2)\varphi\,dudv \\[4ex]
      & = & \displaystyle\negthickspace
            {\mathcal T}(N)
            +2\int\hspace{-1.5ex}\int\limits_{\hspace{-2.0ex}B}
             |\nabla\varphi|^2\,dudv
            \,\ge\,{\mathcal T}(N)
    \end{array}$$
because the boundary integral and the integral over the divergence term vanish due to the Euler-Lagrange equation. Thus we haved proved
\begin{proposition}
A normal Coulomb frame $N$ minimizes the total torsion, i.e. we have\index{normal Coulomb frame}
  $${\mathcal T}(N)\le{\mathcal T}(\widetilde N)$$
for all normal frames $\widetilde N$ resulting from our $SO(2)$-action. Equality occurs if and only if $\varphi\equiv\mbox{\rm const}.$
\end{proposition}
\vspace*{6ex}
\begin{center}
\rule{35ex}{0.1ex}
\end{center}
\vspace*{6ex}
\section{Estimating the torsion coefficients}
\label{section_estimatingtorsion}
\subsection{Reduction for flat normal bundles}
\label{par_reduction}
Let $(u,v)\in\overline B$ be conformal parameters. We want to study critical normal frames in case of flat normal bundles\index{flat normal bundle}
  $$WS_N=T_{1,1,v}^2-T_{1,2,u}^2=\mbox{div}\,(-T_{1,2}^2,T_{1,1}^2)\equiv 0.$$
If $N$ is a normal Coulomb frame then the Neumann boundary condition $(T_{1,1}^2,T_{1,2}^2)\cdot\nu=0$ of the Euler-Lagrange equation says that the vector field $(-T_{1,2}^2,T_{1,1}^2)$ is parallel to the outer unit normal vector $\nu$ along $\partial B$. Thus a partial integration yields
  $$\int\hspace*{-1.5ex}\int\limits_{\hspace*{-2.0ex}B}
    S_NW\,dudv
    =\int\limits_{\partial B}(-T_{1,2}^2,T_{1,1}^2)\cdot\nu^t\,ds
    =\pm\int\limits_{\partial B}\sqrt{(T_{1,1}^2)^2+(T_{1,2}^2)^2}\,ds.$$
In particular, if $S_N\equiv 0,$ i.e. if the normal bundle is flat, then we find
\begin{equation*}
  T_{\sigma,i}^\vartheta\equiv 0\quad\mbox{on}\ \partial B
\end{equation*}
for all $i=1,2$ and $\sigma,\vartheta=1,2.$ On the other hand, differentiating $0=T_{1,1,v}^2-T_{1,2,u}^2$ w.r.t. $u$ and $v,$ and taking the Euler-Lagrange equation $T_{1,1,u}^2+T_{1,2,v}^2=0$ into account, gives us
  $$\Delta T_{1,1}^2=0\,,\quad
    \Delta T_{1,2}^2=0.$$
Thus $T_{1,1}^2$ and $T_{1,2}^2$ are harmonic functions, and the maximum principle implies
\begin{equation*}
  T_{\sigma,i}^\vartheta\equiv 0\quad\mbox{in}\ B
  \quad\mbox{for all}\ i=1,2\ \mbox{and}\ \sigma,\vartheta=1,2.
\end{equation*}
\begin{theorem}
A normal Coulomb frame $N$\index{normal Coulomb frame} of an immersion $X\colon\overline B\to\mathbb R^4$ with flat normal bundle is free of torsion, i.e. it is parallel.\index{flat normal bundle}
\end{theorem}
\noindent
To summarize the preceding considerations we have proved existence (and simultaneously regularity) of normal frames critical for the functional of total torsion. If additionally the curvature $S_N$ of the normal bundle vanishes, then such a critical frame is free of torsion.
\subsection{Estimates via the maximum principle}
\label{par_maxprinciple}
Next we want to consider the case of {\it non-flat} normal bundles. From the Euler-Lagrange equation we know that the torsion vector $(T_{1,1}^2,T_{1,2}^2)$ of a normal Coulomb frame is divergence-free. Thus the differential $1$-form
  $$\omega:=-T_{1,2}^2\,du+T_{1,1}^2\,dv$$
is closed, i.e. for its outer derivative we calculate
  $$d\omega
    =T_{1,1,u}^2\,du\wedge dv+T_{1,2,v}^2\,du\wedge dv
    =\mbox{div}\,(T_{1,1}^2,T_{1,2}^2)\,du\wedge dv
    =0,$$
and Poincar\'e's lemma\index{Poincar\'e's lemma} ensures the existence of a $C^2$-regular function $\tau$ satisfying (see e.g. Sauvigny \cite{sauvigny_2005}, volume I, chapter I, \S 7)\index{integral function}
  $$d\tau=\tau_u\,du+\tau_v\,dv
         =\omega
    \quad\mbox{which finally implies}\ \nabla\tau=(-T_{1,2}^2,T_{1,1}^2).$$
A second differentiation, taking account of the representation $S_NW=\mbox{div}\,(-T_{1,2}^2,T_{1,1}^2),$ leads us to the inhomogeneous boundary value problem
  $$\triangle\tau=S_NW\quad\mbox{in}\ B,\quad
    \tau=0\quad\mbox{on}\ \partial B.$$
To justify the homogeneous boundary condition note that $\nabla\tau\cdot(-v,u)^t=0$ on $\partial B$ for normal Coulomb frames because $(T_{1,1}^2,T_{1,2}^2)$ is perpendicular to $\partial B.$ Therefore it holds $\tau=\mbox{const}$ along $\partial B.$ But $\tau$ is only defined up to a constant of integration which can be chosen such that the homogeneous boundary condition is satisfied! Thus Poisson's representation formula for the solution $\tau$ yields
  $$\tau(w)
    =\int\hspace*{-1.5ex}\int\limits_{\hspace*{-2.0ex}B}
     \Phi(\zeta;w)S_N(\zeta)W(\zeta)\,d\xi d\eta,\quad
    \zeta=(\xi,\eta)\in B,$$
with the non-positive Green kernel\index{Green kernel}
  $$\Phi(\zeta;w):=\frac{1}{2\pi}\,\log\left|\frac{\zeta-w}{1-\overline w\zeta}\right|,\quad\zeta\not=w,$$
of the Laplace operator $\triangle$ (see e.g. Sauvigny \cite{sauvigny_2005}, volume II, chapter VIII, \S 1). We want to give an estimate for this kernel to establish an estimate for the integral function $\tau:$ For this purpose note that
  $$\psi(w)=\frac{|w|^2-1}{4}
    \quad{\rm solves}\quad\triangle\psi=1\ \mbox{\rm in}\ B
    \ \mbox{\rm and}\ \psi=0\ \mbox{on}\ \partial B.$$
Therefore we conclude
  $$\int\hspace*{-1.5ex}\int\limits_{\hspace*{-2.0ex}B}
    |\Phi(\zeta;w)|\cdot 1\,d\xi d\eta
    =\frac{1-|w|^2}{4}
    \le\frac{1}{4}\,.$$
Thus we arrive at
\begin{lemma}
Let $(u,v)\in\overline B$ be conformal parameters. The integral function $\tau$ for the Poisson problem for the curvature $S_N$ of the normal bundle satisfies
  $$|\tau(w)|\le\frac{1}{4}\,\|S_NW\|_{C^0(\overline B)}\quad\mbox{in}\ \overline B\,,\quad
    \tau(w)=0\quad\mbox{on}\ \partial B.$$
\end{lemma}
\noindent
Thus potential theoretic estimates for the Laplacian (see e.g. Sauvigny \cite{sauvigny_2005}, chapter IX, \S 4, Satz 1) ensure the existence of a real constant $C=C(\alpha)$ such that
  $$\|\tau\|_{C^{2+\alpha}(\overline B)}\le C(\alpha)\|S_NW\|_{C^\alpha(\overline B)}$$
holds true for all $\alpha\in(0,1).$ Finally, this $C^{2+\alpha}$-bound provides simultaneously an upper bound for the $C^{1+\alpha}$-norm of the torsion coefficients. We have proved the main result of the present chapter:
\begin{theorem}
\label{estimate_r4_a}
Let the conformally parametrized immersion $X\colon\overline B\to\mathbb R^4$ with normal bundle of curvature $S_N$ be given. Then there exists a normal Coulomb frame $N$ minimizing the functional of total torsion, and whose torsion coefficients satisfy
  $$\|T_{\sigma,i}^\vartheta\|_{C^{1+\alpha}(\overline B)}\le C(\alpha)\|S_NW\|_{C^\alpha(\overline B)}$$
for all $\alpha\in(0,1)$ with the constant $C(\alpha)$ from above.
\end{theorem}
\noindent
In particular, for flat normal bundles with $S_N\equiv 0$ we recover the characterization of minimizing normal frames from paragraph \ref{par_reduction}: Normal Coulomb frames for flat normal bundles are free of torsion.
\subsection{Estimates via a Cauchy-Riemann boundary value problem}
\label{par_cauchyriemann}
Once again, let us consider the Euler-Lagrange equation together with the representation formula for the curvature $S_N$ of the normal bundle, i.e.
  $$\frac{\partial}{\partial u}\,T_{1,1}^2+\frac{\partial}{\partial v}\,T_{1,2}^2=0,\quad
    \frac{\partial}{\partial v}\,T_{1,1}^2-\frac{\partial}{\partial u}\,T_{1,2}^2=S_NW.$$
We want to present briefly a second method to control the torsion coefficients of a normal Coulomb frame which is strongly adapted to the case of two codimensions.
\begin{lemma}
Let $N$ be a normal Coulomb frame. Then the complex-valued torsion
  $$\Psi:=T_{1,1}^2-iT_{1,2}^2\in{\mathbb C}$$
solves the inhomogeneous Cauchy-Riemann equation\index{Cauchy-Riemann equations}
  $$\Psi_{\overline w}=\frac{i}{2}\,S_NW
    \quad\mbox{in}\ B,\quad
    \mbox{\rm Re}\big[w\Psi(w)\big]=0\quad\mbox{for}\ w\in\partial B$$
using conformal parameters $(u,v)\in\overline B.$
\end{lemma}
\begin{proof}
We compute
  $$2\Psi_{\overline w}
    =\Psi_u+i\Psi_v
    =(T_{1,1,u}^2+T_{1,2,v}^2)+i(T_{1,1,v}^2-T_{1,2,u}^2)
    =0+iS_NW\in\mathbb C$$
as well as
  $$\mbox{Re}\,\big[w\Psi(w)\big]
    =\mbox{Re}\,\big[(u+iv)(T_{1,1}^2-iT_{1,2}^2)\big]
    =uT_{1,1}^2+vT_{1,2}^2
    =(-T_{1,2}^2,T_{1,1}^2)\cdot(-v,u)
    =\nabla\tau\cdot(-v,u)
    =0$$
with the integral function $\tau$ from the previous paragraph.
\end{proof}
\noindent
A relation of this form is called a {\it linear Riemann-Hilbert problem}\index{Riemann-Hilbert problem} for the complex-valued function $\Psi.$ There is a huge complex analysis machinery to attack such a mathematical problem! In particular, we want to derive an integral representation for $\Psi.$
\begin{lemma}
The above Riemann-Hilbert problem for the complex-valued torsion vector $\Psi$ of a normal Coulomb frame possesses at most one solution $\Psi\in C^1(B)\cap C^0(\overline B)$.
\end{lemma}
\begin{proof}
Let $\Psi_1,\Psi_2$ be two such solutions. Then we set $\Phi(w):=w[\Psi_1(w)-\Psi_2(w)]$ and compute
  $$\Phi_{\overline w}=0\quad\mbox{in}\ B,\qquad \mbox{Re}\,\Phi=0\quad\mbox{on}\ \partial B.$$
Thus the real part of the holomorphic function $\Phi$ vanishes on the set $\partial B,$ and the Cauchy-Riemann equations yield $\Phi\equiv ic$ in $B$ with some constant $c\in\mathbb R.$ The continuity of $\Psi_1$ and $\Psi_2$ implies $c=0.$
\end{proof}
\noindent
Now our Riemann-Hilbert problem can be solved by means of so-called {\it generalized analytic functions.} We want to present some basic facts about this important class of complex-valued functions (see e.g. Sauvigny \cite{sauvigny_2005} or the monograph Vekua \cite{vekua_1963}). For arbitrary $f\in C^1(B,\mathbb C)$ we define {\it Cauchy's integral operator}\index{Cauchy integral operator} by
  $$T_B[f](w)
    :=-\frac1\pi
       \int\hspace*{-1.3ex}\int\limits_{\hspace*{-1.8ex}B}
       \frac{f(\zeta)}{\zeta-w}\,d\xi d\eta,\quad w\in\mathbb C.$$
\begin{lemma}
\label{lemma_ableitung}
There hold $T_B[f]\in C^1(\mathbb C\setminus\partial B)\cap C^0(\mathbb C)$ as well as
  $$\frac\partial{\partial\overline w}\,T_B[f](w)=\left\{
    \begin{array}{ll}
      f(w),& w\in B\\[1ex]
      0,& w\in\mathbb C\setminus\overline B
    \end{array}\right..$$
\end{lemma}
\begin{proof}
For a detailed proof see Vekua \cite{vekua_1963}, chapter I, \S 5. We want to verify the complex derivative: Let $\{G_k\}_{k=1,2,\ldots}$ be a sequence of open, simply connected and smoothly bounded domains contracting to some point $z_0\in B$ for $k\to\infty.$ Let $|G_k|$ denote its area. We compute (see Sauvigny \cite{sauvigny_2005}, chapter IV, \S 5)
  $$\begin{array}{lll}
      \displaystyle
      \frac{1}{2i|G_k|}\,\int\limits_{\partial G_k}T_B[f](w)\,dw\negthickspace
      & = & \negthickspace\displaystyle
            \frac{1}{2i|G_k|}\,
            \int\limits_{\partial G_k}
            \left(
              -\frac{1}{\pi}
               \int\hspace{-1.3ex}\int\limits_{\hspace*{-1.8ex}B}
               \frac{f(\zeta)}{\zeta-w}\,d\xi d\eta
            \right)dw \\[5ex]
      & = & \negthickspace\displaystyle
            \frac{1}{2\pi i|G_k|}\,
            \int\hspace*{-1.3ex}\int\limits_{\hspace{-1.8ex}B}
            \left(
              f(\zeta)
              \int\limits_{\partial G_k}
              \frac{1}{w-\zeta}\,dw
            \right)d\xi d\eta \\[5ex]
      & = & \negthickspace\displaystyle
            \frac{1}{2\pi i|G_k|}\,
            \int\hspace*{-1.3ex}\int\limits_{\hspace{-1.8ex}B}
            f(\zeta)\cdot 2\pi i\chi_{G_k}(\zeta)\,d\xi d\eta \\[5ex]
      & = & \negthickspace\displaystyle
            \frac{1}{|G_k|}\,
            \int\hspace*{-1.3ex}\int\limits_{\hspace*{-1.8ex}G_k}
            f(\zeta)\,d\xi d\eta
    \end{array}$$
with the characteristic function $\chi.$ Here we have used Cauchy's formula
  $$\int\limits_{G_k}\frac{g(w)}{w-\zeta}\,dw=2\pi ig(\zeta)$$
for a holomorphic function $g$ with $\zeta\in G_k.$ Now recalling the integration by parts rule in complex form
  $$\int\hspace{-1.3ex}\int\limits_{\hspace{-1.8ex}G_k}
    \frac{d}{d\overline w}\,f(w)\,d\xi d\eta
    =\frac{1}{2i}\,\int\limits_{\partial G_k}f(z)\,dz$$
we get in the limit
  $$\frac{d}{d\overline w}\,T_B[f](w)
    =\lim_{k\to\infty}
     \frac{1}{2i|G_k|}\,
     \int\limits_{\partial G_k}T_B[f](w)\,dw
    =f(z_0)$$
for all $z_0\in B.$ The statement follows.
\end{proof}
\noindent
Next we set
  $$P_B[f](w)
    :=-\frac1\pi
       \int\hspace*{-1.3ex}\int\limits_{\hspace*{-1.8ex}B}
       \left\{
         \frac{f(\zeta)}{\zeta-w}+\frac{\overline\zeta\,
         \overline{f(\zeta)}}{1-w\overline\zeta}
       \right\}d\xi\,d\eta
    =T_B[f](w)-\frac1w\,\overline{T_B[wf]\Big(\frac1{\overline w}\Big)}\,.$$
Now Satz 1.24 in Vekua \cite{vekua_1963} states the following
\begin{lemma}
With the definitions above, we have the uniform estimate
  $$\big|P_B[f](w)\big|
    \le C(p)\|f\|_{L^p(B)},\quad w\in\overline B,$$
where $p\in(2,+\infty],$ and $C(p)$ is a positive constant dependent on $p$.
\end{lemma}
\noindent
Using this result (which remains unproved here) we obtain our second torsion esimate in terms of $L^p$-norms, the main result of this paragraph.
\begin{theorem}
\label{estimate_psi}
Let the conformally parametrized immersion $X\colon\overline B\to\mathbb R^4$ be given. Then the complex-valued torsion $\Psi$ of a normal Coulomb frames $(N_1,N_2)$ satisfies\index{normal Coulomb frame}
  $$|\Psi(w)|\le c(p)\|S_NW\|_{L^p(B)}\quad\mbox{for all}\ w\in B$$
with some positive constant $c(p)$ and $p\in(2,+\infty].$
\end{theorem}
\begin{remark}
For a flat normal bundle\index{flat normal bundle} with $S_N\equiv 0$ we verify our results from paragraph \ref{par_reduction}.
\end{remark}
\begin{proof}
Let us write $f:=\frac{i}{2}\,S_NW\in C^1(\overline B)$ to apply the previous results. We claim that the complex-valued torsion $\Psi$ possesses the integral representation 
\begin{equation*}
  \Psi(w)
  =P_B[f](w)
  =-\,\frac{1}{\pi}\,
      \int\hspace{-1.3ex}\int\limits\limits_{\hspace{-1.8ex}B}
      \left\{
        \frac{f(\zeta)}{\zeta-w}+\frac{\overline\zeta\,\overline{f(\zeta)}}{1-w\overline\zeta}	
      \right\}d\xi\,d\eta,
  \quad w\in B.
\end{equation*}
Then the stated estimate follows at once from the above lemma. First we claim
  $$wP_B[f](w)
    =\frac{1}{\pi}\,
     \int\hspace{-1.3ex}\int\limits_{\hspace{-1.8ex}B}
     f(\zeta)\,d\xi\,d\eta+T_B[wf](w)-\overline{T_B[wf]\Big(\frac 1{\overline w}\Big)}\,.$$
Let us check this identity:
  $$\begin{array}{l}
      \displaystyle
      \frac{1}{\pi}\,
      \int\hspace*{-1.3ex}\int\limits_{\hspace*{-1.8ex}B}
      f(\zeta)\,d\xi d\eta
      +T_B[wf](w)-\overline{T_B[wf](\overline w^{\,-1})} \\[5ex]
      \hspace*{6ex}\displaystyle
      =\,\frac{1}{\pi}\,
         \int\hspace*{-1.3ex}\int\limits_{\hspace{-1.8ex}B}
         f(\zeta)\,d\xi d\eta
         -\frac{1}{\pi}\,
          \int\hspace*{-1.3ex}\int\limits_{\hspace{-1.8ex}B}
          \frac{\zeta f(\zeta)}{\zeta-w}\,d\xi d\eta
         -\overline{T_B[wf](\overline w^{\,-1})} \\[5ex]
      \hspace*{6ex}\displaystyle
      =\,-\,\frac{w}{\pi}\,
            \int\hspace*{-1.3ex}\int\limits_{\hspace*{-1.8ex}B}
            \frac{f(\zeta)}{\zeta-w}\,d\xi d\eta
         +\frac{1}{\pi}\,
          \int\hspace*{-1.3ex}\int\limits_{\hspace{-1.8ex}B}
          \frac{\overline\zeta\,\overline{f(\zeta)}}{\overline\zeta-\frac{1}{w}}\,d\xi d\eta \\[5ex]
      \hspace*{6ex}\displaystyle
      =\,-\,\frac{w}{\pi}\,
            \int\hspace*{-1.3ex}\int\limits_{\hspace*{-1.8ex}B}
            \frac{f(\zeta)}{\zeta-w}\,d\xi d\eta
         +\frac{w}{\pi}\,
          \int\hspace*{-1.3ex}\int\limits_{\hspace{-1.8ex}B}
          \frac{\overline\zeta\,\overline{f(\zeta)}}{\overline\zeta\,w-1}\,d\xi d\eta \\[5ex]
      \hspace*{6ex}\displaystyle
      =\,-\,\frac{w}{\pi}\,
            \int\hspace{-1.3ex}\int\limits_{\hspace{-1.8ex}B}
            \left(
              \frac{f(\zeta)}{\zeta-w}
              +\frac{\overline\zeta\,\overline{f(\zeta)}}{1-\overline\zeta\,w}
            \right)d\xi d\eta$$
    \end{array}$$
which shows this statement. Next, taking $f=\frac{i}{2}\,S_NW$ into account, we infer
  $$\mbox{Re}\,\big\{wP_B[f](w)\big\}=0,\quad w\in\partial B,$$
what follows from
  $$\begin{array}{l}
      T_B[{\textstyle\frac{1}{2}}iwS_NW](w)
      -\overline{T_B[{\textstyle\frac{1}{2}}iwS_NW](\overline w^{\,-1})\,} \\[3ex]
      \hspace*{6ex}\displaystyle
      =\,-\,\frac{1}{\pi}\,
            \int\hspace*{-1.3ex}\int\limits_{\hspace{-1.8ex}B}
            \frac{i}{2}\,\frac{\zeta S_NW}{\zeta-w}\,d\xi d\eta
           +\frac{1}{\pi}\,
            \overline{
              \int\hspace*{-1.3ex}\int\limits_{\hspace{-1.8ex}B}
              \frac{i}{2}\,\frac{\zeta S_NW}{\zeta-\overline w^{\,-1}}}
      \,=\,-\,\frac{i}{2\pi}\,
              \int\hspace{-1.3ex}\int\limits_{\hspace{-1.8ex}B}
              \left(
                \frac{\zeta}{\zeta-w}+\frac{\overline\zeta}{\overline\zeta-\frac{1}{w}}
              \right)
            S_NW\,d\xi d\eta.
    \end{array}$$
The entry in the brackets is a real number because $\frac{1}{w}=\frac{\overline w}{|w|^2}=\overline w$ holds true on the boundary $\partial B.$ Investing additionally
  $$\frac{\partial}{\partial\overline w}\,P_B[f](w)=f(w),$$
which follows from Lemma \ref{lemma_ableitung} and our representation of $P_B[f](w),$ we conclude that $P_B[f](w)$ solves the Riemann-Hilbert problem for $\Psi.$ The above uniqueness result for the Riemann-Hilbert problem proves the stated representation.
\end{proof}
\noindent
Note that our proof relies crucially on the fact $f=\frac{i}{2}\,S_NW$  is {\it purely imaginary!}
\vspace*{6ex}
\begin{center}
\rule{35ex}{0.1ex}
\end{center}
\vspace*{6ex}
\section{Estimates for the total torsion}
\label{section_estimatestotaltorsion}
The previous results allow us to establish immediately lower and upper bounds for the total torsion ${\mathcal T}(N)$ for normal Coulomb frames $N.$ In particular, Theorems \ref{estimate_r4_a} and \ref{estimate_psi} show\index{total torsion}
\begin{theorem}
Let the conformally parametrized immersion $X\colon\overline B\to\mathbb R^4$ with a normal Coulomb frame $N$ be given. Then there hold
  $${\mathcal T}(N)\le C(\alpha)^2\|S_NW\|_{C^\alpha(\overline B)}^2$$
for all $\alpha\in(0,1)$ with the real constant $C=C(\alpha)$ from Theorem \ref{estimate_r4_a}, as well as
  $${\mathcal T}(N)\le c(p)^2\|S_NW\|_{L^p(B)}^2$$
for all $p\in(2,+\infty]$ with the real constant $c=c(p)$ from Theorem \ref{estimate_psi}.
\end{theorem}
\noindent
The following lower bound for the total torsion of normal Coulomb frames $N$ is a special case of a general estimate which we will prove later when we consider the case of higher codimension.
\begin{theorem}
Let the conformally parametrized immersion $X\colon\overline B\to\mathbb R^4$ with a normal Coulomb frame $N$ be given. Assume $S_N\not\equiv 0$ for the curvature of its normal bundle. Then it holds
  $${\mathcal T}(N)
    \ge\left(
         \frac{\|S\|_{2}^{2}}{2(1-\varrho)^2\|S\|_{2,\varrho}^2}
         +\frac{\|\nabla S\|_2^{2}}{\|S\|_{2,\varrho}^2}
       \right)^{-1}\|S\|_{2,\varrho}^2>0$$
where $\varrho=\varrho(S)\in(0,1)$ is chosen such that
  $$\int\hspace{-1.5ex}\int\limits_{\hspace{-2.2ex}B_\varrho(0)}
    |S_NW|^2\,dudv>0.$$
\end{theorem}
\begin{remark}
In Lecture III we catch up on estimating the area element $W$ what is still left to complete our investigations so far.
\end{remark}
\vspace*{6ex}
\begin{center}
\rule{35ex}{0.1ex}
\end{center}
\vspace*{6ex}
\section{An example: Holomorphic graphs}
\label{section_holomorphicgraphs}
We consider again graphs $X(w)=(w,\Phi(w)),$ $w=u+iv\in B,$ with a holomorphic function $\Phi=\varphi+i\psi.$ We compute\index{holomorphic graph}
  $$g_{11}=1+|\nabla\varphi|^2=g_{22}\,,\quad
    g_{12}=0$$
due to the Cauchy-Riemann equations $\varphi_u=\psi_v$, $\varphi_v=-\psi_u.$ In particular, there hold $\Delta\varphi=\Delta\psi=0,$ i.e. the graph $X=X(u,v)$ is conformally parametrized and represents a minimal graph in $\mathbb R^4:$
  $$\Delta X(u,v)=0\quad\mbox{in}\ B.$$
The area element $W$ of $X$ equals $W=1+|\nabla\varphi|^2=1+|\nabla\psi|^2.$ Its Euler unit normal vectors read as
  $$N_1=\frac{1}{\sqrt{W}}\,(-\varphi_u,-\varphi_v,1,0),\quad
    N_2=\frac{1}{\sqrt{W}}\,(-\psi_u,-\psi_v,0,1).$$
For the associated torsion coefficients we calculate
  $$T_{1,1}^2
    =\frac{1}{W}\,(-\varphi_{uu}\varphi_v+\varphi_{uv}\varphi_u)
    =\frac1{2W}\frac{\partial}{\partial v}\,(|\nabla\varphi|^2)\,,\quad
    T_{1,2}^2
    =-\,\frac{1}{2W}\,\frac{\partial}{\partial u}\,(|\nabla\varphi|^2)\,.$$
Consequently, due to the special form of $W$ we infer
  $$\begin{array}{lll}
      \displaystyle
      \mbox{div}\,(T_{1,1}^2,T_{1,2}^2)\negthickspace
      & = & \negthickspace\displaystyle
            \frac{\partial}{\partial u}\left(\frac{1}{2W}\right)\frac{\partial}{\partial v}\,|\nabla\varphi|^2
            -\frac{\partial}{\partial v}\left(\frac{1}{2W}\right)\frac{\partial}{\partial u}\,|\nabla\varphi|^2
            +\frac{1}{2W}\,\frac{\partial^2}{\partial v\partial u}\,|\nabla\varphi|^2
            -\frac{1}{2W}\,\frac{\partial^2}{\partial u\partial v}\,|\nabla\varphi|^2 \\[4ex]
      & = & \negthickspace\displaystyle
            \frac{1}{2}\,\frac{\partial}{\partial u}\,|\nabla\varphi|^2\,
            \frac{\partial}{\partial v}\,|\nabla\varphi|^2
            -\frac{1}{2}\,\frac{\partial}{\partial v}\,|\nabla\varphi|^2\,
             \frac{\partial}{\partial u}\,|\nabla\varphi|^2
            \,=\,0.
    \end{array}$$
Thus the Euler-Lagrange equation for a normal Coulomb frame is satisfied! In order to check the boundary condition of the Euler-Lagrange equation for the total torsion we introduce polar coordinates $u=r\cos\alpha,$ $v=r\sin\alpha.$ Note that $\frac{1}{r}\,\frac{\partial}{\partial\alpha}=u\,\frac{\partial}{\partial v}-v\,\frac{\partial}{\partial u}.$ In our case we obtain
  $$(T_{1,1}^2,T_{1,2}^2)\cdot\nu^t
    =\frac{1}{2W}
     \left(
       u\,\frac{\partial}{\partial v}-v\,\frac{\partial}{\partial u}
     \right)
     |\nabla\varphi|^2
    =\frac{1}{2W}\,\frac{\partial}{\partial\alpha}\,|\Phi_w|^2\quad\mbox{on}\ \partial B$$
with the complex derivative $\Phi_w=\frac{1}{2}\,(\Phi_u+i\Phi_v)\in\mathbb C.$ We infer that the necessary boundary condition is satisfied if and only if $\frac{\partial}{\partial\alpha}\,|\Phi_w|$ vanishes at the boundary curve $\partial B.$ Examples of minimal graph satisfying this special property are
  $$X(w)=(w,w^n)\quad\mbox{with}\ n\in\mathbb N.$$
We have proved
\begin{proposition}
Let the conformally parametrized minimal graph $(w,\Phi(w))$ with a holomorphic function $\Phi=\varphi+i\psi$ be given. Then its Euler unit normals $N_1$ and $N_2$ form a normal Coulomb frame $N$ if $\Phi_w$ does not depend on the angle $\alpha$ along the boundary curve $\partial B.$\index{minimal graph}
\end{proposition}
\vspace*{6ex}
\begin{center}
\rule{35ex}{0.1ex}
\end{center}
\vspace*{6ex}
\cleardoublepage\noindent
  $$ $$\\[15ex]
\thispagestyle{empty}
{\bf{\sc{\Large Lecture III}}}\\[4ex]
{\bf{\sc{\huge Normal Coulomb Frames in $\mathbb R^{n+2}$}}}\\[6ex]
\rule{108ex}{0.2ex}
\vspace*{20ex}
\begin{itemize}
\item[15.]
Problem statement
\item[16.]
The Euler-Lagrange equations
\item[17.]
Examples
\item[18.]
Quadratic growth in the gradient
\item[19.]
Torsion free normal frames
\item[20.]
Non-flat normal bundles
\item[21.]
Bounds for the total torsion
\item[22.]
Existence and regularity of weak normal Coulomb frames
\item[23.]
Classical regularity of normal Coulomb frames
\item[24.]
Estimates for the area element $W$
\end{itemize}
\cleardoublepage\noindent
\thispagestyle{empty}
\vspace*{10ex}
\section{Problem formulation}
\label{section_problembeliebig}
In this lecture we want to generalize the previous considerations to the case of higher codimensions $n>2.$ We start with computing the Euler-Lagrange equations for the functional of total torsion\index{total torsion}
  $${\mathcal T}(N)
    =\sum_{i,j=1}^2\sum_{\sigma,\vartheta=1}^n
     g^{ij}T_{\sigma,i}^\vartheta T_{\sigma,j}^\vartheta W\,dudv$$
for a normal frame $N=(N_1,\ldots,N_n).$ These equations form a nonlinear elliptic system with quadratic growth in the gradient. We derive analytical and geometric properties of normal Coulomb frames, and we prove existence and regularity of parallel frames in case of vanishing curvature of the normal bundle as well as critical points of ${\mathcal T}(N)$ in the general situation of nonflat normal bundle.
\vspace*{6ex}
\begin{center}
\rule{35ex}{0.1ex}
\end{center}
\vspace*{6ex}
\section{The Euler-Lagrange equations}
\label{section_eulerlagrangeallgemein}
\subsection{Definition of normal Coulomb frames}
\label{par_definitionallgemein}
In this section we derive the Euler-Lagrange equations for critical normal frames $N.$ We should point out that due to do Carmo \cite{docarmo_1992}, chapter 3, section 2 we can construct a family $R(w,\varepsilon)$ of rotations from $SO(n),$ as considered throughout our first lecture, for arbitrary given skew-symmetric matrix $A(w)=(A_\sigma^\vartheta(w))_{\sigma,\vartheta=1,\ldots,n}\in C^\infty(B,{\rm so}(n))$ by means of the geodesic flow\index{geodesic flow} in the manifold $SO(n).$\\[1ex]
In terms of such rotations we consider variations $\widetilde N=(\widetilde N_1,\ldots,\widetilde N_n)$ of a given normal frame $N=(N_1,\ldots,N_n)$ by means of
  $$\widetilde N_\sigma(w,\varepsilon)
    :=\sum_{\vartheta=1}^n
      R_\sigma^\vartheta(w,\varepsilon)N_\vartheta(w),
    \quad\sigma=1,\ldots,n,$$
with a one-parameter family of rotations
  $$R(w,\varepsilon)
    =\big(R_\sigma^\vartheta(w,\varepsilon)\big)_{\sigma,\vartheta=1,\ldots,n}
    \in C^\infty(B\times(-\varepsilon_0,+\varepsilon_0),{\rm SO}(n)),$$
with sufficiently small $\varepsilon_0>0,$ such that
  $$R(w,0)={\mathbb E}^n\,,\quad
    \frac{\partial}{\partial\varepsilon}\,R(w,0)=A(w)\in C^\infty(B,{\rm so}(n))$$
is true with the $n$-dimensional unit matrix ${\mathbb E}^n.$ Such a matrix $A$ is the essential ingredient for defining the first variation of the functional of total torsion.
\begin{definition}
A normal frame $N$ is called {\it critical for the total torsion} or shortly a {\it normal Coulomb frame} if and only if the first variation\index{normal Coulomb frame}
  $$\delta{\mathcal T}(N,A)
    :=\lim_{\varepsilon\to 0}
      \frac{1}{\varepsilon}\,
      \big\{\mathcal T(\widetilde N)-\mathcal T(N)\big\}$$
vanishes w.r.t. all skew-symmetric perturbations $A(w)=(A_\sigma^\vartheta(w))_{\sigma,\vartheta=1,\ldots,n}\in C^\infty(B,{\rm so}(n)).$
\end{definition}
\subsection{The first variation}
\label{par_firstvariation}
Now we present the Euler-Lagrange equations for normal Coulomb frames.
\begin{proposition}
The normal frame $N$ is a normal Coulomb frame if and only if its torsion coefficients solve the Neumann boundary value problems
  $$\mbox{\rm div}\,(T_{\sigma,1}^\vartheta,T_{\sigma,2}^\vartheta)=0\quad\mbox{in}\ B,\quad
    (T_{\sigma,1}^\vartheta,T_{\sigma,2}^\vartheta)\cdot\nu=0\quad\mbox{on}\ \partial B$$
for all $\sigma,\vartheta=1,\ldots,n$ with $\nu$ being the outer unit normal vector along the boundary $\partial B.$
\end{proposition}
\noindent
This system of conservation laws is again the origin of the terminology ``Coulomb frame.''\index{conservation law}\index{Coulomb frame}
\begin{proof}
We consider the one-parameter family of rotations $R(w,\varepsilon)=(R_\sigma^\vartheta(w,\varepsilon))_{\sigma,\vartheta=1,\ldots,n}$ from above. Expanding around $\varepsilon=0$ yields
  $$R(w,\varepsilon)=\mathbb E^n+\varepsilon A(w)+o(\varepsilon).$$
Now we apply the rotation $R=(R_\sigma^\vartheta)_{\sigma,\vartheta=1,\ldots,n}$ to the given normal frame $N.$ The new normal vectors $\widetilde N_1,\ldots,\widetilde N_n$ are then given by
  $$\widetilde N_\sigma
    =\sum_{\vartheta=1}^n
     R_\sigma^\vartheta N_\vartheta
    =\sum_{\vartheta=1}^n
     \big\{
       \delta_\sigma^\vartheta
       +\varepsilon A_\sigma^\vartheta
       +o(\varepsilon)
     \big\}\,N_\vartheta
    =N_\sigma
     +\varepsilon
      \sum_{\vartheta=1}^n
      A_\sigma^\vartheta N_\vartheta
     +o(\varepsilon),$$
and we compute
  $$\widetilde N_{\sigma,u^\ell}
    =N_{\sigma,u^\ell}
     +\varepsilon\sum_{\vartheta=1}^n
      \big(
        A_{\sigma,u^\ell}^\vartheta N_\vartheta
        +A_\sigma^\vartheta N_{\vartheta,u^\ell}
      \big)
     +o(\varepsilon)$$
for their derivatives. Consequently, the new torsion coefficients can be expanded to
  $$\begin{array}{lll}
      \widetilde T_{\sigma,\ell}^\omega\negthickspace
      & = & \negthickspace\displaystyle
            \widetilde N_{\sigma,u^\ell}\cdot\widetilde N_\omega
            \,=\,N_{\sigma,u^\ell}\cdot N_\omega
                 +\varepsilon
                  \sum_{\vartheta=1}^n
                  \big(
                    A_{\sigma,u^\ell}^\vartheta N_\vartheta+A_\sigma^\vartheta N_{\vartheta,u^\ell}
                  \big)\cdot N_\omega
                 +\varepsilon
                  N_{\sigma,u^\ell}\cdot
                  \sum_{\vartheta=1}^n
                  A_\omega^\vartheta N_\vartheta
                 +o(\varepsilon) \\[4ex]
      & = & \negthickspace\displaystyle
            T_{\sigma,\ell}^\omega
            +\varepsilon A_{\sigma,u^\ell}^\omega
            +\varepsilon\sum_{\vartheta=1}^n
             \big\{
               A_\sigma^\vartheta T_{\vartheta,\ell}^\omega
               +A_\omega^\vartheta T_{\sigma,\ell}^\vartheta
             \big\}
            +o(\varepsilon),
    \end{array}$$
and for their squares we infer
  $$(\widetilde T_{\sigma,\ell}^\omega)^2
    =(T_{\sigma,\ell}^\omega)^2
     +2\varepsilon\,
      \bigg\{
        A_{\sigma,u^\ell}^\omega T_{\sigma,\ell}^\omega
        +\sum_{\vartheta=1}^n
         \big(
           A_\sigma^\vartheta T_{\vartheta,\ell}^\omega T_{\sigma,\ell}^\omega
           +A_\omega^\vartheta T_{\sigma,\ell}^\vartheta T_{\sigma,\ell}^\omega
         \big)
      \bigg\}
     +o(\varepsilon).$$
Before we insert this identity into the functional ${\mathcal T}(N)$ of total torsion, we observe
  $$\begin{array}{lll}
      \displaystyle
      \sum_{\sigma,\omega,\vartheta=1}^n\negthickspace
      \big\{
        A_\sigma^\vartheta T_{\vartheta,\ell}^\omega T_{\sigma,\ell}^\omega
        +A_\omega^\vartheta T_{\sigma,\ell}^\vartheta T_{\sigma,\ell}^\omega
      \big\}
      & = & \displaystyle\negthickspace
            \sum_{\sigma,\omega,\vartheta=1}^n
            \big\{
              A_\sigma^\vartheta T_{\vartheta,\ell}^\omega T_{\sigma,\ell}^\omega
              +A_\sigma^\vartheta T_{\omega,\ell}^\vartheta T_{\omega,\ell}^\sigma
            \big\} \\[4ex]
      & = & \displaystyle\negthickspace
            2\sum_{\sigma,\omega,\vartheta=1}^n
            A_\sigma^\vartheta T_{\vartheta,\ell}^\omega T_{\sigma,\ell}^\omega
            \,=\,0
    \end{array}$$
taking the skew-symmetry of the matrix $A$ into account. Thus the difference between the torsion functionals ${\mathcal T}(\widetilde N)$ and ${\mathcal T}(N)$ computes to ($A_{\sigma,u^\ell}^\omega T_{\sigma,\ell}^\omega=A_{\omega,u^\ell}^\sigma T_{\omega,\ell}^\sigma$!)
  $$\begin{array}{lll}
      \displaystyle
      {\mathcal T}(\widetilde N)-{\mathcal T}(N)\negthickspace
      & = & \displaystyle\negthickspace
            2\varepsilon
            \sum_{\sigma,\omega=1}^n\sum_{\ell=1}^2\,
            \int\hspace{-1.3ex}\int\limits_{\hspace{-2.0ex}B}
            A_{\sigma,u^\ell}^\omega T_{\sigma,\ell}^\omega\,dudv
            +o(\varepsilon) \\[5ex]
      & = & \displaystyle\negthickspace
            4\varepsilon
            \sum_{1\le\sigma<\omega\le n}
            \int\hspace{-1.5ex}\int\limits_{\hspace{-2.0ex}B}
            \Big\{
              A_{\sigma,u}^\omega T_{\sigma,1}^\omega
              +A_{\sigma,v}^\omega T_{\sigma,2}^\omega
            \Big\}
            +o(\varepsilon)\\[5ex]
      & = & \displaystyle\negthickspace
            \,4\varepsilon
            \sum_{1\le\sigma<\omega\le n}\,
            \int\limits_{\partial B}
            A_\sigma^\omega(T_{\sigma,1}^\omega,T_{\sigma,2}^\omega)\cdot\nu^t\,ds
            -4\varepsilon
              \sum_{1\le\sigma<\omega\le n}^n\,
              \int\hspace{-1.5ex}\int\limits_{\hspace{-2.0ex}B}
              A_\sigma^\omega\,\mbox{div}\,(T_{\sigma,1}^\omega,T_{\sigma,2}^\omega)\,dudv
            +o(\varepsilon).
    \end{array}$$
But $A$ was chosen arbitrarily which proves the statement.
\end{proof}
\subsection{The integral functions}
\label{par_integralfunctions}
Interpreting the Euler-Lagrange equations as integrability conditions, analogously to the situation considered in paragraph \ref{par_maxprinciple}, we find integral functions $\tau^{(\sigma\vartheta)}\in C^{k-1}(\overline B,\mathbb R)$ satisfying\index{integral function}
  $$\nabla\tau^{(\sigma\vartheta)}=\big(-T_{\sigma,2}^\vartheta,T_{\sigma,1}^\vartheta\big)
    \quad\mbox{in}\ B\quad\mbox{for all}\ \sigma,\vartheta=1,\ldots,n.$$
Due to the boundary conditions in $(T_{\sigma,1}^\vartheta,T_{\sigma,2}^\vartheta)\cdot\nu=0$ which imply $\nabla\tau^{(\sigma\vartheta)}\cdot(-v,u)=0$ on $\partial B$ with the unit tangent vector $(-v,u)$ at $\partial B,$ we may again choose $\tau^{(\sigma\vartheta)}$ so that
  $$\tau^{(\sigma\vartheta)}=0\quad\mbox{on}\ \partial B\quad\mbox{for all}\ \sigma,\vartheta=1,\ldots,n.$$
Note that the matrix $(\tau^{(\sigma\vartheta)})_{\sigma,\vartheta=1,\ldots,n}$ is skew-symmetric.
\subsection{A nonlinear elliptic system}
\label{par_nonlinearellipticsystem}
\label{prop_tausystem}
Let us now define the quantities
  $$\delta\tau^{(\sigma\vartheta)}
    :=\sum_{\omega=1}^n\mbox{det}\,\Big(\nabla\tau^{(\sigma\omega)},\nabla\tau^{(\omega\vartheta)}\Big),
  \quad\sigma,\vartheta=1,\ldots n.$$
The matrix $(\delta\tau^{(\sigma\vartheta)})_{\sigma,\vartheta=1,\ldots,n}$ is also skew-symmetric. The aim in this paragraph is to establish an elliptic system with quadratic growth in the gradient for $\tau^{(\sigma\vartheta)}.$
\begin{proposition}
\label{par_tausystem}
Let a normal Coulomb frame $N$ be given. Then the functions $\tau^{(\sigma\vartheta)}$, $\sigma,\vartheta=1,\ldots,n$, are solutions of the boundary value problems
  $$\Delta\tau^{(\sigma\vartheta)}
    =-\,\delta\tau^{(\sigma\vartheta)}+S_{\sigma,12}^\vartheta\quad\mbox{in}\ B,\quad
    \tau^{(\sigma\vartheta)}=0\quad\mbox{on}\ \partial B\,,$$
where $\delta\tau^{(\sigma\vartheta)}$ grows quadratically in the gradient $\nabla\tau^{(\sigma\vartheta)}.$
\end{proposition}
\begin{proof}
Choose any $(\sigma,\vartheta)\in\{1,\ldots,n\}\times\{1,\ldots,n\}.$ The representation formula
  $$S_{\sigma,12}^\vartheta
    =T_{\sigma,1,v}^\vartheta-T_{\sigma,2,u}^\vartheta
     +\sum_{\omega=1}^n
      \Big\{
        T_{\sigma,1}^\omega T_{\omega,2}^\vartheta-T_{\sigma,2}^\omega T_{\omega,1}^\vartheta
      \Big\}$$
for the normal curvature tensor together with $\nabla\tau^{(\sigma\vartheta)}=(-T_{\sigma,2}^\vartheta,T_{\sigma,1}^\vartheta)$ yields
  $$\begin{array}{rcl}
      \Delta\tau^{(\sigma\vartheta)}\negthickspace
      & = & \negthickspace\displaystyle
            T_{\sigma,1,v}^\vartheta -T_{\sigma,2,u}^\vartheta
            \,=\,-\sum_{\omega=1}^n
                  \Big\{
                    T_{\sigma,1}^\omega T_{\omega,2}^\vartheta
                    -T_{\sigma,2}^\omega T_{\omega,1}^\vartheta
                  \Big\}
                 +S_{\sigma,12}^\vartheta \\[4ex]
      & = & \negthickspace\displaystyle
            \sum_{\omega=1}^n
            \Big\{
              \tau_v^{(\sigma\omega)}\tau_u^{(\omega\vartheta)}-\tau_u^{(\sigma\omega)}\tau_v^{(\omega\vartheta)}
            \Big\}
            +S_{\sigma,12}^\vartheta
    \end{array}$$
proving the statement.
\end{proof}
\vspace*{4ex}
\begin{center}
\rule{35ex}{0.1ex}
\end{center}
\vspace*{4ex}
\section{Examples}
\label{section_examplesallgemein}
Let us evaluate this nonlinear system in the special cases $n=2$ and $n=3.$
\subsection{The case $n=2$}
\label{par_n2allgemein}
In this case there is only one integral function $\tau^{(12)}$ satisfying
  $$\Delta\tau^{(12)}=S_{1,12}^2\quad\mbox{in}\ B,\quad
    \tau^{(12)}=0\quad\mbox{on}\ \partial B.$$
This is exactly the Poisson equation with homogeneous boundary data from paragraph \ref{par_maxprinciple} with $\tau=\tau^{(12)}$ and $S_NW=S_{1,12}^2.$
\goodbreak\noindent
\subsection{The case $n=3$}
\label{par_n3allgemein}
If $n=3$ we have three relations
  $$\begin{array}{lll}
      \Delta\tau^{(12)}\negthickspace
      & = & \negthickspace\displaystyle
            \tau_v^{(13)}\tau_u^{(32)}-\tau_u^{(13)}\tau_v^{(32)}+S_{1,12}^2\,, \\[2ex]
      \Delta \tau^{(13)}\negthickspace
      & = & \negthickspace\displaystyle
            \tau_v^{(12)}\tau_u^{(23)}-\tau_u^{(12)}\tau_v^{(23)}+S_{1,12}^3\,, \\[2ex]
      \Delta \tau^{(23)}\negthickspace
      & = & \negthickspace\displaystyle
            \tau_v^{(21)}\tau_u^{(13)}-\tau_u^{(21)}\tau_v^{(13)}+S_{2,12}^3\,.
    \end{array}$$
Now if we set
  $${\mathcal T}:=(\tau^{(12)},\tau^{(13)},\tau^{(23)})
    \quad\mbox{and}\quad
    {\mathcal S}:=(S_{1,12}^2,S_{1,12}^3,S_{2,12}^3)=W{\mathfrak S}_N$$
with the curvature vector ${\mathfrak S}_N=W^{-1}(S_{1,12}^2,S_{1,12}^3,S_{2,12}^3)$ of the normal bundle from paragraph \ref{par_normalcurvaturevector}, we find
  $$\Delta{\mathcal T}
    ={\mathcal T}_u\times{\mathcal T}_v+{\mathcal S}\quad\mbox{in}\ B,
    \quad\mathcal T=0\quad\mbox{on}\ \partial B$$
with the usual vector product $\times$ in $\mathbb R^3.$ That means: \emph{$\mathcal T$ solves an inhomogeneous $H$-surface system with constant mean curvature $H=\frac12$ and vanishing boundary data,}\index{mean curvature system} compare with the mean curvature system from Lecture I, that is,
  $$\triangle X=2HWN$$
with the scalar mean curvature $H,$ the area element $W,$ and unit normal vector $N.$ If, additionally, $X$ satsifies the conformality relations, then $X$ even represents an immersion with scalar mean curvature $H.$\\[1ex]
In particular, if ${\mathfrak S}_N\equiv 0,$ then ${\mathcal T}\equiv 0$ by a result of Wente \cite{wente_1975}. We consider the general situation of higher codimensions in the next sections.
\vspace*{6ex}
\begin{center}
\rule{35ex}{0.1ex}
\end{center}
\vspace*{6ex}
\section{Quadratic growth in the gradient}
\label{section_quadraticgrowth}
\subsection{A Grassmann-type vector}
\label{par_grassmanngrowth}
The last example gives rise to the definition of the following {\it vector of Grassmann type}\index{Grassmann type vectors}
  $${\mathcal T}
    :=\big(\tau^{(\sigma\vartheta)}\big)_{1\le\sigma<\vartheta\le n}\in\mathbb R^N\,,
    \quad N:=\frac{n}{2}\,(n-1).$$
In our examples ${\mathcal T}$ works as follows:
  $$\begin{array}{lll}
      {\mathcal T}=\tau^{(12)}\in\mathbb R                                     & \mbox{for}\ n=2, \\[2ex]
      {\mathcal T}=\big(\tau^{(12)},\tau^{(13)},\tau^{(23)}\big)\in\mathbb R^3 & \mbox{for}\ n=3.
    \end{array}$$
Analogously we define the Grassmann-type vectors
  $$\delta\mathcal T:=\big(\delta\tau^{(\sigma\vartheta)}\big)_{1\le\sigma<\vartheta\le n}\in\mathbb R^N\,,\quad
    {\mathcal S}=W{\mathfrak S}_N$$
with the curvature vector ${\mathfrak S}_N=W^{-1}(S_{1,12}^2,S_{1,12}^3,\ldots).$ Then Proposition \ref{prop_tausystem} can be written as
  $$\Delta\mathcal T
    =-\delta\mathcal T+{\mathcal S}\quad\mbox{in}\ B,\quad
   \mathcal T=0\quad\mbox{on}\ \partial B.$$
From the definition of $\delta\mathcal G$ we immediately obtain
  $$|\Delta\mathcal T|\le c\,|\nabla\mathcal T|^2+|{\mathcal S}|\quad\mbox{in}\ B$$
with some real constant $c>0.$
\subsection{Nonlinear system with quadratic growth}
\label{par_quadraticgrowth}
The exact knowledge of this constant $c>0$ will become important later.
\begin{proposition}
It holds
  $$|\Delta{\mathcal T}|
    \le\frac{\sqrt{n-2}}{2}\,|\nabla\mathcal T|^2+|{\mathcal S}|
    \quad\mbox{in}\ B.$$
\end{proposition}
\begin{proof}
We already know
  $$|\Delta{\mathcal T}
    |\le |\delta\mathcal T|+|{\mathcal S}|\quad\mbox{in}\ B.$$
It remains to estimate $|\delta\mathcal T|$ appropriately. We begin with computing
  $$|\delta\mathcal T|^2
    =\sum_{1\le\sigma<\vartheta\le n}
     \left\{\,
       \sum_{\omega=1}^n
       \det\big(\nabla\tau^{(\sigma\omega)},\nabla\tau^{(\omega\vartheta)}\big)
     \right\}^2
    \le(n-2)
       \sum_{1\le\sigma<\vartheta\le n}
       \left\{\,
         \sum_{\omega=1}^n
         \det\big(\nabla\tau^{(\sigma\omega)},\nabla\tau^{(\omega\vartheta)}\big)^2
       \right\}.$$
Note that only derivatives of elements of ${\mathcal T}$ appear on the right hand side of $|\delta{\mathcal T}|^2.$ Moreover, this right hand side can be estimated by $|{\mathcal T}_u\wedge{\mathcal T}_v|^2$ because ${\mathcal T}_u\wedge{\mathcal T}_v$ has actually more elements than ${\mathcal R}.$\footnote{In particular, elements of the form $\mbox{det}\,(\nabla\tau^{(\sigma\omega)},\nabla\tau^{(\omega'\vartheta)})^2$ appear in $|{\mathcal T}_u\wedge{\mathcal T}_v|^2,$ but they do not appear in ${\mathcal R}.$} Thus using Lagrange's identity
  $$|X\wedge Y|^2=|X|^2|Y|^2-(X\cdot Y^t)^2\le|X|^2|Y|^2$$
we can estimate as follows
  $$|\delta{\mathcal T}|^2
    \le (n-2)|{\mathcal T}_u\wedge{\mathcal T}_v|^2
    \le (n-2)|{\mathcal T}_u|^2|{\mathcal T}_v|^2\,.$$
Taking all together shows
  $$|\Delta\mathcal T|
    \le\sqrt{n-2}\,|\mathcal T_u||\mathcal T_v|+|{\mathfrak S}|
    \le\frac{\sqrt{n-2}}{2}\,|\nabla\mathcal T|^2+|{\mathcal S}|$$
which proves the statement.
\end{proof}
\vspace*{6ex}
\begin{center}
\rule{35ex}{0.1ex}
\end{center}
\vspace*{6ex}
\section{Torsion free normal frames}
\subsection{The case $n=3$}
\label{par_torsionfreen3}
As already mentioned, in Wente \cite{wente_1975} we find a uniqueness result for solutions of the homogeneous system, corresponding to the flat normal bundle situation ${\mathfrak S}_N=0,$ if the codimension is $3:$
\begin{proposition}
(Wente \cite{wente_1975} )\\
The only solution of the elliptic system
  $$\Delta{\mathcal T}={\mathcal T}_u\times{\mathcal T}_v\quad\mbox{in}\ B,\quad
    {\mathcal T}=0\quad\mbox{on}\ \partial B,$$
is ${\mathcal T}\equiv 0.$
\end{proposition}
\noindent
A proof of this result follows from asymptotic expansions of solutions ${\mathcal T}$ in the interior $B$ and on the boundary $\partial B.$ We particularly refer to Hartman and Wintner \cite{hartman_wintner_1953} and Hildebrandt \cite{hildebrandt_1970}; see also Heinz' result mentionend in paragraph \ref{par_zerotorsion}. In particular, we infer
\begin{corollary}
Suppose that the immersion $X\colon\overline B\to\mathbb R^5$ admits a normal Coulomb frame\index{normal Coulomb frame}. Then this frame is free of torsion if and only if the curvature vector ${\mathfrak S}_N$ vanishes identically.\index{torsion free normal frame}
\end{corollary}
\noindent
This is a special case of a general result we discuss in paragraph \ref{par_zerotorsion}.
\subsection{An auxiliary function}
\label{par_auxiliaryfunction}
To handle the general case $n>3$ we need some preparations. Let us start with the
\begin{lemma}
Let ${\mathfrak S}_N=0.$ Then the function
  $$\Phi(w)
    ={\mathcal T}_w(w)\cdot{\mathcal T}_w(w)
    =\sum_{1\le\sigma<\vartheta\le n}
     \tau_w^{(\sigma\vartheta)}\tau_w^{(\sigma\vartheta)}\,,$$
with the complex derivative $\phi_{w}=\frac{1}{2}\,(\phi_u+i\phi_v),$ vanishes identically in $\overline B.$
\end{lemma}
\begin{proof}
We will prove that $\Phi$ solves the boundary value problem
  $$\Phi_{\overline w}=0\quad\mbox{in}\ B,\quad
    \mbox{Im}(w^2\Phi)=0\quad\mbox{on}\ \partial B.$$
Then the analytic function $\Psi(w):=w^2\Phi(w)$ has vanishing imaginary part, and the Cauchy-Riemann equations imply $\Psi(w)\equiv c\in\mathbb R$ so that the assertion follows from $\Psi(0)=0.$
\begin{itemize}
\item[1.]
In order to deduce the stated boundary condition, recall that $\tau^{(\sigma\vartheta)}=0$ on $\partial B.$ Thus all tangential derivatives vanish identically because
\begin{equation*}%\label{5.4}
  -v\tau_u^{(\sigma\vartheta)}+u\tau_v^{(\sigma\vartheta)}
  =-\mbox{Im}(w\tau_w^{(\sigma\vartheta)})=0\quad\mbox{on}\ \partial B
\end{equation*}
for all $\sigma,\vartheta=1,\ldots,n.$ The boundary condition follows from
  $$\begin{array}{lll}
      \mbox{Im}\,(w^2\Phi)\negthickspace
      & = & \negthickspace\displaystyle
            \mbox{Im}\,\Big(w^2\,\mathcal T_w\cdot\mathcal T_w\Big)
            \,=\,\mbox{Im}\,
                 \bigg\{
                   w^2\sum_{1\le\sigma<\vartheta\le n}\tau_w^{(\sigma\vartheta)}\tau_w^{(\sigma\vartheta)}
                 \bigg\} \\[4ex]
      & = & \negthickspace\displaystyle
            \sum_{1\le\sigma<\vartheta\le n}\!
                 \mbox{Im}\,\Big\{\big(w\tau_w^{(\sigma\vartheta)}\big)\big(w\tau_w^{(\sigma\vartheta)}\big)\Big\}
            \,=\,2\!\sum_{1\le\sigma<\vartheta\le n}\!
             \mbox{Re}\,\big(w\tau_w^{(\sigma\vartheta)}\big)\,
             \mbox{Im}\,\big(w\tau_w^{(\sigma\vartheta)}\big)\,=\,0
    \end{array}$$
on the boundary $\partial B.$
\item[2.]
Now we show the analyticity of $\Phi$  with the aid of $\Delta\tau^{(\sigma\vartheta)}=4\tau_{w\overline w}^{(\sigma\vartheta)}=-\delta\tau^{(\sigma\vartheta)}:$ Interchanging indices cyclically yields
  $$\begin{array}{lll}
       2\Phi_{\overline w}\negthickspace
      & = & \negthickspace\displaystyle
            4\,\mathcal T_w\cdot\mathcal T_{w\overline w}
            \,=\,4\sum_{1\le\sigma<\vartheta\le n}
                 \tau_w^{(\sigma\vartheta)}\tau_{w\overline w}^{(\sigma\vartheta)}
            \,=\,\frac{1}{2}\,\sum_{\sigma,\vartheta=1}^n
	         \tau_w^{(\sigma\vartheta)}\Delta\tau^{(\sigma\vartheta)}\\[4ex]
      & = & \negthickspace\displaystyle
            \frac{1}{4}\,\sum_{\sigma,\vartheta,\omega=1}^n
            \Big\{
              \tau_v^{(\sigma\omega)}\tau_u^{(\omega\vartheta)}\tau_u^{(\sigma\vartheta)}
              -\tau_u^{(\sigma\omega)}\tau_v^{(\omega\vartheta)}\tau_u^{(\sigma\vartheta)}
            \Big\} \\[4ex]
      &   & \negthickspace\displaystyle
            -\,\frac{i}{4}\,\sum_{\sigma,\vartheta,\omega=1}^n
                \Big\{
                  \tau_v^{(\sigma\omega)}\tau_u^{(\omega\vartheta)}\tau_v^{(\sigma\vartheta)}
                  -\tau_u^{(\sigma\omega)}\tau_v^{(\omega\vartheta)}\tau_v^{(\sigma\vartheta)}
                \Big\} \\[4ex]
      & = & \negthickspace\displaystyle
            \frac{1}{4}\,\sum_{\sigma,\vartheta,\omega=1}^n
            \Big\{
              \tau_v^{(\omega\vartheta)}\tau_u^{(\vartheta\sigma)}\tau_u^{(\omega\sigma)}
              -\tau_u^{(\sigma\omega)}\tau_v^{(\omega\vartheta)}\tau_u^{(\sigma\vartheta)}
            \Big\} \\[4ex]
      &   & \negthickspace\displaystyle
            -\,\frac{i}{4}\,\sum_{\sigma,\vartheta,\omega=1}^n
                \Big\{
                  \tau_v^{(\vartheta\sigma)}\tau_u^{(\sigma\omega)}\tau_v^{(\vartheta\omega)}
                  -\tau_u^{(\sigma\omega)}\tau_v^{(\omega\vartheta)}\tau_v^{(\sigma\vartheta)}
                \Big\},
    \end{array}$$
which shows $\Phi_{\overline w}=0.$ The proof is complete.
\end{itemize}
\end{proof}
\subsection{The case $n>3$}
\label{par_zerotorsion}
The main result of this section is the
\begin{theorem}
Suppose that the immersion $X\colon\overline B\to\mathbb R^{n+2}$ admits a normal Coulomb frame $N.$\index{normal Coulomb frame} Then this frame is free of torsion if and only if the curvature vector ${\mathfrak S}_N$ of its normal bundle vanishes identically.\index{curvature vector of normal bundle}\index{torsion free normal frame}
\end{theorem}
\goodbreak\noindent
\begin{proof}
Let $N$ be a normal Coulomb frame. If it is free of torsion then ${\mathfrak S}_N$ vanishes identically. So assume conversely ${\mathfrak S}_N\equiv 0$ and let us show that $N$ is free of torsion. Consider for this aim the Grassmann-type vector ${\mathcal T}$ from above. Because ${\mathcal T}_w\cdot{\mathcal T}_w\equiv 0$ holds true by the previous lemma, we find
  $$|{\mathcal T}_u|=|{\mathcal T}_v|,\quad
    {\mathcal T}_u\cdot{\mathcal T}_v=0
    \quad\mbox{in}\ B,$$
i.e. ${\mathcal T}$ is a conformally parametrized solution of
  $$\Delta{\mathcal T}=-\,\delta{\mathcal T}\quad\mbox{in}\ B,\quad
    {\mathcal T}=0\quad\mbox{on}\ \partial B.$$
According to the quadratic growth condition $|\delta\mathcal T|\le c|\nabla\mathcal T|^2$ (see paragraph \ref{par_grassmanngrowth}), the arguments in Heinz \cite{heinz_1970} apply\footnote{Let $X\colon\in C^2(\overline B,R^n)$ solve the elliptic system $\triangle X=Hf(X,X_u,X_v)$ together with $X_u^2=X_v^2,$ $X_u\cdot X_v=0,$ where $|f(x,p,q)|\le\mu(|x|)(|p|^2+|q|^2).$ Then if $X_u(w_0)=0$ for some $w_0\in\partial B,$ but $X_u\not_\equiv 0,$ the following asymptotic expansion $X_w(w_0)=a(w-w_0)^\ell-o(|w-w_0|^\ell)$ holds for $w\to w_0,$ where $a\in\mathbb C$ with $a_1^2+\ldots a_n^2=0,$ and $\ell\in\mathbb N\setminus\{0\}.$}: Assume ${\mathcal T}\not\equiv\mbox{const}$ in $B$. Then the asymptotic expansion stated in the Satz of Heinz \cite{heinz_1970} implies that boundary branch points $w_0\in\partial B$ with ${\mathcal T}_u(w_0)={\mathcal T}_v(w_0)=0$ are isolated. But this contradicts our boundary condition ${\mathcal T}|_{\partial B}=0!$ Thus ${\mathcal T}\equiv\mbox{const}=0$ which implies $\tau^{(\sigma\vartheta)}\equiv 0$ for all $\sigma,\vartheta=1,\ldots,n,$ and the normal Coulomb frame is free of torsion. The theorem is proved.
\end{proof}
\vspace*{6ex}
\begin{center}
\rule{35ex}{0.1ex}
\end{center}
\vspace*{6ex}
\section{Non-flat normal bundles}
\label{section_nonflat}
We establish upper bounds for the torsion coefficients of normal Coulomb frames.
\subsection{An upper bound via Wente's $L^\infty$-estimate}
\label{par_wentebound}
Our first result is
\begin{proposition}
Let $N$ be a normal Coulomb frame for the conformally parametrized immersion $X.$ Then the Grassmann-type vector ${\mathcal T}$ satisfies\index{Grassmann type vectors}
  $$\|\mathcal T\|_{L^\infty(B)}
    =\sup_{w\in B}\sqrt{\sum_{1\le\sigma<\vartheta\le n}|\tau^{(\sigma\vartheta)}(w)|^2}
    \le\frac{n-2}{2\pi}\,\|\nabla\mathcal T\|_{L^2(B)}^2
       +\frac{n(n-1)}{8}\,\|{\mathfrak S}_NW\|_{L^\infty(B)}$$
with the Lebesgue $L^2$-norm $\|\cdot\|_{L^2(B)}$ etc.
\end{proposition}
\begin{proof}
\begin{itemize}
\item[1.]
For $1\le \sigma<\vartheta\le n,$ $\omega\in\{1,\ldots,n\}$ with $\omega\not\in\{\sigma,\vartheta\},$ and given integral functions $\tau^{(\sigma\vartheta)}$ we define the functions $y^{(\sigma\vartheta\omega)}$ as the unique solutions of
  $$\Delta y^{(\sigma\vartheta\omega)}
    =-\det\big(\nabla\tau^{(\sigma\omega)},\nabla\tau^{(\omega\vartheta)}\big)
    \quad\mbox{in}\ B,\quad
    y^{(\sigma\vartheta\omega)}=0\quad\mbox{on}\ \partial B.$$
Wente's $L^\infty$-estimate\index{Wente's $L^\infty$-estimate} from Wente \cite{wente_1980} together with Topping \cite{topping_1997} then yields the {\it optimal inequalities}\footnote{In 1980 H. Wente proved: Let $\Phi\in C^0(\overline B)\cap H^{1,2}(B)$ be a solution of $\Delta\Phi=-(f_ug_v-f_vg_u)$ in $B,$ $\Phi=0$ on $\partial\Omega,$ with $f,g\in H^{1,2}(B),$ then $\|\Phi\|_{L^\infty(B)}+\|\nabla\Phi\|_{L^2(B)}\le C\|\nabla f\|_{L^2(B)}\|\nabla g\|_{L^2(B)}\,.$ Following Topping \cite{topping_1997} we may set $\frac{C}{2}=\frac{1}{4\pi}$ after applying H\"older's inequality. See also section 22 below.} (see also section 22.1)
  $$\|y^{(\sigma\vartheta\omega)}\|_{L^\infty(B)}
    \le\frac 1{4\pi}\Big(\|\nabla\tau^{(\sigma\omega)}\|_{L^2(B)}^2+\|\nabla\tau^{(\omega\vartheta)}\|_{L^2(B)}^2\Big),
    \quad 1\le\sigma<\vartheta\le n,\ \omega\not\in\{\sigma,\vartheta\}.$$
In addition, we introduce the Grassmann-type vector $\mathcal Z=(z^{(\sigma\vartheta)})_{1\le\sigma<\vartheta\le n}$ as the unique solution of
  $$\Delta \mathcal Z=\mathcal S\quad\mbox{in}\ B,\quad
    \mathcal Z=0\quad\mbox{on}\ \partial B.$$
with ${\mathcal S}=(S_{1,12}^2,S_{1,12}^3,\ldots)$ and ${\mathcal S}=W{\mathfrak S}_N.$ Introduce new indices to write ${\mathcal Z}=(Z_1,\ldots,Z_N)$ and ${\mathcal S}=(S_1,\ldots,S_N)$ for the moment. We use Poisson's representation formula to estimate as follows
  $$\begin{array}{lll}
      |{\mathcal Z}(w)|\negthickspace
      & \le & \negthickspace\displaystyle
              \sum_{i=1}^N|Z_i|
              \,=\,\sum_{i=1}^N
                   \left|\,
                     \int\hspace{-0.25cm}\int\limits_{\hspace{-0.3cm}B}
                     \phi(\zeta;w)S_i(\zeta)\,d\xi d\eta\,
                   \right|
              \,\le\,\sum_{i=1}^N
                     \int\hspace{-0.25cm}\int\limits_{\hspace{-0.3cm}B}
                     |\phi(\zeta;w)||S_i(\zeta)|\,d\xi d\eta\, \\[5ex]
      & \le & \negthickspace\displaystyle
              \sqrt{N}\,
              \int\hspace{-0.25cm}\int\limits_{\hspace{-0.3cm}B}
              |\phi(\zeta;w)|\sqrt{\sum_{i=1}^N|S_i(\zeta)|^2}\,d\xi d\eta
              \,=\,\sqrt{N}\,
                   \int\hspace{-0.25cm}\int\limits_{\hspace{-0.3cm}B}
                   |\phi(\zeta;w)||{\mathcal S}(\zeta)|\,d\xi d\eta
    \end{array}$$
with the Green function $\phi(\zeta;w)$ for $\Delta$ in $B;$ $\zeta=(\xi,\eta).$ From Lecture II we already know
  $$\int\hspace{-1.3ex}\int\limits_{\hspace{-1.8ex}B}
    |\phi(\zeta;w)|\,d\xi d\eta
    =\frac{1-|w|^2}{4}
     \le\frac{1}{4}\,,$$
which enables us to continue to estimate $|{\mathcal Z}(w)|$ to get
  $$\|\mathcal Z\|_{L^\infty(B)}
    \le\sqrt{N}\,\|{\mathcal S}\|_{L^\infty(B)}
       \int\hspace{-0.25cm}\int\limits_{\hspace{-0.3cm}B}
      |\phi(\zeta;w)|\,d\xi d\eta
    \le\frac{\sqrt N}4\|\mathcal S\|_{L^\infty(B)}\,.$$
\item[2.]
Now recall that
  $$\Delta\tau^{(\sigma\vartheta)}
    =-\sum_{\omega=1}^n
      \mbox{det}\,(\nabla\tau^{(\sigma\omega)},\nabla\tau^{(\omega\vartheta)})
      +S_{\sigma,12}^\vartheta
    =\sum_{\omega=1}^n
     \Delta y^{(\sigma\vartheta\omega)}
     +\Delta z^{(\sigma\vartheta)}\,.$$
Taking account of the unique solvability of the above mentioned Dirichlet problems with vanishing boundary data, the maximum principle yields
  $$\tau^{(\sigma\vartheta)}
    =\sum_{\omega\not\in\{\sigma,\vartheta\}}y^{(\sigma\vartheta\omega)}
     +z^{(\sigma\vartheta)},
     \quad 1\le\sigma<\vartheta\le n.$$
Now we collect all the estimates obtained and get (we rearrange the summations and redefine some indices of the sums)
  $$\begin{array}{rcl}
      \|\mathcal T\|_{L^\infty(B)}\negthickspace
      & \le & \negthickspace\displaystyle
              \sum_{1\le\sigma<\vartheta\le n}
              \sum_{\omega\not\in\{\sigma,\vartheta\}}
              \|y^{(\sigma\vartheta\omega)}\|_{L^\infty(B)}
              +\sum_{1\le\sigma<\vartheta\le n}
               \|z^{(\sigma\vartheta)}\|_{L^\infty(B)} \\[5ex]
      & \le & \negthickspace\displaystyle
              \sum_{1\le\sigma<\vartheta\le n}
              \sum_{\omega\not\in\{\sigma,\vartheta\}}
              \|y^{(\sigma\vartheta\omega)}\|_{L^\infty(B)}
              +\sqrt{N}\,\|\mathcal Z\|_{L^\infty(B)}\\[5ex]
      & \le & \negthickspace\displaystyle
              \frac1{4\pi}
              \sum_{1\le\sigma<\vartheta\le n}
              \sum_{\omega\not\in\{\sigma,\vartheta\}}
              \Big(
                \|\nabla\tau^{(\sigma\omega)}\|_{L^2(B)}^2+\|\nabla\tau^{(\omega\vartheta)}\|_{L^2(B)}^2
              \Big)
              +\frac{N}{4}\,\|\mathcal S\|_{L^\infty(B)}\\[5ex]
      & = & \negthickspace\displaystyle
              \frac1{4\pi}\,
              \Bigg\{
                \sum_{1\le\omega<\sigma<\vartheta\le n}
                \Big(
                  \|\nabla\tau^{(\omega\sigma)}\|_{L^2(B)}^2+\|\nabla\tau^{(\omega\vartheta)}\|_{L^2(B)}^2
                \Big) \\[5ex]
      &   & \negthickspace\displaystyle
            \hspace{6ex}
            +\,\sum_{1\le\sigma<\omega<\vartheta\le n}
                 \Big(
                   \|\nabla\tau^{(\sigma\omega)}\|_{L^2(B)}^2+\|\nabla\tau^{(\omega\vartheta)}\|_{L^2(B)}^2
                 \Big)\\[4ex]
      &     & \negthickspace\displaystyle
              \hspace*{6ex}
                +\sum_{1\le\sigma<\vartheta<\omega\le n}
                 \Big(
                   \|\nabla\tau^{(\sigma\omega)}\|_{L^2(B)}^2+\|\nabla\tau^{(\vartheta\omega)}\|_{L^2(B)}^2
                 \Big)
              \Bigg\}
              +\frac{N}{4}\,\|\mathcal S\|_{L^\infty(B)} \\[5ex]
      &  =  & \negthickspace\displaystyle
              \ldots
    \end{array}$$
  $$\begin{array}{rcl}
      \ldots\negthickspace
      &  =  & \negthickspace\displaystyle
              \frac1{4\pi}\,
              \Bigg\{
                \sum_{1\le \sigma<\vartheta<\omega\le n}
                \|\nabla\tau^{(\sigma\vartheta)}\|_{L^2(B)}^2
                +\sum_{1\le\sigma<\omega<\vartheta\le n}
                 \|\nabla\tau^{(\sigma\vartheta)}\|_{L^2(B)}^2
                +\sum_{1\le\sigma<\vartheta<\omega\le n}
                 \|\nabla\tau^{(\sigma\vartheta)}\|_{L^2(B)}^2\\[5ex]
      &     & \negthickspace\displaystyle
              \hspace*{5.2ex}
                +\!\sum_{1\le\omega<\sigma<\vartheta\le n}\!
                 \|\nabla\tau^{(\sigma\vartheta)}\|_{L^2(B)}^2
                +\!\sum_{1\le\sigma<\omega<\vartheta\le n}\!
                 \|\nabla\tau^{(\sigma\vartheta)}\|_{L^2(B)}^2
                +\!\sum_{1\le\omega<\sigma<\vartheta\le n}\!
                 \|\nabla\tau^{(\sigma\vartheta)}\|_{L^2(B)}^2
              \Bigg\} \\[5ex]
      &     & \negthickspace\displaystyle
              +\,\frac{N}{4}\,\|\mathcal S\|_{L^\infty(B)}\\[4ex]
      &  =  & \negthickspace\displaystyle
              \frac1{2\pi}
              \sum_{1\le\sigma<\vartheta\le n}
              \sum_{\omega\not\in\{\sigma,\vartheta\}}
              \|\nabla\tau^{(\sigma\vartheta)}\|_{L^2(B)}^2
              +\frac{N}{4}\,\|\mathcal S\|_{L^\infty(B)}\\[5ex]
      &  =  & \negthickspace\displaystyle
              \frac{n-2}{2\pi}\,\|\nabla\mathcal T\|_{L^2(B)}^2
              +\frac{1}{4}\,\frac{n(n-1)}{2}\,\|\mathcal S\|_{L^\infty(B)}\,.
    \end{array}$$
This proves the statement.
\end{itemize}
\end{proof}
\subsection{An upper bound via Poincar\'e's inequality}
\label{par_poincarebound}
For large codimension $n$, the above estimate is somewhat unsatisfactory. Alternatively, we show
  $$|z^{(\sigma\vartheta)}(w)|\le\sqrt{\frac{2}{\pi}}\,\|S_{\sigma,12}^\vartheta\|_{L^2(B)}
    \quad\mbox{in}\ B\quad\mbox{for all}\ 1\le\sigma<\vartheta\le n$$
for the Grassmann-type vector ${\mathcal Z}=(z^{(12)},z^{(13)},\ldots)$ from the previous paragraph. Then we would have
  $$\|\mathcal Z\|_{L^\infty(B)}
    =\sup_B\sqrt{\sum_{1\le\sigma<\vartheta\le n}|z^{(\sigma\vartheta)}(w)|^2}
     \le\sqrt{\frac2\pi}\,\sqrt{\sum_{1\le\sigma<\vartheta\le n}\|S_{\sigma,12}^\vartheta\|_{L^2(B)}^2}
    =\sqrt{\frac2\pi}\,\|\mathcal S\|_{L^2(B)}
     \le\sqrt2\,\|\mathcal S\|_{L^\infty(B)}$$
which finally leads us to a smaller upper bound for $\|{\mathcal T}\|_{L^\infty(B)}$ at least for codimensions $n=2,3.$ In order to prove the stated inequality we start with the Poisson representation formula
  $$z^{(\sigma\vartheta)}
    =\int\hspace*{-1.3ex}\int\limits_{\hspace{-1.8ex}B}
     \phi(\zeta;w)S_{\sigma,12}^\vartheta(\zeta)\,d\xi d\eta,\quad
    z^{(\sigma\vartheta)}=0\quad\mbox{on}\ \partial B.$$
Applying the H\"older and the Poincar\'e inequality gives\index{Poincar\'e inequality}
\begin{equation}\label{6.24}
  |z^{(\sigma\vartheta)}(w)|
  \le\|\phi(\cdot\,;w)\|_{L^2(B)}\|S_{\sigma,12}^\vartheta\|_{L^2(B)}
  \le\frac1{2\sqrt\pi}\|\nabla_\zeta\phi(\cdot\,;w)\|_{L^1(B)}\|S_{\sigma,12}^\vartheta\|_{L^2(B)}\,.
\end{equation}
For the optimal constant $\frac1{2\sqrt\pi}$ in the Sobolev inequality we refer to Gilbarg and Trudinger \cite{gilbarg_trudinger_1983}, paragraph 7.7 and the references therein. Furthermore, $\phi=\phi(\zeta;w):=\frac1{2\pi}\log|\frac{\zeta-w}{1-\overline w\zeta}|$ denotes again Green's function for the Laplace operator $\Delta$ in $B$ which satisfies $\phi(\cdot\,;w)\in H^1_0(B)$ as well as
  $$2\phi_\zeta(\zeta;w)
    \equiv\phi_\xi(\zeta;w)-i\phi_\eta(\zeta;w)
    =\frac{1}{2\pi}\,
     \overline{
       \left(
         \frac{\zeta-w}{|\zeta-w|^2}+w\frac{1-\overline w\zeta}{|1-\overline w\zeta|^2}
         \right)},
    \quad w\not=\zeta.$$
A straightforward calculation shows
\begin{equation*}
  |\nabla_\zeta\phi(\zeta;w)|
  \equiv 2|\phi_\zeta(\zeta;w)|
  =\frac1{2\pi}\frac{1-|w|^2}{|\zeta-w|\,|1-\overline w\zeta|}\le\frac1{2\pi}\frac{1+|w|}{|\zeta-w|}
  \le\frac1\pi\frac1{|\zeta-w|},\quad\zeta\not=w.
\end{equation*}
And since the right hand side in the inequality
\begin{equation*}
  \int\hspace{-0.25cm}\int\limits_{\hspace{-0.3cm}B}
  |\nabla_\zeta\phi(\zeta;w)|\,d\xi d\eta
  \le\frac{1}{\pi}\ \,
     \int\hspace*{-0.25cm}\int\limits_{\hspace*{-0.4cm}B_\delta(w)}
     \frac1{|\zeta-w|}\,d\xi\,d\eta
  +\frac{1}{\pi}\ \ \,
   \int\hspace*{-0.4cm}\int\limits_{\hspace*{-0.45cm}B\setminus B_\delta(w)}
   \frac1{|\zeta-w|}\,d\xi\,d\eta
  \le 2\delta+\frac{1}{\delta}
\end{equation*}
becomes minimal for $\delta=\frac1{\sqrt2}$, we arrive at
  $$|z^{(\sigma\vartheta)}(w)|
    \le\frac{1}{2\sqrt{\pi}}
       \left(
         2\cdot\frac{1}{\sqrt{2}}+\sqrt{2}
       \right)\|S_{\sigma,12}^\vartheta\|_{L^2(B)}
    =\frac{\sqrt{2}}{\sqrt{\pi}}\,\|S_{\sigma,12}^\vartheta\|_{L^2(B)}$$
proving the stated inequality. Rearranging gives
  $$|{\mathcal Z}(w)|
    \le\frac{\sqrt{2}}{\sqrt{\pi}}\cdot\sqrt{\pi}\cdot\|S_{\sigma,12}^\vartheta\|_{L^\infty(B)}
    \le\sqrt{2}\,\|{\mathcal S}\|_{L^\infty(B)}\,.$$
Thus we have
\begin{proposition}
Let $N$ be a normal Coulomb frame for the conformally parametrized immersion $X.$ Then the Grassmann-type vector ${\mathcal T}$ satisfies
  $$\|\mathcal T\|_{L^\infty(B)}
    \le\frac{n-2}{2\pi}\|\nabla \mathcal T\|_{L^2(B)}^2+\sqrt2\,\|\mathcal S\|_{L^\infty(B)}\,.$$
\end{proposition}
\noindent
This provides us a better estimate at least if $n=2,3.$
\subsection{An estimate for the torsion coefficients}
\label{par_torsionestimate}
We are now in the position to prove our main result of this section.
\begin{theorem}
Let $N$ be a normal Coulomb frame for the conformally parametrized immersion $X\colon\overline B\to\mathbb R^{n+2}$ with total torsion ${\mathcal T}(N)$ and given $\|{\mathcal S}\|_{L^\infty(B)}.$ Assume that the smallness condition
  $$\frac{\sqrt{n-2}}2
    \left(
      \frac{n-2}{4\pi}\,{\mathcal T}(N)+C(n)\|\mathcal S\|_{L^\infty(B)}
    \right)<1$$
is satisfied with $C(n):=\min\big\{\frac{n(n-1)}{8},\sqrt2\big\}$. Then the torsion coefficients of $N$ can be estimated by
  $$\|T_{\sigma,i}^\vartheta\|_{L^\infty(B)}\le c,\quad i=1,2,\ 1\le\sigma<\vartheta\le n,$$
with a nonnegative constant $c=c(n,\|S\|_{L^\infty(B)},\mathcal T(N))<+\infty$.
\end{theorem}
\begin{proof}
We have the following elliptic system
 $$\begin{array}{l}
     \displaystyle
     |\Delta{\mathcal T}|
     \le\sqrt{\frac{n(n-1)}{2}}\,|\nabla{\mathcal T}|^2+|{\mathcal S}|\quad\mbox{in}\ B,
     \quad{\mathcal T}=0\quad\mbox{on}\ \partial B, \\[3ex]
     \displaystyle
     \|{\mathcal T}\|_{L^\infty(B)}
     \le\frac{n-2}{2\pi}\,\|\nabla{\mathcal T}\|_{L^2(B)}^2+C(n)\|{\mathcal S}\|_{L^\infty(B)}
     \le M\in[0,+\infty)\,.
   \end{array}$$
The smallness condition ensures that we can can apply Heinz's global gradient estimate Theorem\,1 in Sauvigny \cite{sauvigny_2005}, chapter XII, \S\,3, obtaining $\|\nabla\mathcal T\|_\infty\le c$.\footnote{This theorem states: Let $X\in C^2(\overline B)$ be a solution of the elliptic system $|\Delta X|\le a|\nabla X|^2+b$ in $B$ with $X=0$ on $\partial B$ and $\|X\|_{L^\infty(B)}\le M.$ Assume $aM<1.$ Then there is a constant $c=c(a,b,M,\alpha)$ such that $\|X\|_{C^{1+\alpha}(\overline B)}\le c(a,b,M,\alpha).$} This in turn yields the desired estimate.
\end{proof}
\begin{remark}
It remains open to prove global pointwise estimates for the torsion coefficients without the smallness condition. In particular, we would like to get rid of the a priori knowledge of ${\mathcal T}(N).$
\end{remark}
\vspace*{6ex}
\begin{center}
\rule{35ex}{0.1ex}
\end{center}
\vspace*{6ex}
\section{Bounds for the total torsion}
\label{section_boundstotaltorsion}
\subsection{Upper bounds}
\label{par_upperbounds}
From the torsion estimates above we can easily infer various upper bounds for the functional of total torsion: For example, the previous theorem yields an estimate of the form\index{total torsion}
  $${\mathcal T}(N)
    \le 2\sum_{1\le\sigma<\vartheta\le n}\,
         \int\hspace{-1.3ex}\int\limits_{\hspace{-1.8ex}B}
         \Big\{
           \|T_{\sigma,1}^\vartheta\|_{L^\infty(B)^2}
           +\|T_{\sigma,2}^\vartheta\|_{L^\infty(B)}
         \Big\}\,dudv
    =:C(n,\|{\mathcal S}\|_{L^\infty(B)},{\mathcal T}(N)).$$
Let us focus an alternative way for {\it small solutions} ${\mathcal T}:$ Namely, multiplying
  $$\Delta{\mathcal T}=-\delta{\mathcal T}+{\mathcal S}$$
by ${\mathcal T}$ and integrating by parts yields
\begin{proposition}
For small solutions $\|{\mathcal T}\|_{L^\infty(B)}\le\frac{2}{\sqrt{n-2}}$ it holds
  $$\mathcal T(N)
    =2\|\nabla\mathcal T\|_{L^2(B)}^2
    \le\frac{4\|\mathcal T\|_{L^\infty(B)}\|\mathcal S\|_{L^1(B)}}{2-\sqrt{n-2}\|{\mathcal T}\|_{L^\infty(B)}}\,.$$
\end{proposition}
\noindent
\begin{remark}
The reader is refered to Sauvigny \cite{sauvigny_2005} where such small solutions of nonlinear elliptic systems are constructed. Let us emphasize that the case $n=2$ is much easier to handle: The classical maximum principle controls $\|\mathcal T\|_{L^\infty(B)}$ in terms of $\|S\|_{L^\infty(B)},$ and no smallness condition is needed to bound the functional of total torsion.
\end{remark}
\subsection{A lower bound}
\label{par_lowerbounds}
Finally we complete this section with establishing a lower bound for the functional of total torsion.
\begin{proposition}
Let $N$ a normal Coulomb frame for the immersion $X.$ Assume that the curvature vector of its normal bundle satisfies ${\mathcal S}\not\equiv 0$ as well as $\|\nabla{\mathcal S}\|_{L^2(B)}>0.$ Then it holds\index{total torsion}
  $${\mathcal T}(N)
    \ge\left(
         \sqrt{n-2}\,\|\mathcal S\|_{L^\infty(B)}
         +\frac{\|\mathcal S\|_{L^2(B)}^{2}}{(1-\varrho)^2\|\mathcal S\|_{L^2(B_\varrho)}^2}
         +\frac{2\|\nabla\mathcal S\|_{L^2(B)}^{2}}{\|\mathcal S\|_{L^2(B_\varrho)}^2}
       \right)^{-1}\|\mathcal S\|_{L^2(B_\varrho)}^2>0$$
with $\varrho=\varrho(\mathcal S)\in(0,1)$ constructed in the proof given below, and $B_\varrho=\{(u,v)\in\mathbb R^2\,u^2+v^2<\varrho\}.$
\end{proposition}
\begin{proof}
\begin{itemize}
\item[1.]
Because of ${\mathcal S}\not=0$ there exists (a first) $\varrho=\varrho(\mathcal S)\in(0,1)$ such that
  $$\|\mathcal S\|_{L^2(B_\varrho)}
    =\left(\ \,\,
       \int\hspace{-1.5ex}\int\limits_{\hspace{-2.2ex}B_\varrho(0)}|{\mathcal S}|^2\,dudv
     \right)^{\frac12}>0,$$
and this is already our $\varrho$ from the assumption of the proposition. Now we choose a test function $\eta\in C^0(B,\mathbb R)\cap H_0^{1,2}(B)$ with the properties
  $$\eta\in[0,1]\quad\mbox{in}\ B,\quad
    \eta=1\quad\mbox{in}\ B_\varrho\,,\quad
    |\nabla\eta|\le\frac{1}{1-\varrho}\quad\mbox{in}\ B.$$
Multiplying $\Delta{\mathcal T}=-\delta{\mathcal T}+{\mathcal S}$ by $(\eta{\mathcal S})$ and integrating by parts yields
  $$\int\hspace{-1.3ex}\int\limits_{\hspace{-1.8ex}B}
    \nabla\mathcal T\cdot\nabla(\eta\mathcal S)\,dudv
    =\int\hspace{-1.3ex}\int\limits_{\hspace{-1.8ex}B}
     \eta\,\delta\mathcal T\cdot\mathcal S\,dudv
     -\int\hspace{-1.3ex}\int\limits_{\hspace{-1.8ex}B}
      \eta\,|\mathcal S|^2\,dudv.$$
Taking $|\delta{\mathcal T}|\le\sqrt{n-2}|{\mathcal T}_u||{\mathcal T}_v|$ (see section 18.2) into account, we can now estimate as follows:
  $$\begin{array}{lll}
      \displaystyle
      \int\hspace{-1.3ex}\int\limits_{\hspace{-1.8ex}B_\varrho}
      |{\mathcal S}|^2\,dudv\negthickspace
      & \le & \negthickspace\displaystyle
              \int\hspace{-1.3ex}\int\limits_{\hspace{-1.8ex}B}
              \eta\,|{\mathcal S}|^2\,dudv
              \,\le\,\int\hspace{-1.3ex}\int\limits_{\hspace{-1.8ex}B}
                     \eta\,\big|\delta{\mathcal T}\cdot{\mathcal S}\big|\,dudv
		    +\int\hspace{-1.3ex}\int\limits_{\hspace{-1.8ex}B}
                     \big|\nabla{\mathcal T}\cdot\nabla(\eta{\mathcal S})\big|\,dudv \\[5ex]
      & \le & \negthickspace\displaystyle
              \|{\mathcal S}\|_{L^\infty(B)}
              \int\hspace{-1.3ex}\int\limits_{\hspace{-1.8ex}B}
              \eta\,|\delta{\mathcal T}|\,dudv
              +\int\hspace{-1.3ex}\int\limits_{\hspace{-1.8ex}B}
               |\nabla\eta|\,|\mathcal S|\,|\nabla{\mathcal T}|\,dudv
              +\int\hspace{-1.3ex}\int\limits_{\hspace{-1.8ex}B}
               \eta\,|\nabla\mathcal S|\,|\nabla\mathcal T|\,dudv \\[5ex]
      & \le & \negthickspace\displaystyle
              \frac{\sqrt{n-2}}{2}\,\|{\mathcal S}\|_{L^\infty(B)}
              \int\hspace{-1.3ex}\int\limits_{\hspace{-1.8ex}B}
              |\nabla{\mathcal T}|^2\,dudv
              +\frac\varepsilon2
               \int\hspace{-1.3ex}\int\limits_{\hspace{-1.8ex}B}
               |\mathcal S|^2\,dudv
              +\frac1{2\varepsilon(1-\varrho)^2}
	       \int\hspace{-1.3ex}\int\limits_{\hspace{-1.8ex}B}
               |\nabla{\mathcal T}|^2\,dudv \\[5ex]
      &    & \negthickspace\displaystyle
             +\frac\delta2
              \int\hspace{-1.3ex}\int\limits_{\hspace{-1.8ex}B}
              |\nabla\mathcal S|^2\,dudv
             +\frac1{2\delta}
              \int\hspace{-1.3ex}\int\limits_{\hspace{-1.8ex}B}
	      |\nabla\mathcal T|^2\,dudv
    \end{array}$$
with arbitrary real numbers $\varepsilon,\delta>0$. Summarizing we arrive at
  $$\|\mathcal S\|_{L^2(B_\varrho)}^2
    \le\left(
         \frac{\sqrt{n-2}}{2}\,\|\mathcal S\|_{L^\infty(B)}
         +\frac{1}{2\varepsilon(1-\varrho)^2}
         +\frac1{2\delta}
      \right)\|\nabla\mathcal T\|_{L^2(B)}^2
      +\frac{\varepsilon}{2}\,\|\mathcal S\|_{L^2(B)}^2
      +\frac{\delta}{2}\,\|\nabla\mathcal S\|_{L^2(B)}^2\,.$$
\item[2.]
Now we choose $\varepsilon:$ Inserting $\varepsilon=\|\mathcal S\|_{L^2(B)}^{-2}\|\mathcal S\|_{L^2(B_\varrho)}^2>0$ and rearranging for $\|{\mathcal S}\|_{B_\varrho(B)}^2$ gives
  $$\|\mathcal S\|_{L^2(B_\varrho)}^2
    \le\left(
         \sqrt{n-2}\,\|\mathcal S\|_{L^\infty(B)}
         +\frac{\|\mathcal S\|_{L^2(B)}^{2}}{(1-\varrho)^2\|\mathcal S\|_{L^2(B_\varrho)}^2}
         +\frac{1}{\delta}
       \right)\|\nabla\mathcal T\|_{L^2(B)}^2
       +\delta\|\nabla\mathcal S\|_{L^2(B)}^2\,.$$
And since $\|\nabla{\mathcal S}\|_{L^2(B)}>0$ we can insert $\delta=\frac12\|\nabla\mathcal S\|_{L^2(B)}^{-2}\|\mathcal S\|_{L^2(B_\varrho)}^2$ which implies
  $$\|\mathcal S\|_{L^2(B_\varrho)}^2
    \le 2\left(
           \sqrt{n-2}\,\|\mathcal S\|_{L^\infty(B)}
           +\frac{\|\mathcal S\|_{L^2(B)}^{2}}{(1-\varrho)^2\|\mathcal S\|_{L^2(B_\varrho)}^2}
           +\frac{2\|\nabla\mathcal S\|_{L^2(B)}^2}{\|\mathcal S\|_{L^2(B_\varrho)}^2}
         \right)
         \|\nabla\mathcal T\|_{L^2(B)}^2\,.$$
Having $\mathcal T(N)=2\|\nabla\mathcal T\|_{L^2(B)}^2$ in mind we arrive at the stated estimate.
\end{itemize}
\end{proof}
\vspace*{6ex}
\begin{center}
\rule{35ex}{0.1ex}
\end{center}
\vspace*{6ex}
\section{Existence and regularity of weak normal Coulomb frames}
\label{section_existence}
\subsection{Regularity results for the homogeneous Poisson problem}
\label{par_poisson}
To introduce the function spaces coming next into play we consider the Dirichlet boundary value problem
  $$\Delta\phi(u,v)=r(u,v)\quad\mbox{in}\ B,\quad
    \phi(u,v)=0\quad\mbox{on}\ \partial B.$$
Let us abbreviate this problem by (DP).
\subsubsection{Schauder estimates}
Let $r\in C^\alpha(\overline B)$ hold true for the right hand side $r.$ Then there exist a classical solution of (DB) satisfying\index{Schauder estimates}
  $$\|\phi\|_{C^{2+\alpha}(\overline B)}\le C(\alpha)\|r\|_{C^\alpha(\overline B)}\,,$$
see e.g Gilbarg and Trudinger \cite{gilbarg_trudinger_1983}.
\subsubsection{$L^p$-estimates}
Let $r\in L^2(B)$ then any solution $\phi\in H^{1,2}(B)$ is of class $H^{2,2}(B),$ and it holds\index{$L^p$-estimates}
  $$\|\phi\|_{H^{2,2}(B)}\le C\big(\|\phi\|_{H^{1,2}(B)}+\|r\|_{L^2(B)}\big)\,.$$
In particular, if $r\in H^{m-2,2}(B)$ for the right hand side we have
  $$\|\phi\|_{H^{m,2}(B)}\le C\big(\|\phi\|_{H^{1,2}(B)}+\|r\|_{H^{m-2,2}(B)}\big).$$
For a detailed analysis we refer the reader to Dobrowolski \cite{dobrowolski_2006}, chapter 7. Note that we must require $r\in L^2(B)$ to infer higher regularity $\phi\in C^0(B)$ because $H^{2,2}(B)$ is continuously emedded in $C^0(B)$ by Sobolev's embedding theorem. Buf if, on the other hand, $r\in L^1(B)$ then any weak solution $\phi\in H^{1,2}(B)$ of (DB) satisfies
  $$\begin{array}{l}
      \|\phi\|_{L^q(B)}\le C\|r\|_{L^1(B)}\quad\mbox{for all}\ 1\le q<\infty\,, \\[2ex]
      \|\phi\|_{H^{1,p}(B)}\le C\|r\|_{L^1(B)}\quad\mbox{for all}\ 1\le p<2.
    \end{array}$$
In particular, a function $\phi\in H^{1,2}(B)$ is not necessarily continuous.
\subsubsection{Wente's $L^\infty$-estimate}
This situation changes dramatically if the side hand side $r$ possesses a certain algebraic structure. Namely assume\index{Wente's $L^\infty$-estimate}
  $$r=\frac{\partial a}{\partial u}\,\frac{\partial b}{\partial v}
      -\frac{\partial a}{\partial v}\,\frac{\partial b}{\partial u}$$
with functions $a,b\in H^{1,2}(B).$ Then again $r\in L^1(B)$ but any solution $\phi\in H^{1,2}(B)$ is of class $C^0(B)$ and satisfies Wente's $L^\infty$-estimate
  $$\|\phi\|_{L^\infty(B)}+\|\nabla\phi\|_{L^2(B)}
    \le\frac{1}{4\pi}\,\|\nabla a\|_{L^2(B)}\|\nabla b\|_{L^2(B)}\,;$$
see Wente \cite{wente_1980}. We already used this inequality in section 20.1 for establishing an upper bound for the total torsion of normal Coulomb frames. The idea of its proof of this inequality follows from an ingenious partial integration of the right hand side presented in $\mbox{curl}$-structure using the conformal invariance of the differential equation. Finally we approximate $a$ and $b$ by smooth functions $a_n,b_n\in C^\infty(B)$ which form a Cauchy sequence in $H{1,2}(B)\cap L^\infty(B)$ with continuous limit.
\subsubsection{Hardy spaces}
Wente's discovery is the starting point of the modern harmonic analysis. Its general framework is the concept of {\it Hardy spaces.} From Helein \cite{helein_2002} we quote two possible definitions of Hardy spaces leading to equivalent formulations.\index{Hardy spaces}
\begin{definition}
(Tempered-distribution definition)\\
Let $\Psi\in C_0^\infty(\mathbb R^m)$ such that\index{tempered distribution}
  $$\int\limits_{\mathbb R^m}\Psi(x)\,dx=1.$$
For each $\varepsilon>0$ we set
  $$\Psi_\varepsilon(x)=\frac{1}{\varepsilon^m}\,\Psi(\varepsilon^{-1}x),$$
and for $\phi\in L^1(\mathbb R^m)$ define
  $$\phi^*(x)=\sup_{\varepsilon>0}|(\Psi_\varepsilon\star\phi)(x)|.$$
Then $\phi$ belongs to ${\mathcal H}^1(\mathbb R^m)$ if and only if $\phi^*\in L^1(\mathbb R^m)$ with norm
  $$\|\phi\|_{{\mathcal H}^1(\mathbb R^m)}
    \le\|\phi\|_{L^1(\mathbb R^m)}+\|\phi^*\|_{L^1(\mathbb R^m)}\,.$$
\end{definition}
\begin{definition}
(Riesz-Fourier-transform definition)\\
For any function $\phi\in L^1(\mathbb R^m)$ we denote by $R_\alpha\phi$ the function defined by\index{Riesz-Fourier transform}
  $${\mathcal F}(R_\alpha\phi)
    =\frac{\xi_\alpha}{|\xi|}\,{\mathcal F}(\phi)(\xi)$$
with the $\phi$-Fourier transform
  $${\mathcal F}(\phi)(\xi)
    =\frac{1}{(2\pi)^\frac{m}{2}}\,
     \int\limits_{\mathbb R^m}e^{-i x\cdot\xi}\phi(x)\,dx.$$
Then $\phi$ belongs to ${\mathcal H}^1(\mathbb R^m)$ if and only if
  $$R_\alpha\phi\in L^1(\mathbb R^m)\quad\mbox{for all}\ \alpha=1,\ldots,m$$
with norm
  $$\|\phi\|_{{\mathcal H}^1(\mathbb R^m)}
    =\|\phi\|_{L^1(\mathbb R^m)}
     +\sum_{\alpha=1}^m\|R_\alpha\phi\|_{L^1(\mathbb R^m)}\,.$$
\end{definition}
\noindent
For a profound and comprehensive presentation of harmonic analysis we want to refer the reader to Stein's monograph \cite{stein_1993}.\\[1ex]
Consider again our Dirichlet problem (DP). Let $a,b\in H^{1,2}(B),$ and consider its extensions $a\mapsto\widehat a$ and $b\mapsto\widehat b$ in $H^{1,2}$ to the whole space $\mathbb R^2$ such that these mappings are continuous from $H^{1,2}(B)$ to $H^{1,2}(\mathbb R^2).$ Then, referring again to Helein \cite{helein_2002}, $v\in{\mathcal H}(\mathbb R^2)!$
\subsubsection{Lorentz interpolation spaces}
This fact becomes especially important now. Let us start with
\begin{definition}
Let $\Omega\subset\mathbb R^m$ be open, and let $p\in(1,+\infty)$ and $q\in[1,+\infty].$ The Lorentz space $L^{(p,q)}(\Omega)$ is the set of measurable functions $\phi\colon\Omega\to\mathbb R$ such that
  $$\|f\|_{L^{(p,q)}}
    :=\left(\ \int\limits_0^\infty\big\{t^\frac{1}{p}\phi^*(t)\big\}^q\,\frac{dt}{t}\right)^\frac{1}{q}<\infty
    \quad\mbox{if}\ q<+\infty$$
or
  $$\|f\|_{L^{(p,q)}}:=\sup_{t>0}\,t^\frac{1}{p}\phi^*(t)<\infty
    \quad\mbox{if}\ q=+\infty.$$
Here $\phi^*$ denotes the unique non-increasing rearrangement of $|\phi|$ on $[0,\mbox{meas}\,\Omega].$
\end{definition}
\noindent
Lorentz spaces are Banach spaces with a suitable norm. They may be considered as a deformation of $L^p.$ Note that
  $$L^{(p,p)}(B)=L^p(B)\,,\quad
    L^{(p,1)}(B)\subset L^{(p,q')}(B)\subset L^{(p,q'')}(B)\subset L^{(p,\infty)}(B)$$
for $1<q'<q''.$ Then
\begin{itemize}
\item[(i)]
if $\phi\in H^{1,2}(B)$ solves (DP) with $r\in{\mathcal H}^1(B)$ then $\frac{\partial\phi}{\partial x},\frac{\partial\phi}{\partial y}\in L^{(2,1)}(B);$
\vspace*{-1ex}
\item[(i)]
if $\phi\in H^{1,2}(B)$ with $\frac{\partial\phi}{\partial x},\frac{\partial\phi}{\partial y}\in L^{(2,1)}(B)$ then $\phi\in C^0(B).$
\end{itemize}
\subsubsection{The general regularity result}
Summarizing we can state the following regularity result from Helein \cite{helein_2002}, chapter 3.
\begin{proposition}
Let $a,b\in H^{1,2}(B),$ and assume $\phi\in H^{1,2}(B)$ solves (DB). Then $\frac{\partial\phi}{\partial u},\frac{\partial\phi}{\partial v}\in L^{(2,1)}(B),$ and in particular $\phi\in C^0(\overline B).$
\end{proposition}
\subsection{Existence of weak normal Coulomb frames}
\label{par_weakexistence}
In case $n=2$ we constructed critical points of the functional of total torsion solving the Euler-Lagrange equation explicitely and varified its minimal character. In the general situation now we construct critical points by means of direct methods of the calculus of variations. We start with the following
\begin{definition}
Let $m\in\mathbb N,$ $m\ge 2.$ For two matrices $A,B\in{\mathbb R}^{m\times m}$ we define their inner product
  $$\langle A,B\rangle
    =\mbox{\rm trace}\,(A\circ B^t)
    =\sum_{\sigma,\vartheta=1}^m
     A_\sigma^\vartheta B_\sigma^\vartheta$$
and the associated norm
  $$|A|=\sqrt{\langle A,A\rangle}
       =\left(\ \sum_{\sigma,\vartheta=1}^m(A_\sigma^\vartheta)^2\right)^\frac{1}{2}\,.$$
\end{definition}
\noindent
Helein proved in \cite{helein_2002}, Lemma 4.1.3, existence of weak Coulomb frames in the tangent bundle of surfaces. We want to carry out his arguments and adapt his methods to our situation. Additionally we want to present classical regularity of weak normal Coulomb frames. We always use conformal parameters $(u,v)\in\overline B.$
\begin{proposition}\label{prop_exist}
There exists a weak normal Coulomb frame $N\in H^{1,2}(B)\cap L^\infty(B)$ minimizing the functional $\mathcal T(N)$ of total torsion in the set of all weak normal frames of class $H^{1,2}(B)\cap L^\infty(B)$.\index{normal Coulomb frame}
\end{proposition}
\begin{proof}
We fix\footnote{Note that now we start from $\widetilde N$ and transform into $N.$} some normal frame $\widetilde N\in C^{k-1,\alpha}(\overline B)$ and interpret $\mathcal T(N)$ as a functional ${\mathcal F}(R)$ of $SO(n)$-regular rotations $R=(R_\sigma^\vartheta)_{\sigma,\vartheta=1,\ldots,n}\in H^{1,2}(B,SO(n))$ by setting
  $$\mathcal F(R)=\sum_{\sigma,\vartheta=1}^n\sum_{i=1}^2\int\hspace{-1.3ex}\int\limits_{\hspace{-1.8ex}B}
    (T_{\sigma,i}^\vartheta)^2\,dudv=\int\hspace{-1.3ex}\int\limits_{\hspace{-1.8ex}B}\big(|T_1|^2+|T_2|^2\big)\,dudv,
    \quad N_\sigma:=\sum_{\vartheta=1}^nR_\sigma^\vartheta\widetilde N_\vartheta\,,$$
setting $T_i=(T_{\sigma,i}^\vartheta)_{\sigma,\vartheta=1,\ldots,n}.$ Choose a minimizing sequence ${}^{\scriptscriptstyle \ell}\!R=({}^{\scriptscriptstyle \ell}\!R_\sigma^\vartheta)_{\sigma,\vartheta=1,\ldots,n}\in H^{1,2}(B,SO(n))$ and define ${}^{\scriptscriptstyle \ell}\!N_\sigma:=\sum\limits_{\vartheta=1}^n{}^{\scriptscriptstyle \ell}\!R_\sigma^\vartheta\widetilde N_\vartheta$. As in paragraph \ref{par_curvaturematrices} we find ${}^{\scriptscriptstyle \ell}T_i={}^{\scriptscriptstyle \ell}\!R_{u^i}\circ{}^{\scriptscriptstyle \ell}\!R^t+{}^{\scriptscriptstyle \ell}\!R\circ \widetilde T_i\circ{}^{\scriptscriptstyle \ell}\!R^t$\footnote{Note that the proof of this identity remains true for $R\in H^{2,1}(B,SO(n))\cap H^{1,2}(B,SO(n)).$} which implies
  $$\begin{array}{lll}
      {}^{\scriptscriptstyle \ell}T_i\circ{}^{\scriptscriptstyle \ell}T_i^t\negthickspace
      & = & \displaystyle\negthickspace
            ({}^{\scriptscriptstyle \ell}\!R_{u^i}\circ
             {}^{\scriptscriptstyle \ell}\!R^t
             +{}^{\scriptscriptstyle \ell}\!R\circ
              \widetilde T_i\circ
              {}^{\scriptscriptstyle \ell}\!R^t)\circ
            ({}^{\scriptscriptstyle \ell}\!R\circ
             {}^{\scriptscriptstyle \ell}\!R_{u^i}^t
             +{}^{\scriptscriptstyle \ell}\!R\circ
              \widetilde T_i^t\circ
              {}^{\scriptscriptstyle \ell}\!R^t) \\[2ex]
      & = & \displaystyle\negthickspace
            {}^{\scriptscriptstyle \ell}\!R_{u^i}\circ
            {}^{\scriptscriptstyle \ell}\!R_{u^i}^t
            +{}^{\scriptscriptstyle \ell}\!R\circ
             \widetilde T_i\circ
             {}^{\scriptscriptstyle \ell}\!R_{u^i}^t
            +{}^{\scriptscriptstyle \ell}\!R_{u^i}\circ
             \widetilde T_i^t\circ
             {}^{\scriptscriptstyle \ell}\!R^t
            +{}^{\scriptscriptstyle \ell}\!R\circ
             \widetilde T_i\circ
             \widetilde T_i^t\circ
             {}^{\scriptscriptstyle \ell}\!R^t.
    \end{array}$$
In particular, we conclude
  $$\mbox{trace}\,({}^{\scriptscriptstyle \ell}T_i
    \circ{}^{\scriptscriptstyle \ell}T_i^t)
    =\mbox{trace}\,
     ({}^{\scriptscriptstyle \ell}\!R_{u^i}
     \circ{}^{\scriptscriptstyle \ell}\!R_{u^i}^t)
     +2\,\mbox{trace}\,
      ({}^{\scriptscriptstyle \ell}\!R
      \circ \widetilde T_i\circ{}^{\scriptscriptstyle \ell}\!R_{u^i}^t)
     +\mbox{trace}\,(\widetilde T_i\circ \widetilde T_i^t)$$
or using our notion of a matrix norm
\begin{equation}\label{torsion_sequence}
  |{}^{\scriptscriptstyle \ell}T_i|^2
  =|{}^{\scriptscriptstyle \ell}\!R_{u^i}|^2
   +2\langle{}^{\scriptscriptstyle \ell}\!R\circ\widetilde T_i,{}^{\scriptscriptstyle \ell}\!R_{u^i}\rangle
   +|\widetilde T_i|^2.
\end{equation}
Furthermore, taking $|{}^{\scriptscriptstyle \ell}\!R\circ\widetilde T_i|=|\widetilde T_i|$ into account, we arrive at the estimate
  $$|{}^{\scriptscriptstyle \ell}T_i|^2
    \ge\big(|\widetilde T_i|-|{}^{\scriptscriptstyle \ell}\!R_{u^i}|\big)^2\quad\mbox{a.e.~on}\ B, 
    \quad\mbox{for all}\ \ell\in\mathbb N.$$
Now because the $\widetilde T_i$ are bounded in $L^2(B),$ and since ${}^{\scriptscriptstyle \ell}\!R$ is minimizing for $\mathcal F$, the sequences ${}^{\scriptscriptstyle \ell}T_i$ are also bounded in $L^2(B).$ Thus, ${}^{\scriptscriptstyle \ell}\!R_{u^i}$ are bounded sequences in $L^2(B)$ in accordance with the last inequality. By Hilbert's selection theorem\index{Hilbert's selection theorem} and Rellich's embedding theorem\index{Rellich's embedding theorem} we find a subsequence, again denoted by ${}^{\scriptscriptstyle \ell}\!R$, which converges as follows:
  $${}^{\scriptscriptstyle \ell}\!R_{u^i}
    \rightharpoonup R_{u^i}
    \quad\mbox{weakly in}\ L^2(B,\mathbb R^{n\times n}),\quad
    {}^{\scriptscriptstyle \ell}\!R\rightarrow R
    \quad\mbox{strongly in}\ L^2(B,SO(n))$$
with some $R\in H^{1,2}(B,SO(n))$. In particular, going if necessary to a subsequence, we have ${}^{\scriptscriptstyle \ell}\!R\to R$ a.e. on $B$ as well as
  $$\lim_{\ell\to\infty}
	\int\hspace{-1.3ex}\int\limits_{\hspace{-1.8ex}B}
    |{}^{\scriptscriptstyle \ell}\!R\circ \widetilde T_i-R\circ\widetilde T_i|^2\,dudv=0$$
according to the dominated convergence theorem.\index{dominated convergence theorem} Hence we can compute in the limit
  $$\begin{array}{rcl}
      \displaystyle\lim_{\ell\to\infty}
      \int\hspace{-1.3ex}\int\limits_{\hspace{-1.8ex}B}
      \langle
        {}^{\scriptscriptstyle \ell}\!R\circ\widetilde T_i,{}^{\scriptscriptstyle \ell}\!R_{u^i}
      \rangle\,dudv
      & = & \displaystyle\lim_{\ell\to\infty}
            \Bigg(\,\int\hspace{-1.3ex}\int\limits_{\hspace{-1.8ex}B}
            \langle
              {}^{\scriptscriptstyle \ell}\!R\circ\widetilde T_i-R\circ\widetilde T_i,{}^{\scriptscriptstyle \ell}\!R_{u^i}
            \rangle\,dudv
            +\int\hspace{-1.3ex}\int\limits_{\hspace{-1.8ex}B}
             \langle R\circ\widetilde T_i,{}^{\scriptscriptstyle \ell} R_{u^i}
             \rangle\,dudv\!\Bigg)\\[4ex]
      & = & \displaystyle  
            \int\hspace{-1.3ex}\int\limits_{\hspace{-1.8ex}B}
            \langle R\circ\widetilde T_i,R_{u^i}
            \rangle\,dudv.
    \end{array}$$
In addition, we obtain
\begin{equation*}
  \lim_{\ell\to\infty}
  \int\hspace{-1.3ex}\int\limits_{\hspace{-1.8ex}B}
  |{}^{\scriptscriptstyle \ell}\!R_{u^i}|^2\,dudv
  \ge\int\hspace{-1.3ex}\int\limits_{\hspace{-1.8ex}B}
     |R_{u^i}|^2\,dudv
\end{equation*}
due to the semicontinuity of the $L^2$-norm w.r.t. weak convergence. Putting the last two relations into the identity $|{}^{\scriptscriptstyle \ell}T_i|^2=|{}^{\scriptscriptstyle \ell}\!R_{u^i}|^2+2\langle{}^{\scriptscriptstyle \ell}\!R\circ\widetilde T_i,{}^{\scriptscriptstyle \ell}\!R_{u^i}\rangle+|\widetilde T_i|^2$ we finally infer
  $$\begin{array}{lll}
      \displaystyle
      \lim_{\ell\to\infty}\mathcal F({}^{\scriptscriptstyle \ell}\!R)\negthickspace
      &  =  & \negthickspace\displaystyle
              \lim_{\ell\to\infty}
              \int\hspace{-1.3ex}\int\limits_{\hspace{-1.8ex}B}
              \big(
                |{}^{\scriptscriptstyle \ell}T_1|^2+|{}^{\scriptscriptstyle \ell}T_2|^2
              \big)\,dudv \\[5ex]
      & \ge & \negthickspace\displaystyle
              \int\hspace{-1.3ex}\int\limits_{\hspace{-1.8ex}B}
              \big(
                |R_u|^2+|R_v|^2
              \big)\,dudv
              +2\int\hspace{-1.3ex}\int\limits_{\hspace{-1.8ex}B}
                \big(
                  \langle R\circ\widetilde T_1,R_u\rangle
                  +\langle R\circ\widetilde T_2,R_v\rangle
                \big)\,dudv \\[5ex]
      &     & \negthickspace\displaystyle
              +\,\int\hspace{-1.3ex}\int\limits_{\hspace{-1.8ex}B}
                 \big(
                   |\widetilde T_1|^2+|\widetilde T_2|^2
                 \big)\,dudv \\[5ex]
      &  =  & \negthickspace\displaystyle
              \int\hspace{-1.3ex}\int\limits_{\hspace{-1.8ex}B}
              \big(
                |T_1|^2+|T_2|^2
              \big)\,dudv
              \,=\,\mathcal F(R)
    \end{array}$$
where $T_i=(T_{\sigma,i}^\vartheta)_{\sigma,\vartheta=1,\ldots,n}$ denote the torsion coefficients of the frame $N$ with entries $N_\sigma:=\sum_\vartheta R_\sigma^\vartheta\widetilde N_\vartheta.$ Consequently, $N\in H^{1,2}(B)\cap L^\infty(B)$ minimizes $\mathcal T_X$ and, in particular, it is a weak normal Coulomb frame.
\end{proof}
\subsection{$H_{loc}^{2,1}$-regularity of weak normal Coulomb frames}
\label{par_Hframes}
To prove higher regularity of normal Coulomb frames we make essential use of techniques from harmonic analysis. Consult, if necessary, paragraph \ref{par_poisson}. We always use conformal parameters $(u,v)\in\overline B.$
\begin{proposition}
Any weak normal Coulomb frame $N\in H^{1,2}(B)\cap L^\infty(B)$ belongs to the class $H^{2,1}_{loc}(B)$.\index{normal Coulomb frame}
\end{proposition}
\begin{proof}
\begin{enumerate}
\item
The torsion coefficients $T_{\sigma,i}^\vartheta$ of the normal Coulomb frame $N$ are weak solutions of the Euler-Lagrange equations
  $$\mbox{div}\,(T_{\sigma,1}^\vartheta,T_{\sigma,2}^\vartheta)=0\quad\mbox{in}\ B$$
for all $\sigma,\vartheta=1,\ldots,n.$ Hence, by a weak version of Poincare's lemma\index{Poincar\'e's lemma} (see e.g. Bourgain, Brezis and Mironescu \cite{bourgain_brezis_mironescu_2000} Lemma 3), there exist integral functions $\tau_\sigma^\vartheta\in H^{1,2}(B)$ satisfying
  $$\tau_{\sigma,u}^\vartheta=-T_{\sigma,2}^\vartheta\,,\quad
    \tau_{\sigma,v}^\vartheta=T_{\sigma,1}^\vartheta
    \quad\mbox{in}\ B$$
in weak sense. Thus the weak form of the Euler-Lagrange equations can be written as
  $$0=\int\hspace{-1.3ex}\int\limits_{\hspace{-1.8ex}B}
      \big\{\varphi_{u}\tau_{\sigma,v}^\vartheta-\varphi_{v}\tau_{\sigma,u}^\vartheta\big\}\,dudv
     =\int\limits_{\partial B}
      \tau_\sigma^\vartheta\frac{\partial\varphi}{\partial t}\,ds
    \quad\mbox{for all}\ \varphi\in C^\infty(\overline B),$$
where $\frac{\partial\varphi}{\partial t}$ denotes the tangential derivative of $\varphi$ along $\partial B.$ Note that $\tau_\sigma^\vartheta|_{\partial B}$ means the $L^2$-trace of $\tau_\sigma^\vartheta$ on the boundary curve $\partial B$ (see e.g. Alt \cite{alt_2006}, chapter 6, appendix A6.6)\footnote{Let $1\le p\le\infty.$ Then there is a uniquely determined map $S\colon H^{1,p}(B)\to L^p(B)$ such that $\|S(\phi)\|_{L^{p}(\partial B)}\le C\|\phi\|_{H^{1,2}(B)}.$ Additionally, it holds $S(\phi)=\phi\big|_{\partial B}$ if $\phi\in H^{1,2}(B)\cap C^0(\overline B)$). The map $S$ is called the trace mapping.}. Consequently, the lemma of DuBois-Reymond\footnote{Its one-dimensional version is the following: Let $f\in L^1([a,b]),$ and assume that $\int_a^bf(x)\varphi(x)\,dx=0$ for all $\varphi\in C^\infty([a,b]).$ Then $f\equiv\mbox{const}$ almost everywhere.} yields $\tau_\sigma^\vartheta\equiv\mbox{const}$ on $\partial B$, and by translation we arrive at the boundary conditions
  $$\tau_\sigma^\vartheta=0\quad\mbox{on}\ \partial B\,.$$
\item
Thus the integral functions $\tau_\sigma^\vartheta$ are weak solutions of the second-order system\index{integral function}
  $$\Delta\tau_\sigma^\vartheta
    =-T_{\sigma,2,u}^\vartheta+T_{\sigma,1,v}^\vartheta
    =-\langle N_{\sigma,v},N_{\vartheta,u}\rangle
     +\langle N_{\sigma,u},N_{\vartheta,v}\rangle\quad\mbox{in}\ B$$
where the second identity follows by direct differentiation. By a result Coifman, Lions, Meyer and Semmes \cite{coifman_lions_meyer_semmes_1993},\footnote{Let $\phi\colon H^{1,2}(\mathbb R^2).$ Then $f:=\mbox{det}\,(\nabla\phi)\in{\mathcal H}^1(\mathbb R^2),$ and it holds $\|f\|_{{\mathcal H}^1(\mathbb R^2)}\le C\|\phi\|_{H^{1,2}(\mathbb R^2)}\,;$ see section 22.1.} the right-hand side of div-curl type belongs to the Hardy space\index{Hardy spaces} $\mathcal H_{loc}^1(B)$ and, hence, the $\tau_\sigma^\vartheta$ belong to $H_{loc}^{2,1}(B)$ by Fefferman and Stein \cite{fefferman_stein_1972}.\footnote{Let $\phi\in L^1(\mathbb R^2)$ be a solution of $-\Delta\phi=f\in{\mathcal H}^1(\mathbb R^2).$ Then all second derivatives of $\phi$ belong to $L^1(\mathbb R^2),$ and it holds $\|\phi_{x^\alpha x^\beta}\|_{L^1(\mathbb R^m)}\le C\|f\|_{L^1(\mathbb R^m)}$ for all $\alpha,\beta=1,2.$} Consequently we find $T_{\sigma,i}^\vartheta\in H^{1,1}_{loc}(B)\cap L^2(B).$ Next, we employ the Weingarten equations
  $$N_{\sigma,u^i}
    =-\sum_{j,k=1}^2L_{\sigma,ij}g^{jk}\,X_{u^k}
     +\sum_{\vartheta=1}^nT_{\sigma,i}^\vartheta N_\vartheta$$
in a weak form. For the coefficients of the second fundamental form we have $L_{\sigma,ij}=N_\sigma\cdot X_{u^iu^j}$ leading to $L_{\sigma,ij}\in H^{1,2}(B)$ taking account of $N\in H^{1,2}(B)\cap L^\infty(B)$ and $X_{u^iu^j}\in L^\infty(B).$ Hence we arrive at $N_{\sigma,u^i}\in H_{loc}^{1,1}(B)$ and $N\in H_{loc}^{2,1}(B)$ for our weak normal Coulomb frame. Note that $T_{\sigma,i}^\vartheta\in H^{1,1}_{loc}(B)\cap L^2(B)$ and $N_\vartheta\in H^{1,2}(B)\cap L^\infty(B)$ imply $T_{\sigma,i}^\vartheta N_\vartheta\in H^{1,1}_{loc}(B)$ by a careful adaption of the classical product rule in Sobolev spaces which is explained in the lemma below. We have proved the statement.
\end{enumerate}
\end{proof}
\noindent
So let us catch up on the following lemma to complete the proof.
\begin{lemma}
It holds
  $$\sum_{\vartheta=1}^n
    T_{\sigma,i}^\vartheta N_\vartheta
    \in H^{1,1}(B).$$
\end{lemma}
\begin{proof}
Note that $T_{\sigma,i}^\vartheta N_\vartheta\in L^2(B)$ because $T_{\sigma,i}^\vartheta\in L^2(B)$ and $N_\vartheta\in L^\infty(B).$ We show that $T_{\sigma,i}^\vartheta N_\vartheta$ has a weak derivative, i.e. we prove that
  $$T_{\sigma,i}^\vartheta N_{\vartheta,u^j}+N_\vartheta T_{\sigma,i,u^j}^\vartheta\in L^1(B)$$
is the weak derivative of $T_{\sigma,i}^\vartheta N_\vartheta.$ In other words
  $$\int\hspace{-1.3ex}\int\limits_{\hspace{-1.8ex}B}
    (T_{\sigma,i}^\vartheta N_{\vartheta,u^j}+N_\vartheta T_{\sigma,i,u^j}^\vartheta)\varphi\,dudv
    =-\int\hspace{-1.3ex}\int\limits_{\hspace{-1.8ex}B}
      (T_{\sigma,i}^\vartheta N_\vartheta)\varphi_{u^j}\,dudv$$
for all $\varphi\in C_0^\infty(B).$ For such a test function $\varphi$ define $\psi=T_{\sigma,i}^\vartheta\varphi\in W_0^{1,1}\cap L^2(B),$ and we claim
  $$\int\hspace{-1.3ex}\int\limits_{\hspace{-1.8ex}B}
    N_{u^j}\psi\,dudv
    =-\int\hspace{-1.3ex}\int\limits_{\hspace{-1.8ex}B}
      N\psi_{u^j}\,dudv.$$
For the proof of this relation we approximate $\psi$ with smooth functions $\psi^\varepsilon\in C_0^\infty(B)$ in the sense of Friedrichs.\index{Friedrich approximation} Then $\psi^\varepsilon\to\psi$ in $H^{1,1}(B)\cap L^2(B),$ and $\psi=0$ outside some compact set $K\subset\subset B.$ We estimate as follows
  $$\begin{array}{lll}
      \displaystyle
      \left|\,
        \int\hspace{-1.3ex}\int\limits_{\hspace{-1.8ex}B}
        (N_{\vartheta,u^j}\psi+N_\vartheta\psi_{u^j})\,dudv
      \right|\negthickspace
      &  =  & \negthickspace\displaystyle
              \left|\,
                \int\hspace*{-1.3ex}\int\limits_{\hspace{-1.8ex}B}
                (N_{\vartheta,u^j}\psi^\varepsilon+N_\vartheta\psi_{u^j}^\varepsilon)\,dudv
              \right|
              +\left|\,
                 \int\hspace*{-1.3ex}\int\limits_{\hspace{-1.8ex}B}
                 N_{\vartheta,u^j}(\psi-\psi^\varepsilon)\,dudv
               \right| \\[5ex]
      &  =  & \displaystyle\negthickspace
              +\left|\,
                 \int\hspace*{-1.3ex}\int\limits_{\hspace{-1.8ex}B}
                 N_\vartheta(\psi_{u^j}^\varepsilon-\psi_{u^j})\,dudv
               \right| \\[5ex]
      & \le & \negthickspace\displaystyle
              \|N_{\vartheta,u^j}\|_{L^2(B)}\|\psi-\psi^\varepsilon\|_{L^2(B)}
              +\|N_\vartheta\|_{L^\infty(B)}\|\psi_{u^j}^\varepsilon-\psi_{u^j}\|_{L^1(B)}
    \end{array}$$
taking
  $$\int\hspace{-1.3ex}\int\limits_{\hspace{-1.8ex}B}
    N_{u^j}\psi^\varepsilon\,dudv
    =-\int\hspace{-1.3ex}\int\limits_{\hspace{-1.8ex}B}
      N\psi_{u^j}^\varepsilon\,dudv$$
into account. Because $\|\psi-\psi^\varepsilon\|_{L^2(B)}\to 0$ and $\|\psi_{u^j}^\varepsilon-\psi_{u^j}\|_{L^1(B)}\to 0$ for $\varepsilon\to 0$ we arrive at the stated identity. Now we use the product rule and calculate
  $$\begin{array}{lll}
      \displaystyle
      \int\hspace{-1.3ex}\int\limits_{\hspace{-1.8ex}B}
      (T_{\sigma,i}^\vartheta N_{\vartheta,u^j}+N_\vartheta T_{\sigma,i,u^j}^\vartheta)\varphi\,dudv\negthickspace
      & = & \negthickspace\displaystyle
            \int\hspace{-1.3ex}\int\limits_{\hspace{-1.8ex}B}
            N_{\vartheta,u^j}\psi\,dudv
            +\int\hspace{-1.3ex}\int\limits_{\hspace{-1.8ex}B}
             N_\vartheta T_{\sigma,i,u^j}^\vartheta\varphi\,dudv \\[5ex]
      & = & \negthickspace\displaystyle
            -\int\hspace{-1.3ex}\int\limits_{\hspace{-1.8ex}B}
             (T_{\sigma,i}^\vartheta\varphi)_{u^j}N_\vartheta\,dudv
            +\int\hspace{-1.3ex}\int\limits_{\hspace{-1.8ex}B}
             N_\vartheta T_{\sigma,i,u^j}^\vartheta\varphi\,dudv \\[5ex]
      & = & \negthickspace\displaystyle
            -\int\hspace{-1.3ex}\int\limits_{\hspace{-1.8ex}B}
             T_{\sigma,i,u^j}^\vartheta N_\vartheta\varphi\,dudv
            -\int\hspace{-1.3ex}\int\limits_{\hspace{-1.8ex}B}
             T_{\sigma,i}^\vartheta\varphi_{u^j}N\,dudv
            +\int\hspace{-1.3ex}\int\limits_{\hspace{-1.8ex}B}
             N_\vartheta T_{\sigma,i,u^j}^\vartheta\varphi\,dudv \\[5ex]
      & = & \negthickspace\displaystyle
            -\int\hspace{-1.3ex}\int\limits_{\hspace{-1.8ex}B}
             (T_{\sigma,i}^\vartheta N_\vartheta)\varphi_{u^j}\,dudv.
    \end{array}$$
This proves the statement.
\end{proof}
\vspace*{6ex}
\begin{center}
\rule{35ex}{0.1ex}
\end{center}
\vspace*{6ex}
\section{Classical regularity of normal Coulomb frames}
\label{section_classical}
Our main result of this chapter is the proof of classical regularity of normal Coulomb frames. Up to now we only know regularity in the sense of Sobolev spaces. An essential tool on our road to regularity are again the Weingarten equations from Lecture I.
\begin{theorem}
For any conformally parametrized immersion $X\in C^{k,\alpha}(\overline B,\mathbb R^{n+2})$ with $k\ge3$ and $\alpha\in(0,1)$ there exists a normal Coulomb frame $N\in C^{k-1,\alpha}(\overline B,\mathbb R^{(n+2)\times n})$ minimizing $\mathcal T(N).$\index{normal Coulomb frame}
\end{theorem}
\begin{proof}
\begin{itemize}
\item[1.]
We fix some normal frame $\widetilde N\in C^{k-1,\alpha}(\overline B)$ and construct a weak normal Coulomb frame $N\in W^{1,2}(B)\cap L^\infty(B).$ Furthermore, we then know $N\in W^{2,1}_{loc}(B)$. Defining the orthogonal mapping $R=(R_\sigma^\vartheta)_{\sigma,\vartheta=1,\ldots,n}$ by $R_\sigma^\vartheta:=\langle N_\sigma,\widetilde N_\vartheta\rangle,$
we thus find 
	$$N_\sigma=\sum\limits_{\vartheta=1}^nR_\sigma^\vartheta\widetilde N_\vartheta\quad\mbox{and}\quad R\in W^{2,1}_{loc}(B,SO(n))
      \cap W^{1,2}(B,SO(n))\,.$$
In particular, we can assign a curvature tensor $S_{12}=(S_{\sigma,12}^\vartheta)_{\sigma,\vartheta=1,\ldots,n}\in L^1_{loc}(B)$ to $N$ by formula
  $$S_{\sigma,12}^\vartheta
    =T_{\sigma,1,v}^2-T_{\sigma,2,u}^\vartheta
     +\sum_{\omega=1}^n
      (T_{\sigma,1}^\omega T_{\omega,2}^\vartheta-T_{\sigma,2}^\omega T_{\omega,1}^\vartheta),$$
and from the previous section we infer $S_{12}\in L^\infty(B).$
\item[2.]
Introduce $\tau=(\tau_\sigma^\vartheta)_{\sigma,\vartheta=1,\ldots,n}\in W^{1,2}(B)$ by
  $$\begin{array}{l}
      \tau_{\sigma,u}^\vartheta=-T_{\sigma,2}^\vartheta\,,\quad
      \tau_{\sigma,v}^\vartheta=T_{\sigma,1}^\vartheta
      \quad\mbox{in}\ B, \\[2ex]
      \tau_\sigma^\vartheta=0\quad\mbox{on}\ \partial B.
    \end{array}$$
The definition of the normal curvature tensor gives us the nonlinear elliptic system
  $$\Delta\tau_\sigma^\vartheta
    =-\,\sum_{\omega=1}^n
      (\tau_{\sigma,u}^\omega\tau_{\omega,v}^\vartheta-\tau_{\sigma,v}^\omega\tau_{\omega,u}^\vartheta)
      +S_{\sigma,12}^\vartheta\quad\mbox{in}\ B,\quad
     \tau_\sigma^\vartheta=0\quad\mbox{on}\ \partial B.$$
On account of $S_{12}=(S_{\sigma,12}^\vartheta)_{\sigma,\vartheta=1,\ldots,n}\in L^\infty(B),$ a part of Wente's inequality yields $\tau\in C^0(\overline B),$ see e.g. Brezis and Coron \cite{brezis_coron_1984}; compare also Rivi\`ere \cite{riviere_2007} and the corresponding boundary regularity theorem in M\"uller and Schikorra \cite{mueller_schikorra_2008} for more general results. By appropriate reflection of $\tau$ and $S_{12}$ (the reflected quantities are again denoted by $\tau$ and $S_{12}$) we obtain a weak solution $\tau\in W^{1,2}(B_{1+d})\cap C^0(B_{1+d})$ of
  $$\Delta\tau
    =f(w,\nabla\tau)
    \quad\mbox{in}\ B_{1+d}:=\{w\in\mathbb R^2\,:\,|w|<1+d\}$$
with some $d>0$ and a right-hand side $f$ satisfying
  $$|f(w,p)|\le a|p|^2+b\quad\mbox{for all}\ p\in\mathbb R^{2n^2}\,,\quad w\in B_{1+d}$$
with some reals $a,b>0.$ Now, applying Tomi's regularity result\index{Tomi's regularity result} from \cite{tomi} for weak solutions of this non-linear system possessing small variation locally in $B_{1+d}$, we find $\tau\in C^{1,\nu}(\overline B)$ for any $\nu\in(0,1)$ (note that Tomi's result applies for such systems with $b=0,$ but his proof can easily be adapted to our inhomogeneous case $b>0$).
\item[3.]
From the first-order system for $\tau_\sigma^\vartheta$ we infer $T_i\in C^\alpha(\overline B).$ Thus, the Weingarten equations\index{Weingarten equations} for $N_{\sigma,u^i}$ yield $N\in W^{1,\infty}(B)$ on account of $N\in L^\infty(B),$ and we obtain $N\in C^\alpha(\overline B)$ by Sobolev's embedding theorem.\index{Sobolev embedding theorem} Inserting this again into the Weingarten equations, we find $N\in C^{1,\alpha}(\overline B)$. Hence, we can conclude $R=(\langle N_\sigma,\widetilde N_\vartheta\rangle)_{\sigma,\vartheta=1,\ldots,n}\in C^{1,\alpha}(\overline B),$ and transformation rule $S_{12}=R\circ\widetilde S_{12}\circ R^t$ from Lecture I implies $S_{12}=(S_{\sigma,12}^\vartheta)_{\sigma,\vartheta=1,\ldots,n}\in C^{\alpha}(\overline B)$ (note $\widetilde S_{12}\in C^\alpha(\overline B)$ for $k=3$; in case $k\ge 4$ we even get $S_{12}\in C^{1,\alpha}(\overline B)$). Now the right-hand side of the equation for $\Delta\tau_\sigma^\vartheta$ belongs to $C^\alpha(\overline B),$ and potential theoretic estimates ensure $\tau\in C^{2,\alpha}(\overline B)$. Involving again our first-order system for the $\tau_\sigma^\vartheta$ gives $T_i\in C^{1,\alpha}(\overline B),$ which proves $N\in C^{2,\alpha}(\overline B)$ using the Weingarten equations once more. Finally, for $k\ge4$, we can bootstrap by concluding $R\in C^{2,\alpha}(\overline B)$ and $S_{12}\in C^{1,\alpha}(\overline B)$ from the transformation rule for $S_{121}$ and repeating the arguments above.
\end{itemize}
\vspace*{1ex}
The proof is complete.
\end{proof}
\vspace*{6ex}
\begin{center}
\rule{35ex}{0.1ex}
\end{center}
\vspace*{6ex}
\section{Estimates for the area element $W$}
\label{section_Westimates}
Various of our estimates so far contain norms of the quantity ${\mathcal S}={\mathfrak S}_NW.$ In particular, for given ${\mathfrak S}_N$ we are left to establish upper bounds on the area element $W.$ For this purpose we want to sketch briefly a classical method which goes back essentially to Erhard Heinz.\\[1ex]
Throughout our considerations we are concerned with solutions of nonlinear elliptic problems of the form
  $$\begin{array}{l}
      X(u,v)=\big(x^1(u,v),\ldots,x^{n+2}(u,v)\big)\in C^2(B,\mathbb R^{n+2})\cap C^1(\overline B,\mathbb R^{n+2}), \\[2ex]
      |\triangle X(u,v)|\le a|\nabla X(u,v)|^2\quad\mbox{in}\ B, \\[2ex]
      |X(u,v)|\le M\quad\mbox{in}\ \overline B.
    \end{array}$$
Let us abbreviate this system with (ES). Furthermore, we suppose that $X$ is conformally parametrized
  $$X_u\cdot X_u=W=X_v\cdot X_v\,,\quad
    X_u\cdot X_v=0
    \quad\mbox{in}\ \overline B$$
with the area element $W$ as conformal factor. Therefore, we infer
  $$W=\frac{1}{2}\,(X_u^2+X_v^2)=\frac{1}{2}\,|\nabla X|^2$$
meaning that an upper bound for $|\nabla X|^2$ is immediately an upper bound for the area element $W.$ The first results in this direction go back to Heinz \cite{heinz_1957} and \cite{heinz_1970}. We also want to refer to Schulz \cite{schulz_1991}, Dierkes et al. \cite{dhkw_1992}, volume II, and Sauvigny \cite{sauvigny_2005}, volume II for further developments.
\subsection{Inner estimates}
\label{par_inner}
We start with inner gradient bounds without referring to any boundary data.
\begin{proposition}
(Sauvigny \cite{sauvigny_2005}, volume II, \S 2, Satz 1)\\
Let $X$ be a solution of (ES) with $aM<1.$ Define the distance function
  $$\delta(w)=\mbox{\rm dist}\,(w,\partial B),\quad w\in\overline B.$$
Then there is a real constant $C(a,M)\in[0,\infty)$ satisfying
  $$\delta(w)|\nabla X(w)|\le C(a,M)\quad\mbox{for all}\ w\in B.$$
\end{proposition}
\noindent
In particular, we infer
  $$|\nabla X(0,0)|\le C(a,M).$$
Inequalities of this kind were needed for the curvature estimate for minimal graphs in Lecture I.
\subsection{Global estimates}
\label{par_global}
Next, we want to quote the following global result.
\begin{proposition}
(Sauvigny \cite{sauvigny_2005}, volume II, chapter 12, \S 3, Satz 1)\\
Let $X$ be a solution of (ES) with $aM<1$ satisfying
  $$X(u,v)=0\quad\mbox{on}\ \partial B.$$
Furthermore, let $\alpha\in(0,1).$ Then there is a constant $C(a,M,\alpha)\in[0,\infty)$ so that
  $$\|X\|_{C^{1+\alpha}(\overline B)}\le C(a,M,\alpha).$$
\end{proposition}
\noindent
The point is that we must require vanishing boundary data, at least on small portions of $\partial B$ where we would like to apply Sauvigny's gradient estimate.
\subsection{Straightening the boundary}
\label{par_straight}
Finally we follow an idea of J.C.C. Nitsche, see also Heinz \cite{heinz_1970}, to {\it straighten the boundary locally.} Here are some technical assumptions: First, let
  $$X(w_0)=0\in\mathbb R^{n+2}\quad\mbox{for}\ w_0=(1,0).$$
Moreover, we may represent the arc $X(\{w\in\partial B\,:\,|w-w_0|<\varepsilon\})$ in the form
  $$x^k=g_k(x^{n+2}),\quad k=1,\ldots,n+1,$$
with $\varepsilon>0$ being sufficiently small. Next, w.l.o.g. we assume
  $$g_k(0)=0,\quad
    g_k'(0)=0
    \quad\mbox{for all}\ k=1,\ldots,n+1$$
as well as
  $$|x^\ell(w)|\le\delta\quad\mbox{for all}\ \ell=1,\ldots,n+2,\quad
    \sum_{\ell=1}^{n+1}|g_k'(t)|\le c_1\,,\quad
    \sum_{\ell=1}^{n+1}|g_k''(t)|\le c_2$$
with $c_1\in(0,1).$ To realize these assumptions it is sufficient to require bounds on the $C^2$-Schauder norm of $X$ on $\overline B$ which, in particular, implies a modulus of continuity of $X.$ Consider now the functions
  $$y^k(w)=x^k(w)-g_k(x^{n+2})
    \quad\mbox{on}\ D_\varepsilon:=\{w\in\overline B\,:\,|w-w_0|<\varepsilon\},\quad
    k=1\ldots,n+1.$$
Note that there hold
  $$y^k(w)=0\quad\mbox{on}\ D_\varepsilon\quad\mbox{for all}\ k=1,\ldots,n+1,$$
and these are already the homogenous boundary conditions in Sauvigny's global gradient estimate! We must check that $Y=(y^1,\ldots,y^{n+1})$ solves an elliptic system of the form (ES) with quadratic growth in the gradient. We compute
  $$\begin{array}{lll}
      |\triangle y^k|\negthickspace
      &  =  & \negthickspace\displaystyle
              \big|\triangle x^k-\triangle g_k(x^{n+2})\big|
              \,=\,\big|\triangle x^k-g_k''(x^{n+2})|\nabla x^{n+2}|^2-g_k'(x^{n+2})\triangle x^{n+2}\big| \\[2ex]
      & \le & \negthickspace\displaystyle
              |\triangle x^k|
              +|g_k'(x^{n+2})||\triangle x^{n+2}|
              +|g_k''(x^{n+2})|\nabla x^{n+2}|^2 \\[2ex]
      & \le & \negthickspace\displaystyle
              \big(1+|g_k'(x^{n+2})|\big)|\triangle X|
              +|g_k''(x^{n+2})||\nabla X|^2
              \,\le\,\big\{(1+c_1)a+c_2\big\}\,|\nabla X|^2\,.
    \end{array}$$
Now from the conformality relations we infer
  $$|x_w^{n+2}|^2
    \le\sum_{k=1}^{n+1}|x_w^k|^2$$
for the complex-valued derivatives $x_w^k=\frac{1}{2}\,(x_u-ix_v),$ and considering
  $$\sum_{k=1}^{n+1}|x_w^k|
    \le\sum_{k=1}^{n+1}|y_w^k|+c_1\sum_{k=1}^{n+1}|x_w^k|
    \quad\mbox{implies}\quad
    \sum_{k=1}^{n+1}|x_w^k|\le\frac{1}{1-c_1}\,\sum_{k=1}^{n+1}|y_w^k|.$$
Thus we infer
  $$\sum_{k=1}^{n+2}|x_w^k|^2
    \le c_3(c_1,n)\sum_{k=1}^{n+1}|y_w^k|^2$$
which gives us finally
  $$|\triangle Y|\le c_4(a,c_1,c_2,n)|\nabla Y|^2\,.$$
{\it In conclusion, under the knowledge of $X$ and its first and second derivatives on the boundary curve $\partial B$ we may apply Sauvigny's global gradient estimate which yields in turn estimates fpr the area element $W$ from above on the whole disc $\overline B.$}
\vspace*{10ex}
\begin{center}
\rule{35ex}{0.1ex}
\end{center}
\cleardoublepage
\printindex


{\small
\begin{thebibliography}{99}
\bibitem{alt_2006}
Alt, W.: {\it Lineare Funktionalanalysis.} Springer, 2006.
\bibitem{asperti_1983}
Asperti, A.C.: {\it Minimal surfaces with constant normal curvature.} J. Math. Soc. Japan {\bf 36,} No. 3, 1984.
\bibitem{baer_2009}
B\"ar, C.: {\it Elementare Differentialgeometrie.} Gruyter, 2009.
\bibitem{bergner_froehlich_2008}
Bergner, M.; Fr\"ohlich, S.: {\it  On two-dimensional immersions of prescribed mean curvature in $\mathbb R^n.$} Z. Analysis Anw. {\bf 27,} No. 1, 31--52, 2008.
\bibitem{blaschke_leichtweiss_1973}
Blaschke, W.; Leichtweiss, K.: {\it Elementare Differentialgeometrie.} Springer, 1973.
\bibitem{bourgain_brezis_mironescu_2000}
Bourgain, J.; Brezis, H; Mironescu, P.: {\it Lifting in Sobolev spaces.} J. Anal. Math. {\bf 80}, 37--86, 2000.
\bibitem{brauner_1981}
Brauner, H.: {\it Differentialgeometrie.} vieweg, 1981.
\bibitem{brezis_coron_1984}
Brezis, H.; Coron, J.~M.: {\it Multiple solutions of $H$-systems and Rellich's conjecture.} Comm. Pure Appl. Math. {\bf 37}, 149--187, 1984.
\bibitem{burchard_thomas_2002}
Burchard, A.; Thomas, L.E.: {\it On the Cauchy problem for a dynamical Euler's elastica.} Comm. Part. Diff. Equations {\bf 28,} 271--300, 2003.
\bibitem{cartan_1974}
Cartan, H.: {\it Differentialformen.} BI Wissenschaftsverlag, 1974.
\bibitem{chavel_2006}
Chavel, I.: {\it Riemannian Geometry.} Cambridge University Press, 2006.
\bibitem{chen_1973}
Chen, B.-Y.: {\it Geometry of Submanifolds.} Marcel Dekker, Inc., 1973.
\bibitem{chen_ludden_1972}
Chen, B.Y., Ludden, G.D.: {\it Surfaces with mean curvature vector parallel in the normal bundle.} Nagoya Math. J. {\bf 47,} 161--167, 1972.
\bibitem{chern_1955}
Chern, S.S.: {\it An elementary proof of the existence of isothermal parameters on a surface.} Proc. American Math. Soc. {\bf 6,} No. 5, 771--782, 1955.
\bibitem{coifman_lions_meyer_semmes_1993}
Coifman, R.; Lions, P.-L.; Meyer, Y.; Semmes, S.: {\it Compensated compactness and Hardy spaces.} J. Math. Pure Appl. {\bf 9,} No. 3, 247--286, 1993.
\bibitem{colding_minicozzi_1999}
Colding, T.H.; Minicozzi, W.P.: {\it Minimal surfaces.} Courant Institute of Mathematical Sciences, 1999.
\bibitem{courant_1950}
Courant, R.: {\it Dirichlet's principle, conformal mappings, and minimal surfaces.} Interscience publishing, Inc., New York, 1950.
\bibitem{courant_hilbert_1962}
Courant, R.; Hilbert, D.: {\it Methods of mathematical physics 2.} John Wiley \& Sons, Inc., 1962.
\bibitem{dacosta_1982}
da Costa, R.C.T.: {\it Constraints in quantum mechanics.} Phys. Rev. A {\bf 25,} No. 6, 2893--2900, 1982.
\bibitem{dierkes_2005}
Dierkes, U.: {\it Maximum principles for submanifolds of arbitrary codimension and bounded mean curvature.} Calc. Var. {\bf 22,} 173-184, 2005.
\bibitem{dhkw_1992}
Dierkes, U.; Hildebrandt, S.; K\"uster, A.; Wohlrab, O.: {\it Minimal surfaces I, II.} Grundlehren der mathematischen Wissenschaften {\bf 295/296,} Springer, 1992.
\bibitem{docarmo_1992}
do Carmo, M.P.: {\it Riemannian geometry.} Birkh\"auser, 1992.
\bibitem{dobrowolski_2006}
Dobrowolski, M.: {\it Angewandte Funktionalanalysis.} Springer, 2006.
\bibitem{ecker_2004}
Ecker, K.: {\it Regularity theory for mean curvature flow.} Birkh\"auser, 2004.
\bibitem{ecker_huisken_1989}
Ecker, K.; Huisken, G.: {\it Interior curvature estimates for hypersurfaces of prescribed mean curvature.} Ann. Inst. H. Poincar\'e Anal. Non Lin\'eaire {\bf 6,} 251--260, 1989.
\bibitem{eschenburg_jost_2007}
Eschenburg, J.-H.; Jost, J.: {\it Differentialgeometrie und Minimalfl\"achen.} Springer, 2007.
\bibitem{fefferman_stein_1972}
Fefferman, C.; Stein, E.-M.: {\it $H^p$ spaces of several variables.} Acta Math. {\bf 129}, 137--193, 1972.
\bibitem{froehlich_2005}
Fr\"ohlich, S.: {\it On $2$-surfaces in $\mathbb R^4$ and $\mathbb R^n.$} Proceedings of the 5th Conference of Balkan Society of Geometers, Mangalia, 2005.
\bibitem{froehlich_mueller_2007}
Fr\"ohlich, S.; M\"uller, F.: {\it On critical normal sections for two-dimensional immersions in $\mathbb R^4$ and a Riemann-Hilbert problem.} to appear in Differential Geometry and Applications. Differ.~Geom.~Appl.~{\bf 26}, No. 5, 508--513, 2008.
\bibitem{froehlich_mueller_2008}
Fr\"ohlich, S.; M\"uller, F.: {\it On critical normal sections for two-dimensional immersions in $\mathbb R^{n+2}.$} Calc. Var. {\bf 35}, 497--515, 2009.
\bibitem{froehlich_winklmann_2007}
Fr\"ohlich, S.; Winklmann, S.: {\it Curvature estimates for graphs with prescribed mean curvature and flat normal bundle.} manuscripta mathematica {\bf 122,} No. 2, 2007.
\bibitem{gilbarg_trudinger_1983}
Gilbarg, D.; Trudinger, N.S.: {\it Elliptic Partial Differential Equations of Second Order}. Springer, Berlin Heidelberg New\,York, 1983.
\bibitem{hartman_wintner_1953}
Hartman, P.; Wintner, A.: {\it On the local behaviour of solutions of nonparabolic partial differential equations.} Amer. J. Math. {\bf 75,} 449--476, 1953.
\bibitem{heil_1974}
{\sc Heil, E.:} {\it Differentialformen und Anwendungen auf Vektoranalysis, Differentialgleichungen, Geo\-metrie.} Bibliogr. Inst., 1974.
\bibitem{heinz_1954}
Heinz, E.: {\it \"Uber die Existenz einer Fl\"ache konstanter mittlerer Kr\"ummung bei vorgegebener Berandung.} Math. Ann. {\bf 127}, 258--287, 1954.
\bibitem{heinz_1957}
Heinz, E.: {\it On certain nonlinear elliptic differential equations and univalent mappings.} J.~Anal. Math. {\bf 5}, 197--272, 1957.
\vspace*{-1.2ex}
\bibitem{heinz_1970}
Heinz, E.: {\it \"Uber das Randverhalten quasilinearer elliptischer Systeme mit isothermen Parametern.} Math. Z. {\bf 113}, 99--105, 1970.
\bibitem{helein_2002}
Helein, F.: {\it Harmonic maps, conservation laws and moving frames.} Cambridge University Press, 2002.
\bibitem{hildebrandt_1970}
Hildebrandt, S.: {\it Einige Bemerkungen \"uber Fl\"achen beschr\"ankter mittlerer Kr\"ummung.} Math. Z. {\bf 115,} 169--178, 1970.
%\bibitem{hoffman_1972}
%Hoffman, D.: {\it Surfaces in constant mean curvature manifolds with parallel mean curvature vector field.} Bull. Am. Math. Soc. {\bf 78,} No. 2, 1972.
\bibitem{hopf_1950}
Hopf, H.: {\it \"Uber Fl\"achen mit einer Relation zwischen den Hauptkr\"ummungen.} Mathematische Nachrichten {\bf 4,} 232--249, 1950/51.
\bibitem{klingenberg_1973}
Klingenberg, W.: {\it Eine Vorlesung \"uber Differentialgeometrie.} Springer, 1973.
\bibitem{mueller_schikorra_2008}
M\"uller, F.; Schikorra, A.: {\it Boundary regularity via Uhlenbeck-Riviere decomposition.} To appear in Analysis.
\bibitem{nitsche_1975}
Nitsche, J.C.C.: {\it Vorlesungen \"uber Minimalfl\"achen.} Grundlehren der mathematischen Wissenschaften {\bf 199,} Springer, 1975.
\bibitem{osserman_1986}
Osserman, R.: {\it A survey of minimal surfaces.} Dover Publications, Inc., 1986.
\bibitem{riviere_2007}
Rivi\`ere, T.: {\it Conservation laws for conformally invariant variational problems}, Invent. Math. {\bf 168}, 1--22, 2007.
\bibitem{sauvigny_2005}
Sauvigny, F.: {\it Partielle Differentialgleichungen der Geometrie und Physik.} 2 volumes, Springer, 2005.
\bibitem{ssy_1975}
Schoen, R.; Simon, L.; Yau, S.T.: {\it Curvature estimates for minimal hypersurfaces.} Acta Math. {\bf 134,} 275--288, 1975.
\bibitem{schulz_1991}
Schulz, F.: {\it Regularity Theory for Quasilinear elliptic systems and Monge-Ampere equations in two dimensions.} Springer, 1991.
\bibitem{stein_1993}
Stein, E.M.: {\it Harmonic Analysis: Real-variable methods, orthogonality, and oscillatory integrals.} Princeton University Press, 1993.
\bibitem{svec_1976}
Svec, A.S.: {\it On the existence of parallel normal vector fields for surfaces in $E^4.$} Czech. Math. Journal {\bf 26 (101),} 297--303, 1976.
\bibitem{takahashi_2003}
Takahashi, F.: {\it Multiple solutions of inhomogeneous H-systems with zero Dirichlet boundary conditions.} Nonlinear Anal. {\bf 52,} No. 1, 239--259, 2003.
\bibitem{tomi}
Tomi, F.: {\it Ein einfacher Beweis eines Regularit\"atssatzes f\"ur schwache L\"osungen gewisser elliptischer Systeme}. Math. Z. {\bf 112}, 214--218, 1969.
\bibitem{topping_1997}
Topping, P.: {\it The optimal constant in Wente's $L^\infty$-estimate}. Comment.~Math.~Helv. {\bf 72}, 316--328, 1997.
\bibitem{vekua_1963}
Vekua, I.N.: {\it Verallgemeinerte analytische Funktionen.} Mathematische Lehrb\"ucher und Monographien, Akademie-Verlag, 1963.
\bibitem{wang_2002}
Wang, M.-T.: {\it On graphical Bernstein type results in higher codimension.} Trans. Amer. Math. Soc. {\bf 355,} 265--271, 2003.
\bibitem{wang_2004}
Wang, M.-T.: {\it Stability and curvature estimates for minimal graphs with flat normal bundle.} arXiv:math.DG/ 0411169v2, 2004.
\bibitem{wente_1975}
Wente, H.\,C.: {\it The differential equation $\Delta x=2Hx_u\wedge x_v$ with vanishing boundary values.} Proc. Amer. Math. Soc. {\bf 50}, 131-137, 1975.
\bibitem{wente_1980}
Wente, H.C.: {\it Large solutions to the volume constrained Plateau problem.} Arch.~Rat. Mech. Anal.~{\bf 75}, 59--77, 1980.
\bibitem{weyl_1922}
Weyl, H.: {\it Zur Infinitesimalgeometrie: $p$-dimensionale Fl\"ache im $n$-dimensionalen Raum.} Math. Z. {\bf 12,} 154--160, 1922.
\bibitem{xin_2009}
Xin, Y.L.: {\it Curvature estimates for submanifolds with prescribed Gauss image and mean curvature.} Preprint, 2009.
\end{thebibliography}
}
\end{document}